\documentclass[12pt, a4paper]{amsart}
\usepackage[margin=1.2in]{geometry}
\input xy
\xyoption{all}
\usepackage{amsmath}
\usepackage{amsthm}
\usepackage{stmaryrd}
\usepackage{color, soul}
\usepackage{amsfonts}
\usepackage{hhline}
\usepackage{amssymb}
\usepackage{amscd}
\numberwithin{equation}{section}
\usepackage{amsmath}

\numberwithin{equation}{section}
\newtheorem{theorem}{Theorem}[section]
\newtheorem{cor}[theorem]{Corollary}
\newtheorem{lemma}[theorem]{Lemma}
\newtheorem{prop}[theorem]{Proposition}
\newtheorem*{thkw}{Theorem KWG}

\newtheorem{defi}{Definition}[section]
\newtheorem{nota}{Notation}[section]
\theoremstyle{definition}

\theoremstyle{remark}
\newtheorem{rem}{Remark}[section]

\renewcommand\H{{\mathbb H}}
\newcommand\ba{{\mathbf a}}

\newcommand\half{\frac{1}{2}}

\newcommand\ov{\overline}

\newcommand\be{\beta}

\newcommand\g{\mathfrak g}

\newcommand\h{\mathfrak h}
\newcommand\q{\mathfrak q}

\newcommand\n{\mathfrak n}

\newcommand\D{\Delta}
\renewcommand\l{\lambda}
\newcommand\Dp{\Delta^+}

\renewcommand\d{\delta}
\renewcommand\r{\mathfrak r}
\renewcommand\t{\otimes}
\renewcommand\a{\alpha}

\renewcommand\k{\mathfrak k}

\newcommand\lb{\llbracket }
\newcommand\rb{\rrbracket }

\newcommand\nd{\noindent}
\newcommand\nat{\mathbb N}
\newcommand\ganz{\mathbb Z}

\newcommand\s{\sigma}
\renewcommand\L{\Lambda}
\newcommand\la{\lambda}

\newcommand\e{\epsilon}
\newcommand\C{\mathbb C}
\newcommand\R{\mathbb R}
\newcommand\si{\sigma}

\newcommand{\fg}{\mathfrak{g}}
\newcommand{\fh}{\mathfrak{h}}

\newcommand\p{\mathfrak p}

\newcommand{\ve}{\varepsilon}

\renewcommand{\t}{\mathfrak t}
\newcommand{\tr}{{\mathfrak t}_{\mathbb R}}
\newcommand{\PP}{\mathbb{P}}

\newcommand{\sdim}{\text{\rm sdim}}
\newcommand{\cF}{\mathcal{F}}
\newcommand{\cE}{\mathcal E}
\newcommand{\cD}{\mathcal D}
\newcommand{\cP}{\mathcal P}
\newcommand{\cS}{\mathcal S}

\newcommand{\cQ}{\mathcal Q}
\newcommand{\sn}{\operatorname{sn}}
\newcommand{\htt}{\operatorname{ht}}

\newcommand{\Lem}[1]{Lemma~\ref{#1}}
\newcommand{\Cor}[1]{Corollary~\ref{#1}}

\newenvironment{Dynkin}{\setlength{\unitlength}{1.5pt}\begin{array}{l}}{\end{array}}

\newcommand{\Dbloc}[1]{\begin{picture}(20,20)#1\end{picture}}
\newcommand{\Dcirc}{\put(10,10){\circle{4}}}
\newcommand{\Dcross}{\put(8,8){$\otimes$}}

\newcommand{\Deast}{\put(12,10){\line(1,0){8}}}

\newcommand{\Dwest}{\put(8,10){\line(-1,0){8}}}

\newcommand{\Dnortheast}{\put(20,20){\line(-1,-1){8.6}}}

\newcommand{\Dnorthwest}{\put(0,20){\line(1,-1){8.6}}}

\newcommand{\Dsoutheast}{\put(20,0){\line(-1,1){8.6}}}
\newcommand{\Dsouthwest}{\put(0,0){\line(1,1){8.6}}}

\newcommand{\Ddoubleeast}{\put(10,12){\line(1,0){10}}\put(10,8){\line(1,0){10}}}

\newcommand{\Ddoublewest}{\put(10,12){\line(-1,0){10}}\put(10,8){\line(-1,0){10}}}
\newcommand{\Dtriplewest}{\put(10,12){\line(-1,0){10}}\put(10,8){\line(-1,0){10}}\put(10,10){\line(-1,0){10}}}
\newcommand{\Dtripleeast}{\put(10,12){\line(1,0){10}}\put(10,8){\line(1,0){10}} \put(10,10){\line(1,0){10}} }

\newcommand{\Dskip}{\\ [-4.5pt]}

\newcommand{\Dleftarrow}{\hskip-5pt{\makebox(20,20)[l]{\Large$<$}}\hskip-25pt}
\newcommand{\Drightarrow}{\hskip-5pt{\makebox(20,20)[l]{\Large$>$}}\hskip-25pt}

\begin{document}
\title[Denominator identities for   Lie superalgebras]{Denominator identities for finite-dimensional Lie superalgebras and Howe duality for compact dual pairs}
\author[Gorelik, Kac, M\"oseneder, Papi]{Maria Gorelik}
\author[]{Victor~G. Kac}
\thanks{The first author is supported by the Minerva foundation with funding from the Federal German Ministry for Education and Research. The second author is partially supported by an NSF grant and by an ERC advanced grant.}
\author[]{Pierluigi M\"oseneder Frajria}
\author[]{Paolo  Papi}
\begin{abstract} We provide  formulas for the denominator and superdenominator  of a basic classical  type Lie superalgebra for  any set of positive roots. We establish a connection between certain sets of positive roots and the theory of reductive dual pairs of real Lie groups, and , as an application of these formulas, we recover the Theta correspondence for  compact dual pairs. Along the way we give an explicit description of the real forms of basic classical type Lie superalgebras.
\end{abstract}
\maketitle
\tableofcontents
\section{Introduction}
The Weyl denominator identity
\begin{equation}\label{wdi}\prod_{\a\in\Dp}(1-e^{-\a})=\sum_{w\in W}sgn(w)\,e^{w(\rho)-\rho}\end{equation}
is one of the most intriguing combinatorial identities in the character ring of a complex finite-dimensional simple Lie algebra. It admits far reaching generalizations to the Kac-Moody
setting, where it provides a proof for the Macdonald's identities (see \cite{Kac}). Its role in representation theory is well-understood, 
since the inverse of the l.h.s. of \eqref{wdi} is the character of the Verma module $M(0)$ with the highest weight $0$.\par
In the first part of this paper we provide a generalization of formula 
 \eqref{wdi} to the setting of basic classical type Lie superalgebras.
 Deepening this problem, we came across an  interesting connection with representation theory of Lie groups which is the theme of the second part of the paper.
 \par
 By a basic classical type Lie superalgebra we mean an almost simple finite-dimensional Lie superalgebra $\g=\g_0\oplus\g_1$ with a non-degenerate invariant supersymmetric bilinear form $(\cdot,\cdot)$ and $\g_0$ reductive. Choosing a Cartan subalgebra $\h$ of $\g_0$, we get the set  of roots $\D=\D_0\cup \D_1$, where $\D_i$ is the set  of roots of $\h$ in $\g_i,\,i=0,1$. Choosing a set of positive roots  $\Dp$ in $\D$,  we let $\Dp_i=\Dp\cap \D_i$.
 In trying to extend \eqref{wdi} to a Lie superalgebra $\g$, it is natural to 
 replace the l.h.s of \eqref{wdi} with the character of the Verma module  $M(0)$ over $\g$, which is the inverse of 
\begin{equation}\label{kwd}R=\frac{\prod_{\a\in\Dp_0}(1-e^{-\a})}{\prod_{\a\in\Dp_1}(1+e^{-\a})}, \end{equation}
called the {\it denominator}.
 Beyond the denominator $R$,  very important for us will be the 
 {\it superdenominator}, defined as
\begin{equation}\label{kwdd}
\check R=\frac{\prod_{\a\in\Dp_0}(1-e^{-\a})}{\prod_{\a\in\Dp_1}(1-e^{-\a})}.\end{equation}
 \vskip5pt  Generalizations of formulas for $R$ and $\check R$ to affine superalgebras and their connection with number theory and the theory of special functions are thoroughly discussed in \cite{KW}.
 The striking differences which make the super case very different from the purely even one are the following. First, it is no more true that subsets  of positive roots are conjugate under the Weyl group: to get transitivity on the sets of positive roots one has to consider Serganova's odd reflections, see \eqref{oddref}.
Moreover the restriction of the  nondegenerate invariant bilinear form $(\cdot\, ,\cdot)$ to the real span 
$V_\D$ of roots may be indefinite,  hence isotropic sets of roots appear naturally. One defines the defect $d=\text{def $\g$}$ as the dimension of a maximal isotropic subspace of $V_\D$. A subset of $\D$, consisting of linearly independent pairwise orthogonal isotropic roots is called {\it isotropic}. It is known that any maximal  isotropic subset of $\D$ consists of $d$  roots (\cite{KW}).\par
In this paper we settle completely the problem of finding an analogue of \eqref{wdi} for basic classical type Lie superalgebras, by providing an expression for the r.h.s. which incorporates 
the dependence on the set of positive roots.
The following result is known.
\begin{thkw}\label{kwt} Let $\g$ be a basic classical  type Lie superalgebra and let $\Dp$ be a set of positive roots such that there exists a maximal isotropic subset $S$ of $\Dp$,  contained in the set of simple roots  corresponding to $\Dp$.
Then 
\begin{align}\label{kw2}
e^\rho R &=\sum_{w\in W^\sharp}sgn(w)
 w\left( \frac{e^\rho}{\prod_{\be\in S}(1+e^{-\be})}\right),\\\label{kw1}
e^\rho\check R &=\sum_{w\in W^\sharp}sgn'(w)
 w\left( \frac{e^\rho}{\prod_{\be\in S}(1-e^{-\be})}\right).\end{align}
\end{thkw}
\noindent In the above statement $\rho=\rho_0-\rho_1,$ where $\rho_i=\frac{1}{2}\sum_{\a\in\Dp_i}\a$, $W^\sharp$ is a subgroup of the Weyl group $W_\g$ of $\g$  defined in \eqref{wsharp} and $sgn'$ is defined in \eqref{sgn'}. This result has been stated by Kac and Wakimoto in \cite{KW}, and fully proved  for $d=1$ by using representation theoretical methods. A  combinatorial proof for arbitrary $d$  has been recently  obtained by Gorelik \cite{Gor}. \par
Theorem KWG will be of fundamental importance for our generalization, which is as follows: given any subset of positive roots $\Dp$, we want to express $e^{\rho}R$ by a formula like the r.h.s. of \eqref{kw2}, in which at the numerator appears the $\rho$ corresponding to $\Dp$ 
and at the denominator a suitable maximal isotropic subset $S$ of $\Dp$, so that   the exponents  at the denominator  are  
  linear combinations with non-positive integer coefficients of simple roots: see \eqref{princf}, \eqref{princff}. We also provide a formula in which the denominator is exactly as in the r.h.s. of
  of \eqref{kw2}, but we have to perform a correction on $\rho$, depending on $S$: see \eqref{mermaid}.
  \vskip5pt
 Let us discuss more in detail these formulas.
First we  construct a certain class $\cS$ of maximal isotropic subsets  by the following procedure. 
\begin{defi}\label{cp} Denote by $\cS(\Dp)$ the collection of maximal isotropic subsets of $\Dp$ of the form $S=S_1\cup\ldots\cup S_t$ of $\Dp$, where $S_1$  is  a non empty  isotropic subset of the set  of simple roots, and, inductively, $S_i$ is  such a subset in the set of indecomposable roots  of $S_{i-1}^\perp\setminus S_{i-1}$, where 
$S_{i-1}^\perp=\{\a\in\Dp\mid (\a,\beta)=0\, \forall\,\beta\in S_{i-1}\}$.
\end{defi}
Denote by $Q,\,Q_0$ the lattices  spanned over $\ganz$ by all  roots and even roots, respectively, and let, for $S\in\cS(\Dp)$ and $\gamma\in  S$

\begin{align}\label{<} 
&\varepsilon(\eta)=\begin{cases}
1\quad&\text{if $\eta\in Q_0$,}\\
-1\quad&\text{if $\eta\in Q\setminus Q_0$,}\end{cases}\\
&\gamma^\leq=\{\beta\in S\mid\beta\leq \gamma\},
\ 
\gamma^<=\{\beta\in S\mid\beta< \gamma\},\\\label{parentesi}
&\llbracket \gamma\rrbracket =\sum_{\beta\in \gamma^{\leq}}\varepsilon(\gamma-\beta)\beta,\quad \rrbracket \gamma\llbracket =\sum_{\beta\in \gamma^<}\ \varepsilon(\gamma-\beta)\beta,\\
& sgn(\gamma)=(-1)^{|\gamma^{\leq}|+1},\end{align}
where, as usual, $\beta\leq \gamma$ if $\gamma-\beta$ is a sum of positive roots or zero and
$|X|$ stands for the cardinality of the set $X$.

\begin{theorem}\label{princ} Let $\g=\g_0\oplus\g_1$ be a basic classical type Lie superalgebra having  defect $d$, where $\g=A(d-1,d-1)$ is replaced by  $gl(d,d)$. Let $\Dp$ be  any set of positive roots. For any $S\in\mathcal S(\Dp)$  we have
 \begin{align}\label{princf} C\cdot e^\rho R&=
\sum_{w\in W_\g}sgn(w)
w\left(\frac{e^\rho}
{
\prod_{\gamma\in S}(1+sgn(\gamma)e^{-\lb\gamma\rb})}\right),\\\label{princff}
C\cdot e^\rho \check R&=
\sum_{w\in W_\g}sgn'(w)
w\left(\frac{e^\rho}
{
\prod_{\gamma\in S}(1-e^{-\llbracket \gamma\rrbracket })}\right),
\nd
\end{align}
where 
\begin{equation}\label{costante}
C=\frac{C_\g}{\prod\limits_{\gamma\in S} \frac{ht(\gamma)+1}{2}},
\end{equation}
and $C_\g=|W_\g/W^\sharp|$.
Moreover, there exists $S^{nice}\in \mathcal S(\Dp)$ such that   $\lb\gamma\rb\in Q^+=\ganz_+\Dp$ for any $\gamma\in S^{nice}$.\end{theorem}
The explicit values of $C_\g$ are the following.
\begin{equation}\label{valoricg}
C_\g=\begin{cases}d!\quad&\text{if $\g=A(m-1,n-1)$},\\
2^dd!\quad&\text{if $\g=B(m,n)$},\\
2^dd!\quad&\text{if $\g=D(m,n),\, m> n$},\\
2^{d-1}d!\quad&\text{if $\g=D(m,n),\, n\ge m$},\\
1\quad&\text{if $\g=C(n)$},\\
2\quad&\text{if $\g=D(2,1,\a),\,F(4),\,G(3)$.}
\end{cases}
\end{equation}
If $\text{def}(\g)=1$, and $S\in \cS(\Dp)$, then $S$ consists of a single simple root. Hence we are in the hypothesis of Theorem KWG, in which case \eqref{princf},  \eqref{princff} hold with $C=C_\g$, and these formulas coincide with \eqref{kw2}, \eqref{kw1}.
 Therefore we have to deal only with Lie superalgebras of type $gl(m,n),\,B(m,n),$ $D(m,n)$,  which  have defect $d=\min\{m,n\}.$  We   also treat the case $A(d-1,d-1)$, see
 Subsection \ref{comments}.\par

In these cases we introduce a combinatorial encoding of the elements of $\cS(\Dp)$ using the notion of  an arc diagram (see Definition \ref{ad}).   To any arc diagram $X$ we associate a maximal isotropic set $S(X)$, and in Proposition \ref{cad} we show that this is a bijection between all arc diagrams and $\cS(\Dp)$.

We introduce two types of operations on arc diagrams: odd reflections $r_v$ and interval reflections $r_{[v,w]}$. They  have the following features:
\begin{enumerate}
\item for any  arc diagram $X$ there exists a finite sequence of odd and interval reflections which change $X$ into an arc diagram $X'$ such that $S(X')$ consists of simple roots
(cf. Lemma \ref{esad});
\item we are able to relate $\sum_{w\in W}sgn(w) w\left(\frac{e^{\rho}}{\prod_{\gamma\in S(X)}(1-e^{-\llbracket \gamma\rrbracket})}\right)$ to the similar
sums where the product in the denominator of the l.h.s ranges over $r_v X$ and over $r_{[v,w]}X$ (cf. \eqref{intervalreflection}).
\end{enumerate}
The above properties allow to prove formula   \eqref{princff} starting from Theorem KWG applied to $S(X')$. In Lemma \ref{ded} we show that formula \eqref{princf} can be derived from \eqref{princff}. We also single out a special set of arc diagrams (see Definition \ref{sad}) which have the property described in the last sentence of  Theorem \ref{princ}.
\par
Further, we  provide a different generalization of \eqref{kw1}, in which only positive roots appear in the denominator, namely
we have
\begin{prop}\label{mm} Let $S\in\cS(\Dp)$. Then
\begin{align}\label{exmm}
e^{\rho}R &=\sum_{w\in W^\#}
sgn(w)w \left(\frac{
 (-1)^{\sum_{\gamma\in S} |\gamma^<|}
 e^{\rho+\sum_{\gamma\in S}\rrbracket \gamma\llbracket }
}{\prod_{\beta\in S} (1+e^{-\beta})}\right),
\\\label{mermaid}
 e^{\rho} \check R &=\sum_{w\in W^{\#}}sgn'(w) w \left(\frac{e^{\rho+\sum_{\gamma\in S}\rrbracket \gamma\llbracket }}{\prod_{\beta\in S} (1-e^{-\beta})}\right).
\end{align}
\end{prop}
\vskip5pt 
The main application of our formulas is a conceptually uniform derivation of the Theta correspondence for compact dual pairs.  Basic classical type Lie superalgebras  
are linked to Howe theory of dual pairs through the notion of  a {\it distinguished} set of positive roots, which we introduce in Definition \ref{dis}: it requires that  the depth of the grading, which assigns value $1$ (resp. $0$) to each generator corresponding to a
positive  odd (resp. even) simple root, does not exceed $2$.
The distinguished  sets of positive roots for the basic classical type Lie superalgebras (which are  not  Lie algebras)  are classified in Subsection \ref{classifdis}.\par
To understand the relationship with Howe theory, consider the complex symplectic space $(\g_1,\langle\cdot\,,\cdot\rangle)$, where  $\langle\cdot\,,\cdot\rangle=(\cdot\, ,\cdot)_{|\g_1}$. A choice of a set of 
positive roots $\Dp$ determines a  polarization $\g_1=\g_1^+\oplus\g_1^-$, where 
$\g_1^\pm=\bigoplus\limits_{\a\in\D^\pm_1}\g_{\a}$. Hence we can consider the Weyl algebra
$W(\g_1)$ of $(\g_1,\langle\cdot\,,\cdot\rangle)$ and construct the $W(\g_1)$-module
\begin{equation}\label{met}M^{\Dp}(\g_1)=W(\g_1)/W(\g_1)\g_1^+,\end{equation}
with action by left multiplication.  The module $M^{\Dp}(\g_1)$ is also a $sp(\g_1,\langle\cdot\,,\cdot\rangle)$--module with $T\in sp(\g_1,\langle\cdot\,,\cdot\rangle)$ acting by left multiplication by
\begin{equation}\label{actionsp}
\theta(T)=-\half\sum_{i=1}^{\dim\,\g_1} T(x_i)x^i,
\end{equation}
where $\{x_i\}$ is any basis of $\g_1$ and $\{x^i\}$ is its dual basis w.r.t. $\langle\cdot\,,\cdot\rangle$, i.e. $\langle x_i,x^j\rangle=\d_{ij}$.
It is easy to check that, in $W(\g_1)$, relation
\begin{equation}\label{actionandbracket}
[\theta(T),x]=T(x)
\end{equation}
holds for any $x\in\g_1$. This implies that we have an $\h$-module isomorphism 
\begin{equation}M^{\Dp}(\g_1)\cong S(\g_1^-)\otimes \C_{-\rho_1}\end{equation}
where $S(\g_1^-)$ is  the symmetric  algebra of $\g_1^-$, and $\C_{-\rho_1}$ is the 1-dimensional $\h$-module with highest weight $-\rho_1$. Hence 
the  $\h$-character of $M^{\Dp}(\g_1)$ is given by 
\begin{equation}\label{Weylch}
ch M^{\Dp}(\g_1)=\frac{e^{-\rho_1}}{\prod_{\a\in\Dp_1}(1-e^{-\a})}.
\end{equation}
Since $ad_{|\g_1}(\g_0)\subset sp(\g_1, \langle\cdot\,,\cdot\rangle)$, we obtain an action of $\g_0$ on $M^{\Dp}(\g_1)$.
 Upon multiplication by $e^{\rho_0}\prod_{\a\in\Dp_0}(1-e^{-\a})$ the r.h.s. of \eqref{Weylch} becomes $e^\rho\check R$ and equating it to our formula we obtain the  $\g_0$-character of 
$M^{\Dp}(\g_1)$.\par 
So far we have not used the special features of distinguished sets of positive roots $\Dp$. Restricting to  distinguished sets of positive roots $\Dp$ for Lie superalgebras of type $gl(m,n),\,B(m,n),\,D(m,n)$,   we are able to build up a real form $V$ of $\g_1$ endowed with a standard symplectic basis 
$\{e_\a,f_\a\}_{\a\in\Dp_1}$ such that $$\bigoplus_{\a\in\Dp_1}\C(e_\a\pm\sqrt{-1}f_\a)=\g_1^\pm.$$ 
It turns out that 
$sp(V)\cap ad_{|\g_1}(\g_0)=\mathfrak s_1\times \mathfrak s_2,$
$\mathfrak s_i,\,i=1,2,$ being the Lie algebras of a compact dual pair in $Sp(V)$. As stated in Proposition \ref{B}, distinguished sets of positive roots turn out to correspond in this way to all compact dual pairs.
 Howe theory and our denominator formula allow to recover the explicit computation of the Theta correspondence done by Kashiwara-Vergne \cite{KV} and
  Li-Paul-Tan-Zhu \cite{DD} and related results by Enright \cite{En}.
Before discovering formula \eqref{princf},  we used this argument the other way around, to deduce the denominator identities from the knowledge of the Theta correspondence:
the relevant details are given in  \cite{KMP}.\par
In Section \ref{cinquesei} we show that relaxing the condition on the depth of the grading which defines the distinguished sets of positive roots it is possible to generalize the previous approach to  all but one  noncompact type I dual pairs 
(cf. Proposition \ref{structuredual}).
We plan to investigate in a subsequent paper if  our methods can be pushed further to obtain information on the Theta correspondence in these cases too. 
\par
In conclusion of the paper we use our superdenominator formula to confirm the validity of the
Kac-Wakimoto conjecture \cite{KW} for the natural representations of the classical Lie superalgebras.
We intend to extend this method to a large class of representations in a subsequent publication.
\par
The paper is organized as follows. In Section \ref{sec}  we collect some basic notation and definitions. In Section \ref{df}  we develop the combinatorial machinery needed to prove our superdenominator formulas: we introduce arc diagrams and study the effect of  odd  and interval reflection on arc diagrams.
We finally prove the superdenominator formula for the non-exceptional basic classical type Lie superalgebras and also prove the denominator formula. 
The final  subsection is devoted to the proof of formula \eqref{mermaid}. 
In Section \ref{5} we relate the distinguished sets of positive roots to Howe theory: after a preliminary discussion on Cartan involutions we introduce and analyze the definition of distinguished set of positive roots. Then we relate distinguished sets of positive roots to real forms: as a byproduct of our analysis we obtain a conceptual proof of the classification of  real forms of basic classical type Lie superalgebras
which is in the same spirit as the classification of simple Lie algebra involutions, given in \cite{Kac}. This classification has  been previously obtained  by Kac, Parker and Serganova,
(cf. \cite{Kacsuper}, \cite{parker}, \cite{serganova}).
Finally we provide a concise discussion about certain sets of positive roots related to noncompact dual  pairs.
 In the follwoing  four Sections we deduce from our denominator formula the explicit form of the Theta correspondence for all  compact dual pairs.
It is worthwhile to note that the pair $(O(2m),Sp(2n,\R))$ has additional complications which require to use a
version of the  character formula for disconnected compact groups due to Kostant \cite{Kostant}. In Section \ref{44} we prove the Kac-Wakimoto conjecture for the natural  representation.
\par
The ground field is $\C$ throughout the paper, unless otherwise specified.

\section{Setup}\label{sec} In this section we collect some notation and definitions which will be constantly used throughout the paper. Let $\g$ be a basic classical type  Lie 
superalgebra.  This means that $\g=\g_0\oplus\g_1$ is an almost simple 
finite-dimensional  Lie superalgebra with reductive $\g_0$   and that 
$\g$ admits  a nondegenerate invariant supersymmetric bilinear  form  $(\cdot\, ,\cdot)$. Almost simple means that $[\g,\g]$ modulo its center is simple.             Supersymmetric means that $(\g_0,\g_1)=0$ and the restriction of $(\cdot\, ,\cdot)$ to $\g_0$ (resp. $\g_1$) is symmetric 
(resp. skewsymmetric).
A complete list of the simple ones  consists of four series $A(m,n),\,B(m,n),\,
C(n),\,D(m,n)$ and three exceptional Lie superalgebras
$D(2,1,\a),\,F(4),\,G(3)$, and the non-simple ones are obtained by adding to  a simple one or to $gl(n+1,n+1)$ a central ideal. These Lie superalgebras can be constructed, starting with a Cartan matrix, like simple finite-dimensional Lie algebras, though inequivalent Cartan matrices may correspond to the same algebra (see \cite{Kacsuper}).
\par

 Choose a Cartan subalgebra $\h\subset\g_0$, and let $\D,\D_0,\D_1\subset \h^*$ be the set of
roots, even roots, odd roots, respectively. Let $W_\g\subset GL(\h^*)$ be the group generated by the reflections $s_\a$ w.r.t. {\it even}Ê roots $\a\in\D_0$.
  Choose a set of positive roots $\Dp\subset \D$ and set $\Dp_i=\D_i\cap\Dp,\,i=0,1$. Let $\Pi$ be the  set of simple roots corresponding to the choice of $\Dp$, i.e. the set of indecomposable roots in $\Dp$. 
Set also, as usual, for $i=0,1,\,$
  $\rho_i=\tfrac{1}{2}\sum_{\a\in\Dp_i}\a,\ \rho=\rho_0-\rho_1$. From time to time we will write $\rho(\Dp)$ or $\rho(\Pi)$ to emphasize the dependence of $\rho$ from the choice of the  set of positive roots.\par
Next, recall the notation skipped in the Introduction.   Let $h^\vee$ be the dual Coxeter number of $\g$, i.e. $2h^\vee$ is the eigenvalue of the Casimir operator
  of $(\g,\,(\cdot\, ,\cdot))$ in the adjoint representation. If $h^\vee\ne 0$ (which holds unless $\g$ is of type
  $A(n,n),\,D(n+1,n)$ or $D(2,1,\a)$), we let \cite{KW}
  
\begin{align}\notag
&\D_0^\sharp=\{\a\in\D_0\mid h^\vee (\a,\a)>0\},\\
\label{wsharp}
&W^\sharp= \langle s_\beta\in W_\g \mid\beta\in\D_0^\sharp\rangle,
\end{align} We refer to \cite[Remark 1.1, b)]{KW} for the definition 
  of $W^\sharp$ when $h^\vee=0$.
  Set  
\begin{equation}\label{deltabar}
\overline\D_0=\{\a\in\D_0\mid \tfrac{1}{2}\a\notin \D\},\qquad \overline\D_1=\{\a\in\D_1\mid  (\a,\a)=0\}.\end{equation}
Finally, for $w\in W_\g$, set 
\begin{equation}\label{sgn'}
sgn(w)=(-1)^{\ell(w)},\ 
sgn'(w)=(-1)^m,\end{equation} where $\ell$ is the usual length function on $W_\g$ and  $m$ is the number of reflections from $\overline\D^+_0$ occurring in an expression of $w$. 
 Note that
$$w\left(e^\rho\check R\right)=sgn'(w)\,e^\rho\check R.$$
In particular, $sgn'$ is well-defined.
Note  also that $sgn=sgn'$ in types $A$ and $D$.\par
 
 Let now recall the definition of odd reflections \cite{PS}: for an isotropic root 
 $\a\in\Pi$ we define 
\begin{equation}\label{or}r_\a(\Dp)=(\Dp\setminus\{\a\})\cup
 \{-\a\}.\end{equation}
 It is easy to prove that $r_\a(\Dp)$ is a set of positive roots for $\g$ and that if we set 
 \begin{equation}\label{oddref}r_\a(\beta)=\begin{cases}\a+\beta\quad&\text{if $(\a,\beta)\ne 0$}\\\beta\quad&\text{if $(\a,\beta)=0$, $\a\ne\be$}\\-\a\quad&\text{if $\beta=\a$}
 \end{cases}\end{equation}
 for $\beta\in\Pi$, then $r_\a(\Pi)=\{r_\a(\beta)\mid \beta\in \Pi\}$ is the corresponding set of simple roots. \par
 The importance of odd reflections lies in the fact that, up to  $W_\g$-action, any two sets of positive roots can be obtained one from the other by applying a finite sequence of odd refections.
  Note that, for an isotropic $\a$,
 \begin{equation}\label{rhoodref}\rho(r_\a(\Dp))=\rho(\Dp)+\a,\end{equation}
 from which one deduces that
\begin{equation}\label{erho}e^{\rho(\Dp)} \,\check R(\Dp)=-e^{\rho(r_\a(\Dp))} \,\check R(r_\a(\Dp))\quad\forall\a\,\in\Pi.\end{equation}
\section{Denominator formulas}\label{df}
\subsection{Arc diagrams}\label{ED}
Let $\g$ be a Lie superalgebra of type $gl(m,n), B(m,n), C(m),$ $D(m,n)$.
We first need an encoding of the  sets of positive roots.   The explicit realizations of these Lie superalgebras given in \cite{Kacsuper}, \cite[Section 4]{Kacsuperp} leads to a description of their roots in terms of functionals $\e_i,\d_i\in \h^*$.

We let $$\cE=\{\e_1,\ldots,\e_m\},\ \cD=\{\delta_1,\ldots,\d_n\}$$
if $\g$ is of type $gl(m,n), B(m,n)$ and let
$$\cE=\{\e_1,\ldots,\e_{m-1},\e_m\}\text{ or }\cE=\{\e_1,\ldots,\e_{m-1},-\e_m\}Ê,\quad \cD=\{\delta_1,\ldots,\d_n\}$$
if $\g$ is of type $C(m)$ or $D(m,n)$.
Let  
$$\mathcal B=\cE\cup\cD.$$
Then $\mathcal B$ is a
basis of $\fh^*$ (for $C(m)$ one has $n=1$).
We call two elements $v_1,v_2\in \mathcal B$
{\em elements of the same type} if $\{v_1,v_2\}\subset \cE$   or $\{v_1,v_2\}\subset\cD$
and {\em elements of different types} otherwise.\par
Recall the structure of $\D$ in terms of $\mathcal B$: in type
$gl(m,n)$ we have 
\begin{align*}
\D_0&=\{\e_i-\e_j\mid 1 \leq i\ne j\leq m\}\cup\{\d_i-\d_j\mid 1 \leq i\ne j\leq n\},\\
\D_1&=\{\pm(\e_i-\d_j)\mid 1\leq i\leq m, 1\leq j\leq n\}.
\end{align*}
As an invariant bilinear form on $gl(m,n)$ we choose the supertrace form  $(a,b)=str(ab)$, $str$ being the supertrace of a matrix in $gl(m,n)$, so that
$$(\e_i,\e_j)=\d_{ij}=-(\d_i,\d_j).$$
In the other classical types see  \eqref{B1}, \eqref{positivoD0} for the description of roots; the invariant bilinear form is the restriction of the supertrace form.\par
 Given a total order $>$ on $\mathcal B=\{\xi_1>\ldots>\xi_{m+n}\}$, we define a set of simple roots $\Pi(\mathcal B,>)$ for $\g$
as follows
\begin{equation*}\begin{tabular}{l|l}
$\g$&$\Pi(\mathcal B,>)$\\
\hline\\
$gl(m,n)$&$\{\xi_i-\xi_{i+1}\}_{i=1}^{m+n-1}$\\
$B(m,n)$&$\{\xi_i-\xi_{i+1}\}_{i=1}^{m+n-1}\cup\{\xi_{m+n}\}$\\
$C(m)$ & $\{\xi_i-\xi_{i+1}\}_{i=1}^{m}\cup\{2\xi_{m+1}\}$ if $\xi_{m+1}\in\cE$\\
& $\{\xi_i-\xi_{i+1}\}_{i=1}^{m}\cup\{\xi_{m}+\xi_{m+1}\}$ if $\xi_{m+1}\in\cD=\{\delta_1\}$\\
$D(m,n)$&$\{\xi_i-\xi_{i+1}\}_{i=1}^{m+n-1}\cup\{2\xi_{m+n}\}$ if $\xi_{m+n}\in\cD$\\
 & $\{\xi_i-\xi_{i+1}\}_{i=1}^{m+n-1}\cup\{\xi_{m+n-1}+\xi_{m+n}\}$ if $\xi_{m+n}\in\cE$.
\end{tabular}
\end{equation*}
Using Kac's description of Borel subalgebras (see \cite{Kacsuper}), it is not difficult to prove the following result.
\begin{lemma} Up to $W_\g$-equivalence, any set of positive  roots for $\g$ has $\Pi(\mathcal B,>)$ as a set of  simple roots for  some total order  $>$ on $\mathcal B$. 
\end{lemma}
It may happen in type $D(m,n)$ that two different choices of $\mathcal  B$ give rise to the same set of simple roots $\Pi$, so we make the following choice.
Let $h\in\h$ be such that $\a(h)=1$ for each $\a\in\Pi$. Choose 
$\mathcal B=\{\e_1,\ldots,\e_m\}\cup\{\d_1,\ldots,\d_n\}$ if $\e_m(h)\geq 0$, and
$\mathcal B=\{\e_1,\ldots,-\e_m\}\cup\{\d_1,\ldots,\d_n\}$ if $\e_m(h)<0$.
\vskip5pt
Instead of the Dynkin diagram,
we  encode $\Pi(\mathcal B,>)$ by the ordered sequence $\mathcal B$, which is  pictorially represented as an array of 
dots and crosses, the former corresponding to vertices in $\cE$ and the latter to vertices in $\cD$. 
\par
Examples for $gl(5,4), D(4,3)$ are given in Figures 1, 2, 3 disregarding the arcs. In the first case, we start from the total 
order $\{\e_1>\d_1>\e_2>\d_2>\d_3>\e_3>\e_4>\d_4>\e_5\}$, which corresponds to the set of simple roots
$\{\e_1-\d_1, \d_1-\e_2, \e_2-\d_2, \d_2-\d_3,\d_3-\e_3,\e_3-\e_4,\e_4-\d_5\}$. 
\par
In Figure 2, we start from the total 
order $\e_1>\d_1>\e_2>\d_2>\e_3>\d_3>\e_4$, which gives rise  to the set of simple roots
$\{\e_1-\d_1, \d_1-\e_2, \e_2-\d_2, \d_2-\e_3,\e_3-\d_3,\d_3\pm\e_4\}$, while in Figure 3 the total order is 
$\e_1>\d_1>\e_2>\d_2>\e_3>-\e_4>\d_4$, which corresponds to the set of simple roots $\{\e_1-\d_1, \d_1-\e_2, \e_2-\d_2, \d_2-\e_3,\e_3+\e_4,-\e_4\pm\d_3\}$.
\vskip5pt
For $v,w\in\mathcal B$,  if $v\geq w$, let
$[v,w]=\{u\in \mathcal B|\ v\geq u\geq w\}$. If $\mathcal B'\subset \mathcal B$, we denote by $W_{\mathcal B'}$ the
subgroup of $W_\g$ consisting of (non-signed) permutations of $\mathcal B'\cap\cE$ and of $\mathcal B'\cap\cD$, so that we have
$W_{\mathcal B'}\cong S_k\times S_l$, where $|\mathcal B'\cap\cE|=k,\ |\mathcal B'\cap\cD|=l$.
\vskip5pt
We   now introduce arc diagrams. These are combinatorial data that  encode some maximal isotropic subsets of the set of positive roots.

\begin{defi}\label{ad}
 An arc diagram is the datum consisting of the ordered sequence of vertices representing 
$\mathcal B$, and of  arcs between some of the vertices, satisfying the following properties:

(i) the vertices at the ends of each arc are of different type;

(ii) the arcs do not intersect  (including the end points);

(iii) for each arc $\stackrel{\frown}{v w}$   the interval $[v,w]$ contains
the same number of elements of  $\cE$ and of $\cD$:
$|[v,w]\cap\cE|=|[v,w]\cap\cD|$;

(iv) the number of arcs is $\min(m,n)$.
\end{defi}
\par
We also define the {\it support} of an arc diagram $X$ as
$$Supp(X)=\{v\in\mathcal B\mid v \text{ is an end of an arc in }ÊX\}.$$
If $\Dp$ is the set of positive roots corresponding to $\Pi(\mathcal B,>)$, we denote by $\mathcal A(\Dp)$ the set of arc diagrams  whose underlying set of vertices is $\mathcal B$ and the underlying total order is $>$.
\begin{figure}
  \setlength{\unitlength}{25pt}
  \begin{picture}(8,2)
 
    \put(0,0){\hbox{$\bullet$}}
    \put(0,-0.5){\hbox{$\e_1$}}
    \put(1,0){\hbox{$\times$}}
       \put(1,-0.5){\hbox{$\d_1$}}
    \put(2,0){\hbox{$\bullet$}}
       \put(2,-0.5){\hbox{$\e_2$}}
    \put(3,0){\hbox{$\times$}}
       \put(3,-0.5){\hbox{$\d_2$}}
    \put(4,0){\hbox{$\times$}}
       \put(4,-0.5){\hbox{$\d_3$}}
    \put(5,0){\hbox{$\bullet$}}
       \put(5,-0.5){\hbox{$\e_3$}}
     \put(6,0){\hbox{$\bullet$}}
        \put(6,-0.5){\hbox{$\e_4$}}
    \put(7,0){\hbox{$\times$}}
       \put(7,-0.5){\hbox{$\d_4$}}
     \put(8,0){\hbox{$\bullet$}}
     \put(8,-0.5){\hbox{$\e_5$}}  
    \qbezier(1.2,0.4)(1.7,1)(2.1,0.4)
    \qbezier(4.2,0.4)(4.7,1)(5.1,0.4)
    \qbezier(7.2,0.4)(7.7,1)(8.1,0.4)
   \qbezier(0.2,0.4)(1.65,2)(3.1,0.4)
  \end{picture}
  \vskip5pt
  \caption{An arc diagram $X$ for $gl(5,4)$.}
\end{figure}
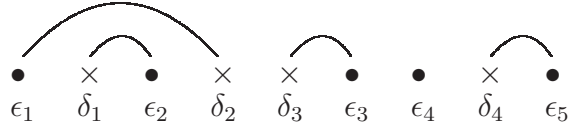

\begin{figure}
  \setlength{\unitlength}{25pt}
  \begin{picture}(6,1)
    \put(0,0){\hbox{$\bullet$}}
    \put(0,-0.5){\hbox{$\e_1$}}
    \put(1,0){\hbox{$\times$}}
       \put(1,-0.5){\hbox{$\d_1$}}
    \put(2,0){\hbox{$\bullet$}}
       \put(2,-0.5){\hbox{$\e_2$}}
    \put(3,0){\hbox{$\times$}}
       \put(3,-0.5){\hbox{$\d_2$}}
    \put(4,0){\hbox{$\bullet$}}
       \put(4,-0.5){\hbox{$\e_3$}}
    \put(5,0){\hbox{$\times$}}
       \put(5,-0.5){\hbox{$\d_3$}}
     \put(6,0){\hbox{$\bullet$}}
        \put(6,-0.5){\hbox{$\e_4$}}
    \qbezier(5.2,0.4)(5.65,1)(6.1,0.4)
    \qbezier(3.2,0.4)(3.65,1)(4.1,0.4)
   \qbezier(0.2,0.4)(0.65,1)(1.1,0.4)
  \end{picture}
  \vskip5pt
  \caption{An arc diagram for $D(4,3)$.}
\end{figure}
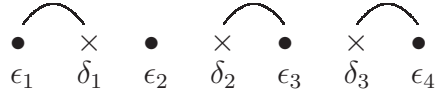
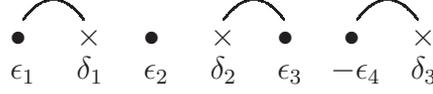
\begin{figure}
  \setlength{\unitlength}{25pt}
  \begin{picture}(6,1)
    \put(0,0){\hbox{$\bullet$}}
    \put(0,-0.5){\hbox{$\e_1$}}
    \put(1,0){\hbox{$\times$}}
       \put(1,-0.5){\hbox{$\d_1$}}
    \put(2,0){\hbox{$\bullet$}}
       \put(2,-0.5){\hbox{$\e_2$}}
    \put(3,0){\hbox{$\times$}}
       \put(3,-0.5){\hbox{$\d_2$}}
    \put(4,0){\hbox{$\bullet$}}
       \put(4,-0.5){\hbox{$\e_3$}}
    \put(5,0){\hbox{$\bullet$}}
       \put(4.8,-0.5){\hbox{$-\e_4$}}
     \put(6,0){\hbox{$\times$}}
        \put(6.0,-0.5){\hbox{$\d_3$}}
  \qbezier(3.2,0.4)(3.65,1)(4.1,0.4)
   \qbezier(0.2,0.4)(0.65,1)(1.1,0.4)
   \qbezier(5.2,0.4)(5.65,1)(6.1,0.4)
  \end{picture}
  \vskip5pt
  \caption{Another arc diagram for $D(4,3)$.}
\end{figure}

\subsubsection{The maximal isotropic set $S(X)$}
Consider an arc diagram $X$. It  encodes the following isotropic set of positive roots:
$$S(X)=\{v-w|\ v,w\in \mathcal B \text{ are connected by an arc and } v>w\}.$$
For example, the arc diagram 
in Figure 1 encodes the isotropic subset $S=\{\d_1-\e_2,\e_1-\d_2,\d_3-\e_3,\d_4-\e_5\}$. The arc diagram 
in Figure 2 encodes the  isotropic subset $S=\{\e_1-\d_1,\d_2-\e_3,\d_3-\e_4\}$, whereas
that in Figure 3 encodes the isotropic subset $S'=\{\e_1-\d_1,\d_2-\e_3,-\d_3-\e_4\}$.

We will often denote by $\Dp(X),\Dp_0(X),\Dp_1(X),\Pi(X)$ the sets of positive, positive even, positive odd, and simple roots associated to the ordering of vertices of $\mathcal B$ underlying $X$ (the arc structure of $X$ is irrelevant for that).
One readily sees that
$S(X)\subset \Delta^+(X)$ and $S(X)$ is a basis
of a maximal isotropic subspace in $\fh^*$ (since the cardinality of $S(X)$  equals the defect of $\g$). This shows that $S(X)$ is indeed a maximal set of isotropic roots.

For each $\beta\in \Delta_1$ let $$\sn \beta=(\beta,\sum_{i=1}^m\e_i)$$
(that is $\sn (\e_i\pm\delta_j)=1,\ \sn (-\e_i\pm\delta_j)=-1$).
For $\beta\in S(X)$ one has  $\sn\beta=1$ (resp., $-1$)
if the left end of the corresponding arc is in $\cE$ (resp., in $\cD$).
For each $\alpha\in S(X)$ we have (cf. \eqref{parentesi})
\begin{equation}\label{aaaaa}\llbracket \alpha\rrbracket =\sum_{\beta\in S(X), \beta\leq \alpha} \sn\alpha\cdot \sn \beta\cdot \beta,\end{equation}
namely
$$\llbracket \e_i-\delta_j\rrbracket =\sum_{\xi\in [\e_i,\delta_j]\cap\cE}\xi-
\sum_{\xi\in [\e_i,\delta_j]\cap\cD}\xi,\ \ \ \
\llbracket \delta_i-\e_j\rrbracket =\sum_{\xi\in [\delta_i,\e_j]\cap\cD}\xi-
\sum_{\xi\in [\delta_i,\e_j]\cap\cE}\xi.$$
 E.g., with reference to Figure 1, we have
$$\llbracket \e_1-\d_2\rrbracket =\e_1+\e_2-\d_1-\d_2,\quad
\llbracket \d_1-\e_2\rrbracket =\d_1-\e_2.$$

 Notice that for an arc $\stackrel{\frown}{v w}$ the element $\llbracket  v-w\rrbracket $ is $W_{[v,w]}$-invariant.
\vskip1pt
\subsubsection{Special arc diagrams}
\begin{defi}\label{sad}
We say that an arc is simple if  it  connects consecutive vertices.
We call an arc diagram {\em simple} if each arc is simple.
\par

We call an arc diagram $X$ nice
if for each $\alpha,\beta\in S(X)$ such that $\alpha<\beta$ one has $\sn \alpha=\sn \beta$
(equivalently, if for any arc,  its  left end   is of the same type
as the left ends of all arcs which are below this arc).
\end{defi}
One readily sees that simple  diagrams correspond to the case when
$S(X)$ consists of simple roots, and
 nice arc diagrams correspond to the case when $\llbracket \gamma\rrbracket =\sum_{\beta\in S(X), \beta\leq\gamma}\beta$. In particular, for a nice arc diagram $X$,
$\llbracket \gamma\rrbracket \in Q^+$ for each $\gamma\in S(X)$. We denote by $\mathcal A^{nice}(\Dp)$ the set of nice arc diagrams $X$ such that $\Dp=\Dp(X)$.

\subsubsection{Existence of arc diagrams and  of nice arc diagrams}
Fix an  order $>$ on $\mathcal B$, set $\Pi=\Pi(\mathcal B, >)$ and let $\Dp$ be the corresponding set of positive roots.
Arc diagrams can be easily constructed inductively.
We start by drawing an arc between any two consecutive  vertices $v,w$ of different types.
Then we throw away these two vertices and obtain an ordered set $\cE'\cup\cD'$,
where the cardinality of $\cE'$ (resp., of $\cD'$) is less by one than
the cardinality of $\cE$ (resp. of $\cD$). We take any arc diagram for $\cE'\cup\cD'$,
and we let $\Sigma(X')$ be the set of its arcs. Then $\Sigma(X)=\Sigma(X')\cup \{ \stackrel{\frown}{v w}\}$ is the set of arcs
of an arc diagram $X$ for $\mathcal B$.

More generally, we can start by drawing
any possible arc (i.e., the arc between two vertices $v,w$ of different types
such that $[v,w]$ has the same numbers of vertices lying in $\cE$ and in $\cD$).
Then we can construct an arc diagram for the sequence which is below the arc
(i.e., for $[v,w]\setminus\{v,w\}$)
and for the  ordered set $\mathcal B\setminus [v,w]$.
The set of arcs in the resulting arc diagram
is the union of the sets of arcs in these two diagrams and the arc $\stackrel{\frown}{v w}$.
As a result, there is an arc diagram $X$ such that  $\beta\in\Delta_1^{+}$ belongs to $S(X)$ iff
the interval defined by $\beta $ (i.e., $[v,w]$ for $\beta=v-w$) contains
the same numbers of vertices lying in $\cE$ and in $\cD$.
\par
Look now at  the above procedure from an algebraic point of view: 
removing a pair $\{u,v\}$ of vertices of different type (or more generally a set $Z=\{u_i,v_i\}_{i=1}^k$ of pairs of vertices of different type) from the ordered set $\mathcal B$  gives the ordered set corresponding  to  $S^\perp\setminus S$ where $S=\{u-v\}$ (resp, $S=\{u_i-v_i\}_{i=1}^k$). This is clearly related to the procedure explained in Definition \ref{cp}. More precisely, we have \begin{prop}\label{cad}ÊIf $X$ is an arc diagram, then $S(X)\in\cS(\Dp)$. Viceversa, any $S\in\cS(\Dp)$ is of the form $S(X)$ for some arc diagram $X$.\end{prop}
\begin{proof} Both assertions follow easily from the remark in the previous paragraph after it is proved that 
an arc diagram has always a simple arc. To prove this, remark that axioms  in  Definition \ref{ad}  
imply that all vertices below a given arc are ends of some arcs. Thus the ``lowest'' arcs (the ones which do not have arcs below them) are necessarily simple.
\end{proof}
Let us now explain how to construct  nice diagrams.
\begin{prop}\label{exnice} $\mathcal A^{nice}(\Dp)\ne\emptyset$.
\end{prop}
\begin{proof}
 Let $\e_1>\e_2>\ldots>\e_m$ and
$\delta_1>\delta_2>\ldots>\delta_n$. Assume that $\e_1>\delta_1$, so that 
our sequence $A_0=\mathcal B$ starts with $\e_1,\ldots,\e_k,\delta_1$ ($k\geq 1$).
We draw the arc $a_1=\stackrel{\frown}{\e_k\delta_1}$
and consider the sequence $A_1=\mathcal B\setminus\{\e_k,\delta_1\}$. We repeat the procedure until
the first element of the sequence $A_{s+1}$ is not in $\cE$.  Note that the right end of the
arc $a_j$ obtained in the $j$th step is $\delta_j$.
Let $U$ be the union of the ends in all  arcs $a_j$, $j=1,\ldots,s$.
Let us show that $U=[\e_1,\delta_s]$.
Indeed, by the above, for $j=1,\ldots,s$ the right end of $a_j$ is $\delta_j$
so both ends lie in $[\e_1,\delta_j]\subset [\e_1,\delta_s]$.
Therefore $U\subset [\e_1,\delta_s]$.
Moreover, $\delta_1,\ldots,\delta_s\in U$
so $([\e_1,\delta_s]\setminus U)\subset \cE$.
Since the first element in $A_{s+1}=A_0\setminus U$ lies in $\cD$, we conclude that
$[\e_1,\delta_s]=U$ as required. Thus the sequence $[\e_1,\delta_s]$ with the ends of the arcs in $U$
is a nice diagram $X$ and each element of this sequence is an end of an arc.
Now we construct a nice diagram $Y$ for the sequence $A_{s+1}$ (which is
shorter than $A_0$). The union of $X$ and $Y$ is a nice arc diagram for the sequence $A_0$.
\end{proof}

\subsubsection{Operations with the arc diagrams}\label{operation}
Consider all arc diagrams on the set $\mathcal B$ endowed by a  total order and
introduce the following operations.
\vskip5pt
{\em Odd reflections $r_{v-w}$.} Let $v,w$ be two consecutive vertices of an arc diagram $X$
connected by an arc $\stackrel{\frown}{v w}$.
 We define $r_{v-w}(X)$ to be the  arc diagram with $v,w$ switched 
(i.e., $r_{v-w}(X)=(X\backslash\{\stackrel{\frown}{v w}\})\cup\{\stackrel{\frown}{w v}\}$; the total order on $\mathcal B$ for $r_{v-w}(X)$ is obtained from the total order for
         $X$ by interchanging $v$ and $w$). Figure 4 displays $r_{\d_1-\e_2}(X)$ where $X$ is the arc diagram of Figure 1.
\begin{figure}
  \setlength{\unitlength}{25pt}
  \begin{picture}(8,2)
    \put(0,0){\hbox{$\bullet$}}
    \put(0,-0.5){\hbox{$\e_1$}}
    \put(1,0){\hbox{$\bullet$}}
       \put(1,-0.5){\hbox{$\e_2$}}
    \put(2,0){\hbox{$\times$}}
       \put(2,-0.5){\hbox{$\d_1$}}
    \put(3,0){\hbox{$\times$}}
       \put(3,-0.5){\hbox{$\d_2$}}
    \put(4,0){\hbox{$\times$}}
       \put(4,-0.5){\hbox{$\d_3$}}
    \put(5,0){\hbox{$\bullet$}}
       \put(5,-0.5){\hbox{$\e_3$}}
     \put(6,0){\hbox{$\bullet$}}
        \put(6,-0.5){\hbox{$\e_4$}}
    \put(7,0){\hbox{$\times$}}
       \put(7,-0.5){\hbox{$\e_5$}}
     \put(8,0){\hbox{$\bullet$}}
     \put(8,-0.5){\hbox{$\d_4$}}  
    \qbezier(1.2,0.4)(1.7,1)(2.1,0.4)
    \qbezier(4.2,0.4)(4.7,1)(5.1,0.4)
    \qbezier(7.2,0.4)(7.7,1)(8.1,0.4)
   \qbezier(0.2,0.4)(1.65,2)(3.1,0.4)
  \end{picture}
  \vskip5pt
  \caption{$r_{\d_1-\e_2}(X)$.}
\end{figure}         
    The odd reflection on  an arc diagram corresponds to an odd reflection with respect to a simple root lying in $S(X)$ (cf. \eqref{or}):

\begin{align*}
     \Pi\bigl(r_{v-w}(X)\bigr)&=r_{v-w}(\Pi(X)),\\
      S(r_{v-w}(X))&=S(X)\setminus\{v-w\}\cup \{w-v\}.
      \end{align*}

Notice that $\sn\alpha\cdot\alpha=\sn(-\alpha)\cdot (-\alpha)$ and thus for
$\gamma\in S(X)\cap S(r_{v-w}(X))$
the element $\llbracket \gamma\rrbracket $ defined for $X$ and for $r_{v-w}(X)$ is the same.

\vskip5pt
{\em Interval reflections $r_{[e_1,d_k]}$.} Suppose $X$ has a subsequence $e_1,d_1,e_2,d_2,\ldots, e_k,d_k\,$ $(k>1),$
where $e_1,\ldots,e_k$ are of the same type
and $d_1,\ldots, d_k$ are of another type, with the arcs
$
\stackrel{\frown}{e_1 d_k},
\stackrel{\frown}{d_1 e_2},
\stackrel{\frown}{d_2 e_3},
 \ldots,
 \stackrel{\frown}{d_{k-1} e_k}.$
We define  $r_{[e_1,d_k]}(X)$ as
the arc diagram with the same total order on $\mathcal B$, where
the above arcs are substituted by the arcs 
$
\stackrel{\frown}{e_1 d_1},
\stackrel{\frown}{e_2 d_2},
\stackrel{\frown}{e_3 d_3},
 \ldots,
 \stackrel{\frown}{e_{k} d_k}.$
 Hence
\begin{align*}
&S\bigl(r_{[e_1,d_k]}(X)\bigr)=\\&
(S(X)\setminus\{e_1-d_k,d_1-e_2,\ldots,e_k-d_k\})\cup
\{e_1-d_1, e_2-d_2,\ldots, e_k-d_k\}.
\end{align*}
Figure~5 displays $r_{[\e_1,\d_2]}(X)$ where $X$ is the arc diagram of Figure 1. Figure 6 provides further examples: the  arc diagram $X$ in the left display  is not nice, since $\llbracket \e_1-\d_2\rrbracket=
(\e_1-\d_2)-(\d_1-\e_2)$. The middle display  represents $r_{\d_1-\e_2}(X)$, which is nice, 
and the right display represents $r_{[\e_1,\d_2]}(X)$, which is simple (hence, in particular, nice).
This example also shows that both odd and interval reflections are necessary moves to change an arc  diagram into a simple one. They are also sufficient, as we show next.

\begin{figure}
  \setlength{\unitlength}{25pt}
  \begin{picture}(8,1)
    \put(0,0){\hbox{$\bullet$}}
    \put(0,-0.5){\hbox{$\e_1$}}
    \put(1,0){\hbox{$\times$}}
       \put(1,-0.5){\hbox{$\d_1$}}
    \put(2,0){\hbox{$\bullet$}}
       \put(2,-0.5){\hbox{$\e_2$}}
    \put(3,0){\hbox{$\times$}}
       \put(3,-0.5){\hbox{$\d_2$}}
    \put(4,0){\hbox{$\times$}}
       \put(4,-0.5){\hbox{$\d_3$}}
    \put(5,0){\hbox{$\bullet$}}
       \put(5,-0.5){\hbox{$\e_3$}}
     \put(6,0){\hbox{$\bullet$}}
        \put(6,-0.5){\hbox{$\e_4$}}
    \put(7,0){\hbox{$\times$}}
       \put(7,-0.5){\hbox{$\e_5$}}
     \put(8,0){\hbox{$\bullet$}}
     \put(8,-0.5){\hbox{$\d_4$}}  
    \qbezier(2.2,0.4)(2.7,1)(3.1,0.4)
    \qbezier(4.2,0.4)(4.7,1)(5.1,0.4)
    \qbezier(7.2,0.4)(7.7,1)(8.1,0.4)
   \qbezier(0.2,0.4)(0.65,1)(1.1,0.4)
  \end{picture}
  \vskip5pt
  \caption{$r_{[\e_1,\d_2]}(X)$.}
\end{figure}
\begin{lemma}\label{esad} There exists a finite sequence of odd and interval reflections which change 
 any arc diagram
into a simple arc diagram.
\end{lemma}
\begin{proof} We proceed by induction on the number $k$ of non-simple arcs in the arc diagram $X$. If $k=0$, then $X$ is a simple arc diagram and there is nothing to prove.
Assume $k\geq 1$ and let $\stackrel{\frown}{e d}$ be a non-simple  arc  such that each arc $\stackrel{\frown}{vw}, \, e>v>w>d$ is simple. Then, necessarily, we have
$[e,d]=e a_1a_1' a_2a_2'\ldots a_ha_h'd$, with arcs $\stackrel{\frown}{e d},\,\stackrel{\frown}{a_i a_i'},\,i=1,\ldots,h$. Use  the odd reflections  $r_{a_i-a'_i}$, if necessary, to 
modify $[e,d]$ in such a way that vertices of different type alternate from $e$ to $d$: call the resulting diagram $X'$. Then   $r_{[e,d]}(X')$ has $k-1$ non-simple arcs and we are done by induction.
\end{proof}
\begin{figure}
  \setlength{\unitlength}{25pt}
  \begin{picture}(4,2)
      \put(0,0){\hbox{$\bullet$}}
      \put(0,-0.5){\hbox{$\e_1$}}
    \put(1,0){\hbox{$\times$}}
    \put(1,-0.5){\hbox{$\d_1$}}
    \put(2,0){\hbox{$\bullet$}}
    \put(2,-0.5){\hbox{$\e_2$}}
        \put(3,0){\hbox{$\times$}}
        \put(3,-0.5){\hbox{$\d_2$}}
     \qbezier(1.2,0.4)(1.7,1)(2.1,0.4)
   \qbezier(0.2,0.4)(1.65,2)(3.1,0.4)
  \end{picture}\quad
    \begin{picture}(4,2)
      \put(0,0){\hbox{$\bullet$}}
      \put(0,-0.5){\hbox{$\e_1$}}
    \put(1,0){\hbox{$\bullet$}}
    \put(1,-0.5){\hbox{$\e_2$}}
    \put(2,0){\hbox{$\times$}}
    \put(2,-0.5){\hbox{$\d_1$}}
        \put(3,0){\hbox{$\times$}}
        \put(3,-0.5){\hbox{$\d_2$}}
     \qbezier(1.2,0.4)(1.7,1)(2.1,0.4)
   \qbezier(0.2,0.4)(1.65,2)(3.1,0.4)
  \end{picture}\quad
    \begin{picture}(4,2)
        \put(0,0){\hbox{$\bullet$}}
      \put(0,-0.5){\hbox{$\e_1$}}
    \put(1,0){\hbox{$\times$}}
    \put(1,-0.5){\hbox{$\d_1$}}
    \put(2,0){\hbox{$\bullet$}}
    \put(2,-0.5){\hbox{$\e_2$}}
        \put(3,0){\hbox{$\times$}}
        \put(3,-0.5){\hbox{$\d_2$}}
          \qbezier(0.2,0.4)(0.65,1)(1.1,0.4)
           \qbezier(2.2,0.4)(2.65,1)(3.1,0.4)
    \end{picture}

  \caption{ }
\end{figure}

\subsection{Proof  of  Theorem \ref{princ}}In this Section we prove formulas \eqref{princf}, \eqref{princff}. For future applications we prove a slightly more general result
(see Proposition \ref{migliore}).
\vskip5pt
For any subset $U$ of $W_\g$  and any rational function $Y$ with $\{e^b\mid b\in\mathcal B\}$ as set of variables, introduce the sum
\begin{equation}\cF_U(Y)=\sum_{w\in U}sgn(w)\, w(Y).\end{equation}
We will need also the sum 
\begin{equation}{\check \cF}_U(Y)=\sum_{w\in U}sgn'(w)\, w(Y).\end{equation}
\begin{nota}\label{W} If $U$ is a subgroup of $W_\g$, and $W\subset W_\g$ is stable under the right action of $U$, then $W$ is a union of right cosets in $W_\g/U$. We denote by $W/U$ a set of representatives.\end{nota}
Note that 
$$\cF_W(Y)=\cF_{W/U}\bigl(\cF_U(Y)\bigr),\quad
{\check \cF}_W(Y)={\check \cF}_{W/U}\bigl({\check \cF}_U(Y)\bigr).$$
The formula displayed on the right follows since $sgn': W_\g\to\{\pm 1\}$ is a homomorphism.\par

\begin{lemma}\label{lem1}
Let $\Pi$ be a set of simple  roots for $\g=gl(k|k), k\geq 2$, all of which are isotropic. Let
$\{\beta_1,\ldots,\beta_k\}\subset \Pi$  be the maximal isotropic subset.
Then
\begin{equation}\label{f35}{\check \cF}_{W_\g}\bigl(\frac{1}{\prod_{i=2}^{k}(1-e^{-\beta_i})}\bigr)
=\frac{1-e^{-\sum_{i=1}^k\beta_i}}{k}{\check \cF}_{W_\g}\bigl(\frac{1}{\prod_{i=1}^{k}(1-e^{-\beta_i})}\bigr).\end{equation}
\end{lemma}
\begin{proof}
Let $A=\frac{1}{\prod_{i=1}^{k}(1-e^{-\beta_i})}$, $x_0={\check \cF}_{W_\g}(A)$ and for $j=1,\ldots, k$ let 
$$x_j={\check \cF}_{W_\g}\bigl(Ae^{-\beta_1-\beta_2-\ldots-\beta_j}\bigr).$$
Then the left-hand side of formula \eqref{f35} is equal to $x_0-x_1$: $LHS=x_0-x_1$.

Write $\Pi=\{\e_1-\delta_1,\delta_1-\e_2,\ldots,\e_k-\delta_k\}$ and
$\beta_i=\e_i-\delta_i$.
Observe that ${W_\g}$ contains a subgroup permuting the $\beta_i$'s, and all elements
of this subgroup are even (the permutations
$(\e_i\ \e_j)(\delta_i\ \delta_j)$ switches $\beta_i$ and $\beta_j$).
Since ${\check \cF}_{W_\g}(A')$ is ${W_\g}$-skew-invariant for any rational function $A'$, we conclude that
\begin{equation}\label{37}x_j={\check \cF}_{W_\g}\bigl(Ae^{-\sum_{i\in J}\beta_i}\bigr)\end{equation}
for any subset $J\subset \{1,2,\ldots,k\}$ of cardinality $j$.
Since  the permutation $(\e_1\ \e_2)$ stabilizes $A$,    one has
$${\check \cF}_{W_\g}\bigl(A\prod_{i=1}^s(1-e^{-\beta_i})\bigr)=0$$
for $s=2,\ldots,k$. Hence, using  \eqref{37}, we obtain that $\sum_{j=0}^s (-1)^j\binom{s}{j}x_j=0$ for $s=2,\ldots,k$.
We deduce by induction on $j\geq 2$ that
$x_j=j(x_1-x_0)+x_0$ and thus
$$LHS=x_0-x_1=\frac{x_0-x_k}{k}=\frac{1}{k}{\check \cF}_{W_\g}\bigl((1-e^{-\sum_{i=1}^k\beta_i})Y\bigr)=\frac{1-e^{-\sum_{i=1}^k\beta_i}}{k}{\check \cF}_{W_\g}(A)$$
which is \eqref{f35}.
\end{proof}

For an arc diagram $X$  we denote by $\rho_X$ the element $\rho_{\Pi(X)}$ and by $R_X,\check R_X$ the fractions $R,\,\check R$
constructed for $\Pi(X)$, that is
$$R_X=\frac{\prod_{\alpha\in\Delta_0^+(X)}(1-e^{-\alpha})}{\prod_{\alpha\in\Delta_1^+(X)}(1+e^{-\alpha})},\quad
\check R_X=\frac{\prod_{\alpha\in\Delta_0^+(X)}(1-e^{-\alpha})}{\prod_{\alpha\in\Delta_1^+(X)}(1-e^{-\alpha})}
.$$
Let
\begin{equation}\label{PP}\cP(X)=\prod_{\gamma\in S(X)}\frac{\htt\gamma+1}{2}\cdot \frac{e^{\rho_X}}
{\prod_{\gamma\in S(X)} (1-e^{-\llbracket \gamma\rrbracket })}.\end{equation}

\begin{cor}\label{cor1}
Let $X$ be an arc diagram and $r_{[v,w]}$ be an interval reflection. For any subset $W$ of the Weyl group which is stable under the right action of
$W_{Supp(X)}$ one has
\begin{equation}\label{intervalreflection}{\check \cF}_W\bigl(\cP(X)\bigr)={\check \cF}_W\bigl(\cP(r_{[v,w]}(X))\bigr).\end{equation}
\end{cor}
\begin{proof}
Denote by $Y$ the arc subdiagram corresponding to the interval $[v,w]$ and let $Y'=r_{[v,w]}(Y)$. Then $X'=r_{[v,w]}(X)$ is obtained
from $X$ by substituting $Y$ by $Y'$. View $Y,Y'$ as arc diagrams of $gl(k,k)$-type.
Then $G=W_{[v,w]}$ is the Weyl group of $Y$ and of $Y'$.

Notice that $(\rho_X-\rho_Y,\alpha)=0$
for each $\alpha\in\Pi(Y)$
and thus $\rho_X-\rho_Y$ is $G$-invariant. Since $[v,w]\subseteq Supp(X)$, we have that $WG=W$. Therefore
$${\check \cF}_W\bigl(\cP(X)\bigr)={\check \cF}_{W/G}\bigl(\frac{e^{\rho_X-\rho_Y}}
{\prod_{\beta\in S(X)\setminus S(Y)}(1-e^{-\llbracket \beta\rrbracket })}\cdot {\check \cF}_G(\cP(Y)\bigr)\bigr)$$
 and a  similar formula holds for $X',Y'$ respectively. 
One has $S(X)\setminus S(Y)= S(X')\setminus S(Y')$. Moreover,
since $\Pi(Y)=\Pi(Y')$ and $\Pi(X)=\Pi(X')$,
one has $\rho_Y=\rho_{Y'}$ and $\rho_X=\rho_{X'}$. Thus formula \eqref{intervalreflection} follows from the following equality
\begin{equation}\label{formula2}
{\check \cF}_G(\cP(Y))={\check \cF}_G(\cP(Y')),
\end{equation}
which we now prove.

Recall that, by  definition of interval reflection, the interval $[v,w]$ is of the form
$v=e_1>d_1>e_2>d_2>\ldots>e_s>d_s=w$, where the $e_i$'s are of the same type and
the $d_i$'s are of another type. Then $S(Y)=\{\alpha\}\cup S'$, where $\alpha=v-w=e_1-d_s$ and
$S'=\{d_i-e_{i+1}\}_{i=1}^{s-1}$. Since $S'\subset\Pi(Y)$ one has
$\llbracket \beta\rrbracket =\beta$ for $\beta\in S'$. Recall that $\llbracket \alpha\rrbracket $ is $G$-invariant and note that
$\rho_Y$ is also $G$-invariant (since $\Pi(Y)$ consists of odd roots).
Therefore
$${\check \cF}_G(\cP(Y))=\frac{e^{\rho_Y}}{1-e^{-\llbracket \alpha\rrbracket }}\cdot
{\check \cF}_G\bigl(\frac{1}{\prod_{\beta\in S'}(1-e^{-\beta})}\bigr).$$

Consider the simple arc diagram $Z$ with the order $d_1>e_2>d_2>\ldots>d_s>e_1$; then
$S(Z)=S'\cup\{-\alpha\}$.
Using~\Lem{lem1} we obtain
$${\check \cF}_G\bigl(\frac{1}{\prod_{\beta\in S'}(1-e^{-\beta})}\bigr)
=\frac{1-e^{-\sum_{\beta\in S'}\beta+\alpha}}{s}
{\check \cF}_G\bigl(\frac{1}{\prod_{\beta\in S(Z)}(1-e^{-\beta})}\bigr).$$
Observe that $s=\frac{\htt\alpha+1}{2}=\frac{\htt(v-w)+1}{2}$ and $\llbracket \alpha\rrbracket =\alpha-
\sum_{\beta\in S'}\beta=\rho_Z-\rho_Y$.
Summarizing, we obtain
\begin{align*}
{\check \cF}_G(\cP(Y))&=e^{\rho_Y}\frac{1-e^{\llbracket \alpha\rrbracket }}{1-e^{-\llbracket \alpha\rrbracket }}
{\check \cF}_G\bigl(\frac{1}{\prod_{\beta\in S(Z)}(1-e^{-\beta})}\bigr)\\
&=-e^{\rho_Z}{\check \cF}_G\bigl(\frac{1}{\prod_{\beta\in S(Z)}(1-e^{-\beta})}\bigr)=-C\check  R_Ze^{\rho_Z},
\end{align*}
where the last equality follows from \eqref{kw1} (because $Z$ is a simple arc diagram).
Since $S(Y')\subset \Pi(Y')$, Theorem KWG gives also ${\check \cF}_G(\cP(Y'))=C \check R_{Y'}e^{\rho_{Y'}}$.

Since $\Pi(Z)$ corresponds to the total order $d_1>e_2>d_2>\ldots>d_s>e_1$
and $\Pi(Y')$ corresponds to the total order $e_1>d_1>e_2>d_2>\ldots>d_s$, one has
$\Pi(Z)=r_{e_1-d_s}\ldots r_{e_1-e_2}r_{e_1-d_1}(\Pi(Y'))$. Using~(\ref{refl}) we get
$$\check R_{Y'}e^{\rho_{Y'}}=(-1)^{2s-1}\check  R_Ze^{\rho_Z}=-\check  R_Ze^{\rho_Z}.$$
This establishes~(\ref{formula2}) and completes the proof.
\end{proof}
To complete the proof of Theorem \ref{princ}, we will need the following observation.
\begin{lemma}\label{ded} Formula \eqref{princf} follows from \eqref{princff}.
\end{lemma}
\begin{proof} Fix any set of positive roots $\Dp$ and let $\Pi$ be the associated set of simple roots. Take  $h\in\h$ such that $\a(h)= 0$ if $\a\in\Pi\cap\D_0$ and $\a(h)=1$ if $\a\in\Pi\cap\D_1$. 
  Then  $\a(h)\equiv 0 \mod 2$ if $\a\in\D_0$ and $\a(h)\equiv1 \mod 2 $ if $\a\in\D_1$. We claim that  $F(e^\a)=e^{\pi\sqrt{-1}\a(h)}e^\a$
changes   \eqref{princff} to \eqref{princf}. Let $k=e^{\pi\sqrt{-1}\rho(h)}$. Then, obviously,  $F(e^\rho)= k e^{\rho}$, so we have $F(e^{\rho}\check R)=ke^\rho R$. We need only to check that $F(e^{w(\rho)})=k\frac{sgn(w)}{sgn'(w)}e^{w(\rho)}$. First of all observe that $F(w (e^{\rho} \check R))=sgn'(w) F(e^{\rho} \check R) =sgn'(w)ke^\rho R$. On the other hand
this equals $F(e^{w(\rho) }w \check R)=F(e^{w(\rho)})F(w\check R)$. Since $w$ permutes the roots of the same parity, we have that $F(w\check R)=w(F(\check R))$ and in turn  $F(e^{w(\rho)})F(w\check R)=F(e^{w(\rho)})w(F(\check R))$. It follows that $F(e^{w(\rho) }w \check R)= e^{\pi\sqrt{-1}w(\rho)(h)}e^{w(\rho)}w(R)=e^{\pi\sqrt{-1}w(\rho)(h)}sgn(w) e^{\rho}R
$. Hence $$e^{\pi\sqrt{-1}w(\rho)(h)}sgn(w)=k\,sgn'(w),$$ and this relation implies our claim.
 \end{proof}
 
We are finally ready to give the proof of Theorem \ref{princ}:  in the current setting, formula \eqref{princff} becomes
\begin{equation}
\label{formula1}
{\check \cF}_{W_\g}(\cP(X))=C_\g e^{\rho_X}\check R_X.
\end{equation}
By Lemma \ref{esad},   any arc diagram can be transformed to a simple one by a sequence of
odd and interval reflections.
By Theorem KWG, ${\check \cF}_{W_\g}(\cP(X))=C_\g\check  R_X e^{\rho_X}$ if $X$ is a simple diagram.
For any  odd reflection $r_{\alpha}$  of $\Pi$ one has
$\rho_{r_{\alpha}(\Pi)}=\rho_{\Pi}+\alpha$ (cf. \eqref{rhoodref}),
and so
\begin{equation}\label{refl}
e^{\rho_{r_{\alpha}\Pi}} \check R_{r_{\alpha}\Pi}= - e^{\rho_{\Pi}} \check R_{\Pi}.
\end{equation}
Since for $\alpha\in S(X)$  one has $\cP(r_{\alpha}(X))=-\cP(X)$,
the fraction ${\check \cF}_{W_\g}\bigl(\cP(X)\bigr)/(\check  R_Xe^{\rho_X})$ is not changed by the action of $r_{\alpha}$.
The interval reflections do not change $\check  R_Xe^{\rho_X}$, since they do not change $\Pi(X)$;
moreover, applying \Cor{cor1} with $W=W_\g$,
they do not change ${\check \cF}_{W_\g}(\cP(X))$. This proves  formula~\eqref{formula1}, hence \eqref{princff}. Lemma \ref{ded} implies that \eqref{princf} also holds. By choosing $X$ to be a nice arc diagram, we have that $\llbracket \gamma\rrbracket\in Q^+$ for any $\gamma\in S(X)$. This concludes the proof of Theorem \ref{princ}.
\vskip5pt
\subsubsection{}\label{mmm} For future applications we will need a slightly stronger version of Theorem \ref{princ}. Given an arc diagram $X$, let $\mathcal B'$ be a subset of $\mathcal B$ containing $Supp(X)$. Let $\D(\mathcal B')$ be the set of roots that are linear combinations of the simple roots that are in the span of $\mathcal B'$. Assume for simplicity that $\D(\mathcal B')$ is irreducible  and let $\D^\sharp(\mathcal B')$ be the irreducible component of $\D_0(\mathcal B')$ which is not the smallest one in the sense of \cite[Section 1.2]{Gor}. Let $W(\mathcal B')$ and $W^\sharp(\mathcal B')$ be the  corresponding Weyl groups. Clearly $W_{\mathcal B'}W^\sharp(\mathcal B')$ is a subgroup of $W(\mathcal B')$ and let $T=(W_{\mathcal B'}W^\sharp(\mathcal B'))\backslash W(\mathcal B')$. Set $Z=W_\g/W(\mathcal B')$ and $W_0=ZW_{\mathcal B'}W^\sharp(\mathcal B')$ so that $W_\g=W_0T$.
\begin{prop}\label{migliore}
\begin{equation}
\label{formula3}
{\check \cF}_{W_0}(\cP(X))=\frac{C_\g}{|T|}e^{\rho_X}\check{R}_X.
\end{equation}
\end{prop}
\begin{proof}  Recall that any arc diagram can be transformed into a simple arc diagram
by a sequence of odd and interval reflections.
Note that these reflections permute the ends of arcs and
do not change  the positions of other vertices
(the interval reflections do not change the order of vertices and the odd reflections
permute two vertices connected by an arc). Thus these reflections do not change $\mathcal B'$ and $Supp(X)$. 
Since $W_0W_{Supp(X)}=W_0$, we can argue as in the proof of Theorem \ref{princ} and assume that $X$ is simple. Let $Y$ be the same arc diagram viewed as a diagram for  $\D(\mathcal B')$. Then, since $(\rho_X-\rho_Y,\a)=0$ for any simple root of $\D(\mathcal B')$, we have that 
 $\rho_X-\rho_Y$ is $W(\mathcal B')$-invariant.
Since $X$ is simple, $Y$ is also simple and \eqref{kw1} gives
$$
\cF_{W^\sharp(\mathcal B')}(\cP(Y))\bigr)=\frac{1}{|T|}\cF_{W^\sharp(\mathcal B')T}(\cP(Y)).
$$
It follows that
\begin{align*}
\check{\cF}_{W_0}\bigl(\cP(X))&=\check{\cF}_{W_0/W^\sharp(\mathcal B')}\bigl(e^{\rho_X-\rho_Y}\cdot \cF_{W^\sharp(\mathcal B')}(\cP(Y))\bigr)\\&=
\frac{1}{|T|}\check{\cF}_{W_0/W^\sharp(\mathcal B')}\bigl(e^{\rho_X-\rho_Y}\cdot \cF_{W^\sharp(\mathcal B')T}(\cP(Y))\bigr)\\
&=\frac{1}{|T|}
\check{\cF}_{W_{\fg}}\bigl(\cP(X)\bigr)=\frac{C_\g}{|T|}e^{\rho_X}\check{R}_X.
\end{align*}
This completes the proof.
\end{proof}

\subsubsection{Comments on type $A$}\label{comments}  Note that  if $m\ne n$, formula \eqref{princff} restricts plainly to $A(m,n)$ $=sl(m+1,n+1)$. If instead $m=n$,  the formula  does not restrict to $\h$  when $\lb\gamma\rb =\sum_{i=1}^{n+1}(\d_i-\e_i)$ for $\gamma\in S$.
Note that the factor $\frac{1}{1-e^{-\lb\gamma\rb}}$ 
 is $W_\g$-invariant, hence it can be taken out of the sum.
Since the left hand side of \eqref{princff} restricts to $\h$,  the sum $$\check\cF_{W_\g}\bigl(\frac{e^\rho}
{
\prod\limits_{\beta\in S\setminus\{\gamma\}}(1-e^{-\lb\beta\rb })
}\bigr)$$  is divisible by $1-e^{-\lb\gamma\rb}$. After simplifying, we may restrict to the Cartan subalgebra of $A(n,n)$ getting a superdenominator formula  in this type too.

\subsection{Proof of Proposition \ref{mm}} Recall from \eqref{parentesi} Êthe definition of $\rrbracket \gamma\llbracket $;  note that it may combinatorially rewritten as $\rrbracket \gamma\llbracket=\sum_{\beta\in\gamma^<}\ \sn\beta\cdot\sn\gamma\cdot\beta$
 and that 
$\rrbracket\gamma\llbracket=0$ if $\gamma\in \Pi$.
Let
$$\cQ(X)=\frac{e^{\rho+\sum_{\gamma\in S}\rrbracket \gamma\llbracket }}{\prod_{\beta\in S(X)} (1-e^{-\beta})}.$$

By Theorem KWG, we have ${\check \cF}_{W^{\#}}(\cQ(X))=\check R_X e^{\rho_X}$ if $X$ is a simple arc diagram.
Notice that an odd reflection changes the sign of  $\check R_Xe^{\rho_X}$ and does the same on  $\cQ(X)$.
The interval reflections do not change $R_Xe^{\rho_X}$, since they do not change $\Pi(X)$.
Thus, due to Lemma \ref{esad},  in order to prove~(\ref{mermaid}), it is enough to verify that
the interval reflections do not change
 ${\check \cF}_W(\cQ(X))$. This is done in~\Lem{lem3} below.

\begin{lemma}\label{lem3}
Let $X$ be an arc diagram and $r_{[v,w]}$ be an interval reflection.
Then
$${\check \cF}_{W^{\#}}\bigl(\cQ(X)\bigr)={\check \cF}_{W^{\#}}\bigl(\cQ(r_{[v,w]}(X))\bigr).$$
\end{lemma}
\begin{proof}
Denote by $Y$ the arc subdiagram corresponding to the interval $[v,w]$ and let  $Y'=r_{[v,w]}(Y)$. Then $X'=r_{[v,w]}(X)$ is obtained
from $X$ by substituting $Y$ by $Y'$. View $Y,Y'$ as arc diagrams of $gl(k,k)$-type.
Then $G=W_{[v,w]}\cap W^{\#}$ is $W^{\#}$ constructed for $Y$ and for $Y'$.

Let $e,d$ be  vertices such that $e-d\in  S(X)$. Denote by  $]e,d[$ the interval $[e,d]\setminus\{e,d\}$.
Then $\rrbracket e-d\llbracket =\pm(\sum_{\xi\in ]e,d[\cap\cE}\xi-\sum_{\xi\in ]v,w[\cap\cD}\xi)$ with the sign "$+$"
if $e\in\cE$ and the sign "$-$" if $e\in\cD$. We see that $\rrbracket e-d\llbracket $ is $W_{]e,d[}$-invariant.
In particular, $\rrbracket \gamma\llbracket $ is $G$-invariant for
each $\gamma\in S(X)\setminus S(Y)$, since if $v',w'$ are vertices  such that $v'-w'\in  S(X)\setminus S(Y)$,
then $[v,w]\cap [v',w']=\emptyset$ or $[v,w]\subset[v',w']$ (because the arcs do not intersect).
Notice that $(\rho_X-\rho_Y,\alpha)=0$ for each $\alpha\in\Pi(Y)$, hence
 $\rho_X-\rho_Y$ is $G$-invariant.
Therefore
$${\check \cF}_{W^{\#}}\bigl(\cQ(X)\bigr)={\check \cF}_{W^{\#}/G}\bigl(
\frac{e^{\rho_X-\rho_Y+\sum_{\gamma\in S(X)\setminus S(Y)}\rrbracket \gamma\llbracket }}
{\prod_{\beta\in S(X)\setminus S(Y)}(1-e^{-\beta})}\cdot {\check \cF}_G(\cQ(Y)\bigr)\bigr)$$
 and the similar formula holds for $X',Y'$ respectively. Notice that $S(X)\setminus S(Y)=S(X')\setminus S(Y')$.
 Since $\Pi(Y)=\Pi(Y')$ and $\Pi(X)=\Pi(X')$,
one has $\rho_Y=\rho_{Y'}$, $\rho_X=\rho_{X'}$. We see that the required equality 
${\check \cF}_{W^{\#}}(\cQ(X))={\check \cF}_{W^{\#}}(\cQ(X'))$
follows from the equality
\begin{equation}\label{poi}
{\check \cF}_G(\cQ(Y))={\check \cF}_G(\cQ(Y')),
\end{equation}
which we verify below.

Recall that, by the definition of the interval reflection, the interval $[v,w]$ is of the form
$v=e_1>d_1>e_2>d_2>\ldots>e_s>d_s=w$, where the $e_i$'s are of the same type and
the $d_i$'s are of another type. Then $S(Y)=\{\alpha\}\cup S'$, where $\alpha=v-w=e_1-d_s$ and
$S'=\{d_i-e_{i+1}\}_{i=1}^{s-1}$. Since $S'\subset\Pi(Y)$ one has
$$\sum_{\gamma\in S(Y)}\rrbracket \gamma\llbracket =\rrbracket \alpha\llbracket =\sum_{i=1}^s (e_i-d_i)-\alpha.$$

Consider the simple arc diagram $Z$ with the order $d_1>e_2>d_2>\ldots>d_s>e_1$; then
$S(Z)=S'\cup\{-\alpha\}$ and

$$\rho_Z=\frac{1}{2}\sum_{i=1}^s (e_i-d_i)=-\rho_Y,\ \text{ that is }\ \rho_Z=\rho_Y+\a+\rrbracket \alpha\llbracket .$$
We have
$${\check \cF}_G(\cQ(Y))={\check \cF}_G\bigl(\frac{e^{\rho_Y+\rrbracket \alpha\llbracket }}{\prod_{\beta\in S(Y)}(1-e^{-\beta})}\bigr)=-
{\check \cF}_G\bigl(\frac{e^{\rho_Y+\rrbracket \alpha\llbracket +\alpha}}{\prod_{\beta\in S(Z)}(1-e^{-\beta})}\bigr)=-{\check \cF}_G(\cQ(Z)).$$
Since $Z$ and $Y'$ are simple, one has ${\check \cF}_G(\cQ(Z))=e^{\rho_Z}\check R_Z$ and ${\check \cF}_G(\cQ(Y'))=e^{\rho_{Y'}}\check R_{Y'}$. 
Arguing as in the proof of Corollary \ref{cor1}, one shows that $e^{\rho_Z}\check R_Z=-e^{\rho_{Y'}}\check R_{Y'}$, and this  completes the proof
of \eqref{poi}.\end{proof}
\begin{rem} Arguing as in Lemma \ref{ded}, one deduces \eqref{exmm}Ê from \eqref{mermaid}.
\end{rem}

\section{Distinguished sets of positive roots and compact dual pairs}
\label{5}

\subsection{Dual pairs and Theta correspondence}
Let us now recall what dual pairs and the Theta correspondence are: this involves some basic and well-known facts on 
the oscillator representations of symplectic groups (see e.g. \cite{Adams} for more details and  a rich list of references).\par
 Let $(V, \langle\cdot\,,\cdot\rangle)$ be a $2n$-dimensional real symplectic vector space. Fix a polarization $V=A^+\oplus A^-$ ($A^\pm$ are isotropic subspaces) and a standard symplectic basis
 w.r.t. $\langle\cdot\,,\cdot\rangle$:
 $A^+=\oplus_{i=1}^n \R e_i,\,A^-=\oplus_{i=1}^n \R f_i$, so that $\langle e_i,f_j\rangle =\d_{ij}$. 
 
  Starting from this polarization of $V$ we can
 construct a complex polarization of $V_\C=V\otimes_\R\C$ by setting
 $$V^+_\C=\bigoplus_{i=1}^n\C(e_i+\sqrt{-1}f_i),\quad
 V^-_\C=\bigoplus_{i=1}^n\C(e_i-\sqrt{-1}f_i).$$
This polarization is ``totally complex", i.e. $V^+_\C\cap V=\{0\}$.  As in \eqref{met} we can consider 
 the  representation 
$M=W(V_\C)/W(V_\C)V_\C^+$
of the Weyl algebra $W(V_\C)$ of $(V_\C, \langle\cdot\,,\cdot\rangle_\C)$, and, by means of \eqref{actionsp}, we define an action of $sp(V_\C,\langle\cdot\,,\cdot\rangle_\C)$, on $M$. This representation is usually called the oscillator representation.
The choice of a totally complex polarization is equivalent to assigning a compatible complex structure $J$ on $V$ (i.e., $J\in Sp(V)$ such that $J^2=-1$). Explicitly $J$ is defined by setting $J(e_i)=-f_i$ and $J(f_i)=e_i$. Let $W$ be the space $V$ seen as a complex space via the complex structure $J$.
The elements $g\in Sp(V)$ commuting with $J$ form  a maximal compact subgroup $K$ of $Sp(V)$ and we let $\k$ be its complexified Lie algebra viewed as a subalgebra of $sp(V_\C,\langle\cdot\,,\cdot\rangle_\C)$. Since $K$ commutes with $J$, we may let it act  $\C$-linearly on $W$. We let $\det(k)$ be the determinant of the action of $k\in K$ on $W$. If   $\tilde K$ is the  $\sqrt{\det}$ cover of $K$, then 
$M$ has an action of $\tilde K$ whose differential coincides with the action of $\k$ as a subalgebra of $sp(V_\C,\langle\cdot\,,\cdot\rangle_\C)$.

For future reference we describe explicitly this action. Recall that
$$
\widetilde K=\{(g,z)\in K\times \C^\times\mid z^2=\det g\}.
$$
The covering map is the projection $\pi$ on the first factor. The Lie algebra of $\widetilde K$ is the subalgebra of $\k\times \C$ given by 
$$
\widetilde \k=\{(A,\frac{tr(A)}{2})\mid A\in \k\}.
$$
Note that $d\pi$ is the projection on the first factor and provides an isomorphism between $\widetilde \k$ and $\k$.
We want to define an action of $\widetilde K$ on $M$ in such a way that 
\begin{equation}\label{exp}
\exp(X)\cdot v=e^{d\pi(X)}(v)
\end{equation}
for any $X\in\widetilde \k$. Identify $V_\C^+$ with $W$ by mapping $v\otimes (a+ib)$ to $av+J(bv)$.
 Then $M$ is linearly isomorphic to the polynomial  algebra $P(W)$ on $W$ by identifying $v\in V_\C^-$ with the linear function on $V_\C^+$ given by $v(u)=\langle u,v\rangle$. Recall that $K$ acts on $W$,  hence also on $P(W)$. With this identifications we can define an action of $\widetilde K$ on $M\simeq P(W)$  by 
\begin{equation}\label{actionk}
(g,z)\cdot p=z^{-1}g\cdot p.
\end{equation}
We now check that \eqref{exp} holds. According to our definitions, if $X=(A, \frac{tr(A)}{2})\in \tilde\k$ then $exp(X)\cdot p=e^{-\frac{tr(A)}{2}}e^A\cdot p$. On the other hand,  according to \eqref{actionsp}, $A$ acts on $M$ by left multiplication by $\theta(A)$. Now, applying \eqref{actionandbracket}, we see that
$$
\theta(A)p=[\theta(A),p]+ p\theta(A)\cdot1=A\cdot p + p\theta(A)\cdot1.
$$
Choose a basis $\{x_i\}$ of $V_\C^-$ and let $\{y_i\}$ be the basis of $V^+_\C$ such that $\langle x_i,y_j\rangle=\d_{ij}$. Then $\{x_i,y_i\}$ is a basis of $V_\C$ and $\{y_i,-x_i\}$ is its dual basis. Hence
\begin{align*}
\theta(A)\cdot 1&=\frac{1}{2}\sum_i A(y_i)x_i\cdot 1=\frac{1}{2}\sum_{i,j}\langle x_j, A(y_i)\rangle y_jx_i\cdot 1\\&=-\frac{1}{2}\sum_{i}\langle x_i, A(y_i)\rangle=-\frac{tr(A)}{2}.
\end{align*}
Therefore
$$
\theta(A) p= A\cdot p-\frac{tr(A)}{2}p.
$$
Exponentiating, we find \eqref{exp}.

Let  $H$ be the element of $sp(V_\C,\langle\cdot\,,\cdot\rangle_\C)$ such that $H_{|V_\C^\pm}=\pm I$. Then bracketing with $H$ defines a $\ganz$-gradation $$sp(V_\C,\langle\cdot\,,\cdot\rangle_\C)=\bigoplus_{n\in\ganz}sp(V_\C,\langle\cdot\,,\cdot\rangle_\C)_n.$$ We set $\p=\oplus_{n\ge 0}sp(V_\C,\langle\cdot\,,\cdot\rangle_\C)_n$. Clearly $\p$ is a parabolic subalgebra of\break  $sp(V_\C,\langle\cdot\,,\cdot\rangle_\C)$.

A reductive dual pair is a pair of real Lie subgroups 
$G_1,G_2$ of $Sp(V)$ which act reductively on $V$ and  such that each 
is the centralizer of the other in $Sp(V)$. We say that the dual pair is {\it compact} if one of the two subgroup is compact. In the following 
we deal always with compact dual pairs, assuming $G_1$ compact. We also assume  that $G_1\subset K$ and let
 $\mathfrak{s}_i$ be the Lie algebras of $G_i$ ($i=1,2$). Denote by $\mathfrak{s}_i^{\C}, i=1,2$ their complexifications. Let $\widetilde G_1$ be the lift of $G_1$ to $\widetilde K$. Since $G_1\subset K$ it follows that the center of $K$ is contained in $G_2$. The center of $K$ is $\{exp(\sqrt{-1}tH)\mid t\in\R\}$, thus we can conclude that $H\in \mathfrak{s}_2^\C$. It follows that $\p_2=\p\cap\mathfrak{s}_2^{\C}$ is a parabolic subalgebra of $\mathfrak{s}_2^{\C}$. 
Howe duality in this setting gives the following result (see also \cite{KV}):
\begin{theorem}  \cite{howeclassical}\label{howe}
There is a set $\Sigma$ of irreducible finite-dimensional representations of $\widetilde G_1$ such that, as $\widetilde G_1\times \mathfrak{s}_2^{\C}$-module,
$$
M=\bigoplus_{\eta\in\Sigma} \eta \otimes \tau(\eta),
$$
where $\tau(\eta)$ are irreducible quotients of $\p_2$-parabolic Verma modules.
\end{theorem}
\noindent The map 
$\eta\mapsto \tau(\eta)$ is called the Theta correspondence.
\vskip5pt
We need also to recall the following structure theorem (cf. \cite{howee}).
\begin{prop}\label{structuredual}\
\begin{enumerate}
\item A compact dual pair $(G_1,G_2)$ is of type I, i.e., $G_1G_2$ acts irreducibly on $V$.
\item A reductive dual pair $(G_1,G_2)$ is of type I if and only if there exists a division algebra $D$ over $\R$ with involution $e$, an Hermitian right $D$-vector space $(W_1,(\cdot,\cdot)_1)$, a skew-Hermitian left $D$-vector space $(W_2,(\cdot,\cdot)_2)$ and an isomorphism $V\cong W_1\otimes_DW_2$ of $\R$-vector spaces such 
the symplectic form $\langle\cdot,\cdot\rangle$ corresponds to $Tr_{D/\R}( (\cdot,\cdot)_1\otimes e\circ (\cdot,\cdot)_2)$ and 
under which $G_1$ and $G_2$ map to the isometry groups $U(G_1,(\cdot,\cdot)_1)$, $U(G_2,(\cdot,\cdot)_2)$, respectively.
\end{enumerate}
\end{prop}
\subsection{Cartan involutions}Suppose that $\g$ is a Lie superalgebra of basic classical type. If $\g$ is simple of type $A(1,1)$ let $\D$ be the set of roots of $gl(2,2)$. In all other cases we let $\D$ be, as usual, the set of roots of $\g$. Choose a set $\Dp$ of positive roots and let
 $\Pi=\{\a_1,\dots,\a_n\}$ be the corresponding set of simple roots. If $\g$ is not  of type $A(1,1)$ then the root spaces have dimension $1$, so we can choose for each $\a\in\Dp$ root vectors $X_\a\in \g_\a$ and $X_{-\a}\in \g_{-\a}$ with the property that $(X_\a,X_{-\a})=1$, and let $h_\a=[X_\a,X_{-\a}]$. We set $e_i=X_{\a_i}$ and $f_i=X_{-\a_i}$. If $\g$ is simple of type $A(1,1)$, then, given $\a\in\D$,  we let $X_\a$ be the projection on $\g$ of the corresponding root vector in $gl(2,2)$.
  
Recall from \cite{Kacsuper} that $\g$ is the minimal $\ganz$-graded Lie superalgebra 
with local part  
$$\bigoplus\limits_{i=1}^n\C f_i\oplus
\h\oplus
\bigoplus\limits_{i=1}^n\C e_i
$$ and relations 
\begin{equation} \label{r1}
[e_i,f_j]=\d_{ij}h_{\a_i},\quad
[h_{\a_i},e_j]=(\a_i,\a_j)e_j,\quad
[h_{\a_i},f_j]=-(\a_i,\a_j)f_j,
\end{equation}
on the local part. From now on we assume that 
$(\a_i,\a_j)\in\R$ for any $i,j$. In particular, if $\g$ is of type $D(2,1,\a)$, we assume that $\a\in\R$.\par
We let $N_{\a,\beta}$ be the structure constants for the chosen basis of root vectors:
$$[X_\a,X_\beta]=N_{\a,\beta}X_{\a+\beta},\quad\text{if $\a,\beta,\a+\beta\in\D$.}$$ Set $\sigma_\a=-1$ if $\a$ is an odd negative root and  $\sigma_\a=1$ otherwise, so that 
$(X_\a,X_{-\a})=\sigma_\a$. We also let  $p(\a)$ be the parity of $\a$: $p(\a)=1$ if $\a$ is odd and $p(\a)=0$ if $\a$ is even.
The following statement is a reformulation of Lemma 3.2 of \cite{YK}.
\begin{lemma}\label{chev}Given $\a,\beta\in \Delta$ such that $\a+\beta\in\D$, let $p,q$ be non-negative integers such that $\beta+i\a \in\Delta \cup \{0\},\,i\in\ganz$ if and only if $-p\le i \le q$.
\begin{enumerate}
\item  If $\a$ is even, then $N_{\a,\beta}N_{-\a,\a+\beta}=\frac{1}{2}q(p+1)(\a,\a)$
\item  If $\a$ is odd and $(\a,\a)\ne 0$, then $$N_{\a,\beta}N_{-\a,\a+\beta}=\begin{cases}-\sigma_\a \frac{1}{2}q(\a,\a)&\text{\rm if $p$ is even}\\
\sigma_\a\frac{1}{2}(p+1)(\a,\a)&\text{\rm if $p$ is odd}\end{cases}.$$
\item  If $\a$ is odd and $(\a,\a)= 0$, then $N_{\a,\beta}N_{-\a,\a+\beta}=\begin{cases}\sigma_\a (\a,\beta)&\text{\rm if $p$ is even}\\
0&\text{\rm if $p$ is odd}\end{cases}$.
\end{enumerate}

\end{lemma}
\begin{proof}
The first statement follows as in the case when $\g$ is a Lie algebra.

For the second statement, let us first assume that $\a$ and $\beta$ are not proportional. Then, as in   \cite[Lemma 3.2]{YK}, we obtain
$$
N_{\a,\beta}N_{-\a,\a+\beta}=\s_\a\sum_{i=0}^p(-1)^i(\beta-i\a)(h_\a).
$$
Since $2\a$ is an even root, we see that the subspaces $$\sum_{-p\le 2i\le q} \C\g_{\beta+2i\a},\quad\sum_{-p\le 2i+1\le q} \C\g_{\beta+(2i+1)\a}$$ are the irreducible components of $\sum_{-p\le i\le q} \C\g_{\beta+i\a}$ viewed as a $\langle X_{2\a}, h_{2\a},X_{-2\a}\rangle$-module. It follows that $\beta(h_\a)+q(\a,\a)= -\beta(h_\a)+p(\a,\a)$ if $p-q$ is even while, if $q-p$ is odd, then  $\beta(h_\a)+q(\a,\a)= -\beta(h_\a)+(p-1)(\a,\a)$   and $\beta(h_\a)+(q-1)(\a,\a)= -\beta(h_\a)+p(\a,\a)$. This implies that $p-q$ is even and $\beta(h_\a)=\frac{p-q}{2}(\a,\a)$. Substituting we find the statement. Suppose now that $\a$ and $\beta$ are proportional. There are only two possibilities: $\beta=\alpha$ or $\beta=-2\alpha$. Both cases follow directly from the Jacobi identity 
$$
[[X_{\a},X_{-\a}],X_{\beta}]=[X_{\a},[X_{-\a},X_{\beta}]]+(-1)^{p(\a)p(\beta)}[X_{\a},X_{\beta}],X_{-\a}].
$$

Finally, if $(\a,\a)=0$, then, as in Lemma 3.2 of \cite{YK}, we obtain
$$
N_{\a,\beta}N_{-\a,\a+\beta}=\s_\a\sum_{i=0}^p(-1)^i(\beta-i\a)(h_\a),
$$
hence the statement follows readily.
\end{proof}

As shown in the proof Lemma 3.3 of \cite{YK}, we have that 
$$
N_{-\a,\a+\beta}=(-1)^{p(\a)}\frac{\s_{\a+\beta}}{\s_\beta}N_{-\beta,-\a}.
$$
Substituting in Lemma \ref{chev}, and using the fact that if $\a,\beta,\a+\beta\in\D$ then $N_{\a,\beta}\ne 0$, we find
\begin{equation}\label{nab1}
N_{\a,\beta}N_{-\beta,-\a}=\frac{\s_\beta}{2\s_{\a+\beta}}q(p+1)(\a,\a)
\end{equation}
if $\a$ is even,
\begin{equation}\label{nab2}
N_{\a,\beta}N_{-\beta,-\a}=\begin{cases} \frac{\s_\a\s_\beta}{2\s_{\a+\beta}}q(\a,\a)&\text{\rm if $p$ is even}\\
\frac{-\s_\a\s_\beta}{2\s_{\a+\beta}}(p+1)(\a,\a)&\text{\rm if $p$ is odd}\end{cases}
\end{equation}
if $\a$ is odd and $(\a,\a)\ne 0$, and
\begin{equation}\label{xcv}
N_{\a,\beta}N_{-\beta,-\a}=\frac{-\s_\a\s_\beta}{\s_{\a+\beta}}(\a,\beta)
\end{equation}
if $\a$ is odd and $(\a,\a)=0$.

Introduce the anti-involution $T$ of $\g$ defined on the local part by
$$
\begin{cases}
T(h)=h&\text{if $h\in\h$},\\
T(e_i)=f_i& \text{if $\a_i$ is even},\\
T(e_i)=\sqrt{-1}f_i&\text{if $\a_i$ is odd}.
\end{cases}
$$
Since it preserves relations \eqref{r1}, it can be extended to $\g$, and $T(\g_\a)=\g_{-\a}$. 
As shown in Lemma 3.8 of \cite{YK}, we can choose $X_\a,X_{-\a}$ in such a way that $(X_\a,X_{-\a})=1$ and $T(X_\a)=\sqrt{-1}^{p(\a)}\s_\a X_{-\a}$, and  we can still assume that $X_{\a_i}=e_i$ and $X_{-\a_i}=f_i$.

It follows that, if $\a,\beta,\a+\beta\in\D$, then 
\begin{align*}
&N_{\a,\beta}\sqrt{-1}^{p(\a+\beta)}\s_{\a+\beta}X_{-\a-\beta}=T([X_\a,X_{\beta}])\\
&=[T(X_\beta),T(X_{\a})]=N_{-\beta,-\a}\sqrt{-1}^{p(\a)+p(\beta)}\s_{\a}\s_{\beta}X_{-\a-\beta},
\end{align*}
hence
\begin{equation}\label{nab}
N_{\a,\beta}=(-1)^{p(\a)p(\beta)}\frac{\s_{\a}\s_{\beta}}{\s_{\a+\beta}}N_{-\beta,-\a}=-\frac{\s_{\a}\s_{\beta}}{\s_{\a+\beta}}N_{-\a,-\beta}.
\end{equation}
Combining \eqref{nab} with \eqref{nab1}, \eqref{nab2},
 \eqref{xcv},  it follows that 
\begin{equation}\label{strucosteven}
N_{\a,\beta}^2=\frac{1}{2}q(p+1)(\a,\a)
\end{equation}
if $\a$ is even,
\begin{equation}\label{strucostodd}
N_{\a,\beta}^2=\begin{cases} (-1)^{p(\beta)}\frac{q(\a,\a)}{2}&\text{\rm if $p$ is even}\\
-(-1)^{p(\beta)}\frac{(p+1)(\a,\a)}{2}&\text{\rm if $p$ is odd}\end{cases}
\end{equation}
if $\a$ is odd and $(\a,\a)\ne 0$, and
\begin{equation}\label{strutcostiso}
N_{\a,\beta}^2=-(-1)^{p(\beta)}(\a,\beta).
\end{equation}
if $\a$ is odd and $(\a,\a)=0$. Thus $N_{\a,\beta}$ is either real or pure imaginary.

If $(\a,\a)\ne 0$, let $\e_\a=sgn(\a,\a)$. Let 
$$\xi_\a=\begin{cases}\epsilon_\a \quad&\text{if $p(\a)=0$,} \\ 1\quad &\text{if $p(\a)=1$.}\end{cases}$$
\begin{lemma} 
\begin{equation}\label{congovvia}
\ov N_{\a,\beta}=\frac{\xi_{\a}\xi_{\beta}}{\xi_{\a+\beta}}N_{\a,\beta}.
\end{equation}
\end{lemma}
\begin{proof}
Since $N_{\alpha,\beta}=\pm N_{\beta,\alpha}$ it is enough consider the following three cases:
\begin{enumerate}
\item $\alpha$  is even;
\item $\alpha$ is odd non-isotropic and $\beta$ is odd;
\item $\alpha$ and $\beta$ are isotropic.
\end{enumerate}
In case (1), by \eqref{strucosteven}, we have $\ov N_{\a,\beta}=\xi_\a N_{\a,\beta}$. Thus the result follows obviously if $\beta$ is odd. If $\beta$ is even, then observe that \eqref{strucosteven} implies that $N_{\a,\beta}\ne 0$ if and only if $\e_\a=\e_\beta=\e_{\a+\beta}$. This observation implies the result.\par
In case (2),  by Lemma \ref{chev} (2),
the product $N_{\alpha,\beta} N_{-\alpha,\alpha+\beta}$ is real; by  (1), $N_{-\alpha,\alpha+\beta}$
is real iff $\xi_{\alpha+\beta}=1$,  so
$N_{\alpha,\beta}$ is real iff $\xi_{\alpha+\beta}=1$ as required.\par
In case (3), we have  $2(\a,\beta)=(\a+\beta,\a+\beta)$, so, by \eqref{strutcostiso}, $\ov N_{\a,\beta}=\xi_{\a+\beta}N_{\a,\beta}$ as desired.\par 

\end{proof}

\vskip5pt
In the above setting, given complex numbers $\l_1,\ldots,\l_n$ such that $\l_i\in \sqrt{-1}\R$ if $p(\a_i)=1$ and $\l_i\in \R$ if $p(\a_i)=0$,
one can define an antilinear involution $\omega:\g\to\g$ setting
\begin{equation}\label{omega}
\omega(e_i)=\l_if_i,\quad\omega(f_i)={\bar\l_i}^{-1}e_i,\quad
\omega(h_{\a_i})=-h_{\a_i},\ 1\leq i\leq n.
\end{equation}
Since $\omega$ preserves  relations \eqref{r1}, it follows that it extends to $\g$. 

Clearly $\omega(X_\a)$ is a multiple of  $X_{-\a}$. We define
$\l_\a$ for each $\a\in\D$ by
\begin{equation}\label{ox}
\omega(X_\a)=-\s_\a\xi_\a\l_\a X_{-\a}.
\end{equation}
 Hence, if $\a,\beta,\a+\beta\in\D$, since $\omega$ is an antilinear automorphism, we have 
$$
-\s_{\a+\beta}\xi_{\a+\beta}\l_{\a+\beta}\bar N_{\a,\beta}=\s_{\a}\s_\beta\xi_\a\xi_\beta\l_\a\l_\beta N_{-\a,-\beta}.$$
Using \eqref{nab} and \eqref{congovvia}, we deduce that 
\begin{equation}\label{lab}
\l_{\a+\beta}=\l_\a\l_\beta,\quad  \l_\a\l_{-\a}=1.
\end{equation}
 It follows that, if $\a=\sum_{i=1}^{n}n_i\a_i$, then
\begin{equation} \label{lambdaalfa}
\l_\a=\prod_i(-\xi_{\a_i}\l_i)^{n_i}.
\end{equation}
\vskip5pt
Endow $\g$ with the $\ganz$-grading
\begin{equation}\label{grading}\g=\bigoplus_{i\in\ganz} \mathfrak q_i\end{equation}
which assigns   degree $0$ to $h\in\h$ and to $e_i$ and $f_i$ if $\a_i$ is even,
and degree $1$ to $e_i$ and degree $-1$ to $f_i$, if $\a_i$ is odd. 

Let $\pi=\{i\mid 1\leq i\leq n, p(\a_i)=1\},\,\pi^c=\{1,\ldots,n\}\setminus \pi$.
\begin{prop}\label{cf}\phantom{z}\
\begin{enumerate}
\item The set $\g_0^\omega$ of $\omega$-fixed points in $\g_0$  is a real form of $\g_0$.
\item $\q_0^\omega$ is a compact form of $\q_0$ if and only  if 
$\l_i(\a_i,\a_i)<0$ for all $i\in \pi^c$.
\item If  $\l_i(\a_i,\a_i)<0$ for all $i\in\pi^c$ and $\sqrt{-1}\l_i$ have the same sign for all $i\in \pi$,  then,   for any
 positive integer $r$, $(X_\a,\omega(X_\a))(\a,\a)<0$ if $X_\a\in\q_{4r}\oplus \q_{-4r}$ and $(X_\a,\omega(X_\a))(\a,\a)>0$ if 
$X_\a\in\q_{4r-2}\oplus \q_{-4r+2}$.
\end{enumerate}
\end{prop}
\begin{proof} The proof of the first assertion is standard.
 For the latter two claims observe that, by \eqref{ox}, 
\begin{equation}\label{primaa}(X_\a,\omega(X_\a))=-\s_\a\xi_\a\l_\a.\end{equation}
If $\a$ is an even root, the r.h.s. of \eqref{primaa} is a real number whose sign does not depend on the choice of the basis of root vectors $X_\a$. In fact, if $\{X'_\a\}$ is another basis with 
 $[X'_\a,X'_{-\a}]=\s_\a h_\a$, then $X'_\a=c_\a X_\a$ for suitable complex numbers $c_\a$ and  
 $(X'_\a,\omega(X'_\a))=-|c_\a|^2\s_\a\xi_\a\l_\a$. Therefore we can use formula \eqref{lambdaalfa} to prove the statements
 in a straightforward way. As for (2), recall that a real Lie algebra is compact if the form $(\cdot,\omega(\cdot))$ has signature opposite to that of the invariant form 
 $(\cdot,\cdot)$. If $X_\a\in \mathfrak q_0$ then $\l_\a=\prod (-\xi_{\a_i}\l_i)^{n_i}$. Therefore
 $-\xi_\a\s_\a\l_\a=-\xi_\a\prod (-\xi_{\a_i}\l_i)^{n_i}>0$ for any $\a$ if and only if $\l_i(\a_i,\a_i)<0$ for all $i\in\pi^c$. Claim (3) is proved in a similar way.
  \end{proof}
\medskip
Let 
$$V=\text{fixed point set of $\omega_{|\g_1}$.}$$
 Clearly $V$ is a real form of $\g_1$. Since a basis of $V$ is given by 
\begin{equation}\label{basesimp}
\{X_{\a}+\omega( X_{\a})\mid X_\a\in \mathfrak q_i,\, i\in 2\ganz+1\},
\end{equation}
 it is checked easily that $\langle\cdot\,,\cdot\rangle=(\cdot\,,\cdot)_{|V}$ is a real nondegenerate symplectic form.
Since $ad_{|\g_1}((\g_0)_{\R})\subset ad_{|\g_1}(\g_0)\cap sp(V)$, we have that
$$
ad_{|\g_1}((\g_0)_{\R})=ad_{|\g_1}(\g_0)\cap sp(V)
$$
and that $ad_{|\g_1}(\g_0)\cap sp(V)$ is a real form of $ad_{|\g_1}(\g_0)$.
\vskip5pt
\subsection{Dynkin diagrams}
We will briefly recall the usual encoding of sets of positive roots by means of Dynkin diagrams.
In the following $N=n+m+1$ and $\bigcirc$, $\otimes$, $\square$ correspond respectively to even, isotropic and nonisotropic (both even and odd) roots. 
A $\cdot$ is a placeholder for even or isotropic roots. The possible diagrams in each  type 
 are 
listed in \cite{Kacsuper}. We reproduce this list below.
\subsubsection{Type $gl(m+1,n+1)$}
\begin{equation}\label{diaA}\entrymodifiers={! ! <0pt, .7ex>+}
\xymatrix@!=7ex{
\mathop\cdot\limits_{\alpha_1} \ar@{-}[r]&\mathop\cdot\limits_{\a_2} \ar@{-}[r]& \dots \ar@{-}[r] &\mathop\cdot\limits_{\a_N}\ar@{-}[r]&\mathop\cdot\limits_{\a_{N+1}}
}\end{equation}

\subsubsection{Types $B(m+1,n+1)$ and $D(m+1,n+1)$}\label{BD}
The possible diagrams in these types 
 are 
\begin{equation}\label{diaB}\entrymodifiers={! ! <0pt, .7ex>+}
\xymatrix@!=7ex{
\mathop\cdot\limits_{\alpha_1} \ar@{-}[r]&\mathop\cdot\limits_{\a_2} \ar@{-}[r]& \dots \ar@{-}[r] &\mathop\cdot\limits_{\a_N}\ar@{=>}[r]&\mathop\square\limits_{\a_{N+1}}
}\end{equation}
for type $B(m+1,n+1)$,
and
\begin{equation}\label{diaD1}\entrymodifiers={! ! <0pt, .7ex>+}
\xymatrix@!=7ex{
\mathop\cdot\limits_{\alpha_1} \ar@{-}[r]&\mathop\cdot\limits_{\a_2} \ar@{-}[r]& \dots \ar@{-}[r] &\mathop\bigcirc\limits_{\a_N}\ar@{<=}[r]&\mathop\bigcirc\limits_{\a_{N+1}}}
\end{equation}
\begin{equation}\label{diaD2}\entrymodifiers={! ! <0pt, .7ex>+}
\xymatrix@!=7ex{
\mathop\cdot\limits_{\alpha_1} \ar@{-}[r]&\mathop\cdot\limits_{\a_2} \ar@{-}[r]& \dots \ar@{-}[r] &\mathop\cdot\limits_{\a_{N-1}}\ar@{<-}[r]&\mathop\otimes\limits_{\a_{N}}\ar@{->}[r]&\mathop\bigcirc\limits_{\a_{N+1}}
}\end{equation}
\begin{equation}\label{diaD3}\entrymodifiers={! ! <0pt, .7ex>+}
\xymatrix@!=7ex{
 & & &  & \mathop\bigcirc\limits_{\alpha_{N+1}}\\
\mathop\cdot\limits_{\alpha_1} \ar@{-}[r]&\mathop\cdot\limits_{\a_2} \ar@{-}[r]& \dots \ar@{-}[r] &{\mathop\cdot\limits_{\a_{N-1}}}\ar@{->}[ur]\ar@{->}[dr]\\
& & &  & \mathop\bigcirc\limits_{\alpha_{N}}
}\end{equation}
\begin{equation}\label{diaD4}\entrymodifiers={! ! <0pt, .7ex>+}
\xymatrix@!=7ex{
 & & &  & \mathop\otimes\limits_{\alpha_{N+1}}\\
\mathop\cdot\limits_{\alpha_1} \ar@{-}[r]&\mathop\cdot\limits_{\a_2} \ar@{-}[r]& \dots \ar@{-}[r] &\mathop\cdot\limits_{\a_{N-1}}\ar@{-}[ur]\ar@{-}[dr]\\
& & &  & \mathop\otimes\limits_{\alpha_{N}\ar@{-}[uu]}
}\end{equation}
for $D(m+1,n+1)$.

\subsection{Distinguished sets of positive roots}\label{classifdis}
\begin{defi}\label{dis} We say that a set of positive roots $\Dp$ is  distinguished if $\mathfrak q_i=\{0\}$ for $|i|>2$ in grading \eqref{grading}.  We say that a set of positive roots $\Dp$ is very distinguished if the corresponding set of simple roots contains a unique odd root.
 \end{defi}
Since $\max\{a_j\mid \sum_{i=1}^na_i\a_i\in\Dp,\,p(\a_j)=1\}=2$  for any choice of $\Dp$, we have that  if $\Dp$ is very distinguished, then it is distinguished. The  sets of positive roots which  in 
\cite{Kacsuperp} and in most of the subsequent literature are called distinguished are all ``very distinguished''. We  have preferred to change the terminology because of the relationships with  the Theta correspondence.
\par 
We now discuss the possible distinguished subsets of positive roots up to $W_\g$-equivalence. The classification is basically a case by case inspection, looking at all 
possible diagrams (cf. \cite{Kacsuper}  and \cite{vdl} for the exceptional types) and calculating the coefficients which express the positive roots as a linear combination
of the simple roots.\par
\subsubsection{$gl(m,n)$} 
Given non negative integers $p,q$ with $p+q=m$, we let $\D_{gl}^{(p,q)}$ be the  set of positive roots for $gl(m,n)$, corresponding to the following set of simple roots
$$\Pi_{gl}^{(p,q)}=\{
\e_1-\e_2,\dots, ,\e_p-\d_1,\d_1-\d_2,\ldots,\d_n-\e_{p+1},\e_{p+1}-\e_{p+2},\ldots,\e_{m-1}-\e_m\}.
$$
This is essentially  the only possibility for  a distinguished set of positive roots (including the possibility $p=0$ or $q=0$, in which case 
$\mathfrak q_i=\{0\}$ for $|i|>1$). The other one is to take  non negative integers $r,s$ such that $r+s=n$ and exchanging $\e$'s with $\d$'s. But this case is equivalent
to the above by exchanging $m$ and $n$, so we
 will not distinguish these two possibilities.
\subsubsection{$B(m,n)$}
There is a unique, up to $W_\g$-action,  distinguished 
set of positive roots $\Dp_B$. The corresponding set of simple roots,  with notation as in \cite{Kacsuper}, is
\begin{align}
&\label{B3}
\Pi_B=\{\d_1-\d_2,\ldots,\d_n-\e_1,\e_1-\e_2,\ldots,\e_{m-1}-\e_m,\e_m\}.\end{align}
We also allow the case $m=0$, in which the set of simple roots is
$\{\d_1-\d_2,\ldots,\d_{n-1}-\d_n,\d_n\}$.
\par
\subsubsection{$D(m,n)$} There are three distinguished sets of positive roots, up to $W_\g$-action :
$\Dp_{D1},\Dp_{D2},\Dp_{D2'}$. The corresponding sets of simple roots are 
\begin{align*}
\Pi_{D1}&=\{\d_1-\d_2,\ldots,\d_n-\e_1,\e_1-\e_2,\ldots,\e_{m-1}-\e_m,\e_{m-1}+\e_m\},\\
\Pi_{D2}&=\{\e_1-\e_2,\ldots,\e_{m-1}-\e_m,\e_m-\d_1,\d_1-\d_2,\ldots,\d_{n-1}-\d_n,2\d_n\},\\
\Pi_{D2'}&=\{\e_1-\e_2,\ldots,\e_{m-1}+\e_m,-\e_m-\d_1,\d_1-\d_2,\ldots,\d_{n-1}-\d_n,2\d_n\}.
\end{align*}
\subsubsection{$C(n+1)$} There are three  distinguished sets of positive roots, up to $W_\g$-action:
$\Dp_{C1},\Dp_{C2}, \Dp_{C2'}.$ The corresponding sets of simple roots are 
\begin{align*}
\Pi_{C1}&=\{\d_1-\d_2,\ldots,\d_{n-1}-\d_n,\d_n-\e,\d_n+\e\},\\
\Pi_{C2}&=\{\e-\d_1,\d_1-\d_2,\ldots,\d_{n-1}-\d_n,2\d_n\},\\
\Pi_{C2'}&=\{-\e-\d_1,\d_1-\d_2,\ldots,\d_{n-1}-\d_n,2\d_n\}.\\
\end{align*}
\subsubsection{Exceptional types} A direct inspection shows that the distinguished sets of positive roots correspond to the following diagrams

\begin{equation}\label{d21}\begin{Dynkin}
\Dbloc{}\Dbloc{\Dcirc\Dsouthwest}
\Dskip
\Dbloc{\Dcross\Dnortheast\Dsoutheast}
\Dskip
\Dbloc{}\Dbloc{\Dcirc\Dnorthwest}
\end{Dynkin} 
\qquad
\begin{Dynkin}
\Dbloc{\Dcross\Deast}
\Dbloc{\Dcirc\Dwest\Dtripleeast}
\Dleftarrow
\Dbloc{\Dcirc\Dtriplewest}
\end{Dynkin}
\end{equation}
for $D(2,1,\a)$ (left display; it will be denoted by  $\Dp_{D(2,1,\a)}$), $G(3)$ (right display; it will be denoted by $\Dp_G$), and to 
\begin{equation}\label{tipoF}
\begin{Dynkin}
\Dbloc{\Dcross\Deast}
\Dbloc{\Dcirc\Dwest\Ddoubleeast}
\Dleftarrow
\Dbloc{\Dcirc\Ddoublewest\Deast}
\Dbloc{\Dcirc\Dwest}
\end{Dynkin}
\qquad
\begin{Dynkin}
\Dbloc{\Dcirc\Dtripleeast}
\Drightarrow
\Dbloc{\Dcross\Deast\Dtriplewest}
\Dbloc{\Dcirc\Dwest\Ddoubleeast}
\Dleftarrow
\Dbloc{\Dcirc\Ddoublewest}
\end{Dynkin}
\end{equation}
for $F(4)$ (they will be denoted by $\Dp_{F1},\,\Dp_{F2}$, respectively).

 \subsection{Distinguished sets of positive roots and real forms} \label{distinguishedrealforms}
If $\Dp$ is distinguished we choose $\omega$ corresponding to $\l_i=-\epsilon_{\a_i}$ if $\a_i$ is an even simple root  and 
$\l_i=\sqrt{-1}$ if $\a_i$ is an odd simple root. Set,  for any 
$\a\in\Dp_1$,
$$e_\a=\frac{1}{\sqrt{2}}(X_\a-\sqrt{-1}X_{-\a}),\quad
f_\a=\frac{1}{\sqrt{2}}(X_{-\a}-\sqrt{-1}X_{\a}).$$
Since $\g_1=\mathfrak q_1+\mathfrak q_{-1}$, applying formula \eqref{basesimp}  with our choice of $\omega$, we get  that these vectors form a standard symplectic basis of $V$.
Since $$\bigoplus_{\a\in\D_1^+}\C(e_\a\pm\sqrt{-1}f_\a)=\g_1^\pm,$$
 the oscillator representation is exactly $M^{\Dp}(\g_1)$ as a  $sp(\g_1)$-module.

\par
 Observe now that for any Lie superalgebra of basic classical type
we have 
\begin{equation}\label{decg12}\g_0=\g_0^1\times \g_0^2\end{equation} with $\g_0^1, \g_0^2$ reductive Lie algebras.  
Set
\begin{equation}\label{si}\mathfrak s_i=ad_{|\g_1}(\g_0^i)\cap sp(V),\,\,\,i=1,2.\end{equation}

A case by case check in the distinguished sets of positive roots   shows that we can always choose  $\g_0^1,\,\g_0^2$ so that:
$$\g_0^1\subset \mathfrak q_0,\quad \g_0^2=(\mathfrak q_0\cap \g_0^2)\oplus(\mathfrak q_{-2}+\mathfrak q_2).
$$
In the case of $\g=gl(m,n)$, we choose $\g_0^1=gl(m),\,\g_0^2=gl(n)$ or the other way around.
Combining  these observations with Lemma \ref{cf} we readily obtain the following result.
 \begin{prop} If $\Dp$ is any  distinguished set of positive roots of a Lie superalgebra of basic classical type,  then $\mathfrak s_i$  is a real form of $ad_{|\g_1}(\g_0^i)$. Moreover $\mathfrak s_1$ is a compact form  of $ad_{|\g_1}(\g^1_0)$ and there is a Cartan decomposition $\mathfrak k\oplus \mathfrak r$ of $\mathfrak s_2$ such that $\mathfrak k^\C=ad_{|\g_1}(\g_0^2\cap \mathfrak q_0)$ and $\r^\C=ad_{|\g_1}(\mathfrak q_2\oplus\mathfrak q_{-2})$.
 
 In particular, if $\mathfrak q_2=\mathfrak q_{-2}=\{0\}$ then $\mathfrak s_1\times \mathfrak s_2$ is a compact form of $ad_{|\g_1}(\mathfrak q_0)=ad_{|\g_1}(\g_0)$.
  \end{prop}
\subsection{Detour: classification of real forms}\label{realforms} We now deepen some aspects  of the theory developed in the last two Subsections, in order to obtain an explicit construction of  all real forms of  Lie superalgebras of basic classical type,
  which have been classified in \cite{Kacsuper}, \cite{parker} and \cite{serganova}. In this section when referring to type $A(m,n)$ we mean either $\g=gl(m,n)$ or the simple superalgebra of type $A(m,n)$.\par
  
  Let $\g_\R$ be a real form of $\g$ and let $\omega$ be the corresponding complex conjugation. Then $\omega_{|\g_0}$ is an antilinear involution of the Lie algebra $\g_0$, hence there is a corresponding Cartan decomposition $\g_0=\k\oplus\p$, with Cartan involution $\zeta_\omega$.  We will describe explicitly the involution $\zeta_{\omega}$.
  
    Let $K_0$ be a compact Lie group having $\k^\omega$ as Lie algebra and let $\tr$ be the Lie algebra of a torus in $K_0$. Then  $\t=\tr\otimes_\R\C$ is a Cartan subalgebra of $\k$. Let $\h$ be the centralizer in $\g_0$ of $\t$. It is clear that both $\t$ and $\h$ are $\omega$-stable. Since $\h$ is $\zeta_\omega$-stable, we have 
 $\h=(\h\cap\k)\oplus(\h\cap\p)$; since $\omega\zeta_\omega=\zeta_\omega\omega$, we have  $\omega(\h\cap\p)\subseteq\h\cap\p$. So $\h\cap\p=(\h\cap\p)^\omega + \sqrt{-1}(\h\cap\p)^\omega$. It is known that 
 $\tr\oplus\sqrt{-1}(\h\cap\p)^\omega$ is a Cartan subalgebra of a compact real form of $\g_0$; therefore roots are purely imaginary-valued on it. So the real span of roots
 is contained in $\sqrt{-1}\tr\oplus(\h\cap\p)^\omega$.
 Now fix a set of positive roots $\Dp$ and let   $\omega_0$  be the involution given by \eqref{omega}  with the following choice of parameters:
 \begin{equation}\label{omega0}
 \l_i=-\e_{\a_i}\text{ for $i\notin\pi$,}\quad  \l_i=\sqrt{-1}\text{ for $i\in\pi$}.
 \end{equation} 
 Then 
 $$(\omega\circ\omega_0)_{|\t}=I,\quad(\omega\circ\omega_0)_{|\h\cap\p}=-I.$$
 In particular,  $\t$ is a Cartan subalgebra of $\g_0^{ \omega\circ\omega_0}$. It follows from \cite[Chapter 8]{Kac} that
 $(\omega\circ \omega_0)_{|\g_0}=\eta_0\circ e^{ad(h)}$
 with $\eta_0$ a diagram automorphism of $\g_0$ and $h\in\t$. Clearly $\eta=\omega\circ \omega_0\circ e^{-ad(h)}$ is an extension of $\eta_0$ to $\g$ hence we can write that
 $$
 \omega=\eta\circ e^{ad(h)}\circ \omega_0,
 $$
 with $\eta$ an automorphism of $\g$ such that $\eta_{|\g_0}$ is a diagram automorphism of $\g_0$. 
 We now list the possible choices for automorphisms $\eta$ of $\g_0$, such that $\eta_{|\g_0}$ is a diagram automorphism of $\g_0$.

  \begin{lemma}\label{out}
 There exists $\eta\in Aut(\g)$ such that $\eta_{|\g_0}$ is a nontrivial diagram automorphism if and only if
 $\g$ is of type $A(m,n)$, $D(m,n)$, $D(2,1,\a)$, $\a\in\{1, (-2)^{\pm1}\}$, and
$\eta_{|\g_0}$ is as follows.
 \begin{enumerate}
 \item If $\g$ is of type $A(n,m)$ with $n\ne m$, then  $\eta_{|\g_0}$ restricts to the nontrivial diagram automorphism  of both $A_n$ and $A_m$.
 \item If $\g$ is of type $A(n,n)$, then $\eta_{|\g_0}$ is either the nontrivial diagram automorphism of both $A_n$ components, or it is the flip automorphism between the two $A_n$ components, or the composition of these two automorphisms.
 \item If $\g$ is of type $D(m,n),\,m>2$, then   $\eta_{|\g_0}$ is the unique diagram automorphism of $\g_0$.
 \item  If $\g$ is of type $D(2,1)\cong D(2,1,\a)$, $\a\in\{1, (-2)^{\pm1} \}$, then $\eta$ is the unique diagram automorphism of  the diagram in the left display of \eqref{d21}.
 \end{enumerate}
 \end{lemma}
 \begin{proof}No $\eta$ as in the statement can exist if $\g_0$ has no nontri\-vial diagram automorphisms. This rules out the cases 
 $F(4),G(3),B(m,n)$ with $(m,n)\ne(2,2)$. The case when $\g$ is of type $B(2,2)$ is also excluded because the flip automorphism of $\g_0\cong B_2\times B_2$
 sends the unique odd simple root to a weight which is not a root.  Let us now deal with type $A(m,n)$. An automorphism of $\g$  of type $A(m,n)$ which restricts to the nontrivial diagram automorphisms of both $A_n$ and $A_m$
is given by $\left(\begin{array}{cc}
A&B\\C&D\end{array}\right)^{-st}=\left(\begin{array}{cc}
-A^t&C^t\\-B^t&-D^t\end{array}\right)$. The flip automorphism $F$ is given by  $F\left(\begin{array}{cc}
A&B\\C&D\end{array}\right)=\left(\begin{array}{cc}
D&C\\B&A\end{array}\right)$ and the composition $F\circ (-st)$ gives the remaining automorphism.  We are left with proving that there is no automorphism
of $\g$  of type $A(m,n)$ which restricts trivially to, say,   $A_n$  and nontrivially to  $A_m$ or it is obtained from this one by composing it with the flip. This is easily checked by observing that these latter automorphisms map the simple odd root to a weight that is not a root. 

If $\g$ is of type $D(m,n)$, an automorphism  that restricts to the unique diagram automorphism of $\g_0$ is $Ad(J)$ with $J=\left(\begin{array}{cc}
I_{m-1,1}&0\\0&I_n\end{array}\right)$, where $I_{m-1,1}=\left(\begin{array}{cc}
I_{m-1}&0\\0&-1\end{array}\right)\in gl(m)$.

We now deal with the case when $\g$ is of type $D(2,1,\a)$. If $\a\in\{1,(-2)^{\pm 1}\}$, then  $\g$ is isomorphic to $D(2,1)$ and claim $Ad(J)$ gives as above the desired automorphism.  We need to check that, if
$\a\notin\{1,(-2)^{\pm 1}\}$, then there is no automorphism $\eta$ of $\g$ which restricts to a non trivial 
diagram automorphism of $\g_0$ and that, if $\a=1$, then $Ad(J)_{|\g_0}$ is the only one. We check this as follows. Recall that 
$\g_0\cong(sl(2,\C))^{\times 3}$, and let $\beta_1,\beta_2,\beta_3$ be the simple roots of
these copies of $sl(2,\C)$ Since the form  $(\eta(\cdot),\eta(\cdot))$ is invariant, then 
there exists a constant $c$ such that $(\eta(\cdot),\eta(\cdot))=c (\cdot,\cdot)$.
If $\eta$ restricts to a diagram automorphism of $\g_0$, then $\eta(\beta_i)=\beta_{\sigma(i)},\,i=1,2,3,$ for a suitable permutation $\sigma$. Hence $(\beta_{\sigma(i)},\beta_{\sigma(i)})=c^{-1}(\beta_i,\beta_i)$. We check directly that this is not possible for a real value of $\a$ different from $1,(-2)^{\pm 1}$ and also that, if $\a=1$, then the only possibility is given by $Ad(J)$. 
 \end{proof} 

The case of $\g$ being the simple Lie superalgebra of type $A(1,1)$ needs special care, so we postpone the discussion of this case to the end of this section. Until then, we are excluding this case from our discussion.

In order to complete the classification we choose a very distinguished  set of positive roots. If $\g$ is of type $D(m,n)$, we choose $\Dp_{D1}$. We wish to compute the action of $\omega$ on the generators $e_i$, $f_i$.
We will often use the following result of Serganova. Let $\d_\l$ be the map on $\g$ defined by setting 
\begin{equation}\label{deltal}
{\d_\l}_{|\g_0}=Id,\quad {\d_\l}_{|\g_1^+}=\l Id,\quad{\d_\l}_{|\g_1^-}=\l^{-1}Id.\end{equation}  If $\g$ is of type $I$, then $\d_\l$ is an automorphism of $\g$ for any $\l\in\C$. If $\g$ is of type $II$, then $\d_\l$ is an automorphism of $\g$ if and only if $\l=\pm1$. By \cite[Lemmas 1 and 2]{SERG}, an automorphism of
$\g$ that restricts to $Id_{\g_0}$  is necessarily $\d_\l$.\par

Since the form $(\eta(\cdot),\eta(\cdot))$ is invariant, we have that there is $s_\eta\in\C$ such that $(\eta(\cdot),\eta(\cdot))=s_\eta(\cdot,\cdot)$. Since $\eta^2=Id$ on $\g_0$, we see that $s_\eta=\pm 1$.
 Let now $i_0$ be the index of the unique odd simple root. Set $\gamma=\eta(\a_{i_0})$.   Then there is $\l\in\C$ such that $\eta(e_{i_0})=\l X_\gamma$ so that $\eta(f_{i_0})=s_\eta\sigma_\gamma \l^{-1} X_{-\gamma}$. If $\eta_{|\g_0}=Id_{|\g_0}$, then we set $\eta_1=Id_{\g}$. If $\eta_{|\g_0}\ne Id_{|\g_0}$ and $\g$ is of type $A(m,n)$, then we set $\eta_1=\eta\circ \d_{\l^{-1}}$, while, if $\g$ is of type $D(m,n)$ we set $\eta_1=Ad(J)$ with $J$ as in the proof of Lemma \ref{out}. Then, in all cases $\eta_1(e_{i_0})=X_\gamma$ and we can write
   $$
   \omega=\eta_1 \d_{\l}e^{ad(h)}\omega_0.
   $$
We want to find conditions on $h$ and $\l$ for which  $\omega=\eta_1 \d_{\l}e^{ad(h)}\omega_0$ has order 2.

First assume that $\eta_{|\g_0}=Id$. In this case $\h=\t$ and  $\eta=\d_\l$ for some  $\l\in\C$. We see that $\omega(e_{i_0})=\sqrt{-1}\l^{-1}e^{-\a_{i_0}(h)}f_{i_0}$, 
$\omega(e_i)=-\e_{\a_i} e^{-\a_i(h)}f_i$ if $i\ne i_0$, $\omega(f_{i_0})=\sqrt{-1}\l e^{\a_{i_0}(h)}e_{i_0}$, 
$\omega(e_i)=-\e_{\a_i} e^{-\a_i(h)}f_i$ if $i\ne i_0$. Since $\omega$ is an involution we obtain that $\l e^{\a_{i_0}(h)}\in\R$ and  $e^{\a_{i}(h)}\in\R$ if $i\ne i_0$, hence setting $\l_{i_0}=\sqrt{-1}\l^{-1} e^{-\a_{i_0}(h)}$ and $\l_i=-\e_{\a_i}e^{-\a_i(h)}$ we see that $\omega$ has the form of \eqref{omega}.

 Assume now that $\eta_{|\g_0}\ne Id$. Recall that $\eta$ acts on $\g_0$ as a diagram automorphism. Since, by our choice of the set of positive roots, the simple roots of $\g_0$ are simple in $\g$ if $\g$ is of type $A(m,n)$ while, in type $D(m,n)$, $\eta_1$ is a diagram automorphism of $\g$,  $\eta_1$ induces in both cases  a transposition $i\mapsto i'$ on the indices corresponding to the simple even roots.  
 If $i\notin\pi$, then 
  $$\omega(e_i)=\eta e^{ad(h)}  \omega_0(e_i)=-\e_{\a_i}\eta (e^{ad(h)} f_i)=-\e_{\a_i}e^{-\a_i(h)}\eta(f_{i})=-\e_{\a_i}e^{-\a_i(h)}s_{\eta}f_{i'}.$$
  Hence
  $$\omega^2(e_i)=\overline{e^{-\a_i(h)}}e^{\a_{i'}(h)}e_i.$$
  Since $h\in\t$ we have $\a_{i'}(h)=\a_i(h)$, therefore we get $e^{\a_i(h)}\in\R$.  
  
  We distinguish the following cases.
  \begin{enumerate}
    \item
   If $\g$ is of type $D(m,n)$,  then $\eta(\a_{i_0})=\a_{i_0}$ and $s_{\eta_1}=1$, hence
   $$
   \omega^2(e_{i_0})= e^{\a_{i_0}(h)}\overline{e^{-\a_{i_0}(h)}}e_{i_0}.
   $$ 
Therefore  $e^{\a_{i_0}(h)}\in \R$.
\item
 If $\g$ is of type $A(m,n)$ and $(\eta_1)_{|\g_0}=(-st)_{|\g_0}$,  then $\s_\gamma=-1$ and $(\eta_1)=(-st)\d_\mu$ for some $\mu$. Hence $(\eta_1)^2=(-st)^2$, and  we see that $(\eta_1)^2_{|\g_1}=-Id$. Moreover $s_{\eta_1}=1$.
   
It follows that
  $$\omega^2(e_{i_0})= | \l|^{-2} e^{\gamma(h)}\overline{e^{-\a_{i_0}(h)}}e_{i_0}.
  $$
 Since $h\in \t$, we have $\gamma(h)=\a_{i_0}(h)$, hence $|\l|=1$ and  $e^{\a_{i_0}(h)}\in \R$.
 \item
  If $\g$ is of type $A(n,n)$ and $(\eta_1)_{|\g_0}$ is the flip,  then $(\eta_1)_{|\g_0}=F_{|\g_0}$. It follows that $\s_\gamma=-1$ and $\eta_1=F\d_\mu$ for some $\mu$. Hence $(\eta_1)^2=F^2$, so $(\eta_1)^2_{|\g_1}=Id$. Moreover $s_{\eta_1}=-1$.
  
It follows that
  $$\omega^2(e_{i_0})= | \l|^{-2} e^{\gamma(h)}\overline{e^{-\a_{i_0}(h)}}e_{i_0}.
  $$
 Since $h\in \t$, we have $\gamma(h)=\a_{i_0}(h)$, hence $|\l|=1$ and  $e^{\a_{i_0}(h)}\in \R$.
\item
If $\g$ is of type $A(n,n)$ and $(\eta_1)_{|\g_0}=F(-st)$ then $\gamma=\a_{i_0}$ and  $s_{\eta_1}=-1$.

It follows that
 $$
   \omega^2(e_{i_0})= -\l e^{\a_{i_0}(h)}\overline{\l^{-1}e^{-\a_{i_0}(h)}}e_{i_0},
   $$ 
 hence   $\l e^{\a_{i_0}(h)}\in\sqrt{-1}\R$.
 \end{enumerate}
 

 Let us now return to the general case.
Fix $h$ as above and define
$$\omega_h=\begin{cases}\d_\l e^{ad(h)}\omega_0\quad&\text{if $\eta_{|\g_0}=Id_{\g_0}$}
 \\
e^{ad(h)}\omega_0\quad&\text{in cases (1), (2), (3)}
 \\
  \d_{\sqrt{-1}\l} e^{ad(h)}\omega_0\quad&\text{in case (4).}
\end{cases}$$
As observed above $\omega_h$ is an antilinear involution of the form of \eqref{omega}. 
 
 We are now ready to give an explicit formula for the Cartan involution corresponding to $\omega_{|\g_0}$.  Define
the map $\zeta_{\omega_h}:\g\to\g$ by
\begin{equation}\label{invo}\zeta_{\omega_h}(x)=x,\,\,\,\,x\in\h,\quad\zeta_{\omega_h}(X_\a)=sgn(\l^h_\a)X_\a,\end{equation}
where $\l^h_\a$ is the multiplier defined in \eqref{ox}.
It is easy to check using \eqref{lab} that $\zeta_{\omega_h}$ is an involution of $\g_0$.
We claim that 
$$\zeta_\omega=\eta_1\circ\zeta_{\omega_h}.$$
is the Cartan involution corresponding to $\omega_{|\g_0}$.

First we prove that $sgn(\l^h_\a)=sgn(\l^h_{\eta_1(\a)})$ for $\a\in\Dp_0$.
We prove in fact that $\l^h_\a=\l^h_{\eta_1(\a)}$.  A direct computation shows that   $\l^h_\a=c_\a e^{-\a(h)}$ for any $\a\in\Dp_0$, where
$c_\a=\begin{cases} 1 &\text{if $X_\a\in\mathfrak q_0$}\\-\l^{\mp2}&\text{if $X_\a\in \mathfrak q_{\pm2}$}\end{cases}$. The claim follows since 
our choice of $h$ forces $\a(h)=\eta_1(\a)(h)$.\par

Clearly $\eta_1(X_\a)=d_\a X_{\eta_1(\a)}$. We now prove that $|d_\a|=1$ for any $\a\in\D_0$. For this observe that $d_{-\a}=s_{\eta_1}d^{-1}_{\a}$ and $d_{\eta_1(\a)}=d^{-1}_{\a}$. Hence, since $\omega^2(X_\a)=(\eta_1\omega_h)^2(X_\a)=X_\a$, it follows that
$$
1=\e_\a\e_{-\eta_1(\a)}\l^h_\a\l^h_{-\eta_1(\a)} \bar d_{-\a}d_{\eta_1(\a)}, 
$$
so $1=\e_\a\e_{-\eta_1(\a)}s_{\eta_1}|d_\a|^{-2}=|d_\a|^{-2}$ as wished.

In order to prove that $\zeta_\omega$ is the Cartan involution corresponding to $\omega_{|\g_0}$,  we need only to  check that $\omega$ and $\zeta_\omega$ commute and that $\omega_{|\g_0}\circ \zeta_\omega$ is the conjugation of a  compact real form of ${\g_0}$.
The first claim is a simple computation: 
$$\omega_{|\g_0}\zeta_\omega(X_\a)=-\e_{\eta_1(\a)}sgn(\l^h_\a)\l^h_{\eta_1(\a)}\bar d_{\a}d_{-\eta_1(\a)}X_{-\a}=-s_{\eta_1}\e_{\eta_1(\a)}sgn(\l^h_\a)\l^h_{\a}X_{-\a},$$
while
$$\zeta_\omega\omega_{|\g_0}(X_\a)=-\e_{\a}sgn(\l^h_{-\eta_1(\a)})\l^h_{\a}d_{-\a}d_{-\eta_1(\a)}X_{-\a}=-\e_{\a}sgn(\l^h_\a)\l^h_{\a}X_{-\a},
$$
so the claim follows from the observation that $s_{\eta_1}\e_{\eta_1(\a)}=\e_\a$.

For the second claim it is enough to check that  $\e_\a(X_\a,\omega\zeta_\omega(X_\a))<0$ for any $\a\in\D_0$. Indeed
$$\e_\a(X_\a,\omega\zeta_\omega(X_\a))=-\e_\a sgn(\l^h_\a)\l^h_{\eta_1(\a)}\e_{\eta_1(\a)} s_{\eta_1}= - sgn(\l^h_\a)\l^h_{\a},$$
and this number is always negative.
\vskip 5pt
Recall that we are excluding the case when $\g$ is the simple Lie superalgebra of type $A(1,1)$. It follows from \cite[Prop. 5.3.2]{Kacsuper} that, if $\omega$ and $\omega'$ are antilinear involutions of $\g$ such that $\g_0^\omega$ is isomorphic to $\g_0^{\omega'}$, then $\omega'$ is conjugated by an isomorphism  of $\g$ to either $\omega$ or $\omega\circ \d_{-1}$. Consider now the antilinear involution $\omega_\R$ on $\g$ defined by $\omega_\R(e_i)=e_i$, $\omega_\R(f_i)=f_i$, and $\omega_\R(h)=h$ for $h\in\sqrt{-1}\h^\omega$. Then, since $\omega_h$ is an antilinear involution of the type described in \eqref{omega}, we have that $\omega_\R\omega_h\omega_\R=\omega_h\delta_{-1}$. Moreover, if $\eta_1(\g^+_1)\subset\g_1^+$, we have that $\omega_\R\eta_1=\eta_1\omega_\R$, so $\omega_\R$ provides a real isomorphism between $\g^\omega$ and $\g^{\omega\d_{-1}}$. If instead $\eta_1(\g^+_1)\subset\g_1^-$, then $\g$ is of type $A(m,n)$, and $\d_{\sqrt{-1}}$ provides an isomorphism between $\g^\omega$ and $\g^{\omega\d_{-1}}$. 
It follows that we need only to classify the symmetric pairs $(\g_0,\k)$ up to isomorphism.
First of all we consider cases (3) and (4). In both cases $(\zeta_\omega)_{|\g_0^1}$ gives an isomorphism from $\g_0^1\to \g_0^2$, so the map $(X,Y)\mapsto X+\zeta_\omega(Y)$ gives an isomorphism between $ \g_0^1\times \g_0^1$ and $\g_0$ that maps the diagonal copy of $\g_0^1$ to the set of fixed points of $\zeta_\omega$. Thus the real form $\g_0^\omega$ is isomorphic to $\g_0^1$ seen as a real Lie algebra. 

Let us now turn to the other cases. Choose $t\in \h$ in such a way that $e^{\a_i(t)}=sgn(\l^h_{\a_i})$ for $i\ne i_0$.
We claim that we can choose $t\in \t$. Since $sgn(\l^h_{\a_i}) = sgn(\l^h_{\eta_1(\a_i)})$, it is clear that we can choose $t$ so that $\eta_1(\a_i)(t)=\a_i(t)$ for $i\ne i_0$. It remains to check that we can choose $t$ so that $\a_{i_0}(t)=\eta_1(\a_{i_0})(t)$. This is clear if $\eta_1(\a_{i_0})=\a_{i_0}$, so we are left only with case (3). In this case we know that 
$$e^{\a_{i_0}(h)}= e^{\eta_1(\a_{i_0})(h)}=e^{-\a_{i_0}(h)}\prod_{i\ne i_0}e^{-\a_i(h)}.$$
 This implies that $(e^{\a_{i_0}(h)})^2=\prod_{i\ne i_0}e^{-\a_i(h)}$. It follows that $\prod_{i\ne i_0}e^{-\a_i(h)}>0$ so that
 $$
 e^{\a_{i_0}(t)}=e^{-\a_{i_0}(t)}\prod _{i\ne i_0}sgn(e^{-\a_i(h)})=e^{\eta_1(\a_{i_0})(t)}
 $$
 hence we can choose $t\in\t$  as wished.
 
 If $\g$ is not  of type $A(m,n)$, then the simple roots of $\g$ are linearly independent, so we can further assume that $e^{\a_{i_0}(t)}\in\R$. If $\g$ is  of type $A(m,n)$ and $\eta_{1}=Id_\g$, then we can choose $\l$ so that $\l e^{\a_{i_0}(t)}\in\R$. If, instead, we are  in case (3), then, as shown above, $(e^{\a_{i_0}(t)})^2=1$, so, again, $ e^{\a_{i_0}(t)}\in\R$.
 This implies that we can define an antilinear involution $\omega'=\eta_1\d_\l e^{ad(t)}\omega_0$. 
Observe that $\zeta_{\omega}=\zeta_{\omega'}$ so $\omega$ and $\omega'$ define the same real form. Moreover $\zeta_\omega=\eta_1e^{ad(t')}$ with $t=t'$ if $\g$ is of type $C(n+1)$ or $A(n,m)$ and $t'$ defined by setting $\a_i(t)=\a_i(t')$ for $i\ne i_0$, $\a_{i_0}(t')=\frac{\pi}{2}\sqrt{-1}+\a_{i_0}(t)$ in the other cases.

We will now use Kac classification of Lie algebra  involutions \cite{Kac} to classify $\zeta_{\omega}$. Let $\{\beta^i_1,\dots ,\beta^i_k\}$ be the simple roots  of $(\g_0^i)^{\eta_1}$ corresponding to the  set of positive roots $\Dp_{|\t}$. Recall that $\beta_j^i=\a_{|\t}$ for some simple root of $\g_0^i$ corresponding to $\Dp_0$. Let $X^i_{N}$ be the diagram of $\g_0^i$ and let $\left( X^{i}_{N}\right)^{(r_i)}$ be the corresponding affinization with $r_i=1$ if $(\eta_1)_{|\g_0^i}=Id$ and $r_i=2$ otherwise. Let $a^i_j$ be the labels of $\left( X^{i}_{N}\right)^{(r_i)}$ and set $\theta_i=\sum_{j\ge 1}a^i_j\beta^i_j$. Using the analysis done in \cite[\S3]{MMJ} we find that there is $\tau\in\t$ and $w$ in the Weyl group of $\g_0^{\eta_1}$ such that $2\pi \sqrt{-1}t''=w( t') +\tau$ with $e^{ad(\tau)}=\d_\mu$ and $t''\in\t\cap[\g_0,\g_0]$ having the following property: if $s^i_j=\beta^i_j(t'')$ and $s_0^i=\frac{1}{r_i}-\theta_i(t'')$, then 
  \begin{equation}\label{ordertwo}
  s^i_j\in\{ 0, \frac{1}{2},1\}\text{ and }\sum_{j\ge 0}a^i_js^i_j=\frac{1}{r_i}.
  \end{equation}
  
 It follows that  $e^{2\pi \sqrt{-1}w(\a_i)( t'')}=e^{\a_i(t')}$ for all $i\ne i_0$ and that $\l\mu^{-1} e^{w(\a_{i_0})(2\pi \sqrt{-1} t'')}=\l e^{\a_{i_0}(t')}$. This implies that $e^{2\pi \sqrt{-1}\a_i( t'')}\in\R$ for all $i\ne i_0$. Moreover, in type $A(m,n)$ and $C(n+1)$ with $\eta_1=Id_\g$,  $\l \mu^{-1}e^{\a_{i_0}(2\pi \sqrt{-1} t'')} \in \R$. 
 In type $A(m,n)$ with $\eta_1$ as in case (3), since $\tau\in \t$ we must have $\mu^2=1$ so $e^{\a_{i_0}(2\pi \sqrt{-1} t'')} \in \R$. In the other cases we have $e^{2\pi \sqrt{-1}\a_{i_0}( t'')}\in \sqrt{-1}\R$. In turn, this implies  that  there is an antilinear involution $\omega''$ such that $\zeta_{\omega''}=\eta_1e^{2\pi\sqrt{-1} ad(t'')}$. Since $\zeta_{\omega''}$ is conjugated to $\zeta_\omega$, we see that $\omega$ and $\omega''$ define isomorphic real forms.
 The outcome is the following:\par
(1) If $\g$ is of type $C(n+1)$ or $A(m,n)$ and $\eta_1=Id_\g$, then the list of all $t''\in\h\cap [\g_0,\g_0]$ that satisfy \eqref{ordertwo} and such that there is $\l\in\C$ with $\l e^{\a_{i_0}(2\pi \sqrt{-1} t'')}\in\R$ gives a list of all real forms. Since the latter is an empty condition, we see that the list of all real forms is given by all the $t''\in\h\cap [\g_0,\g_0]$ that satisfy \eqref{ordertwo}.
\par
(2)  If $\g$ is of type $A(m,n)$ and $\eta_1$ is as in case (3), then the list of all  $t''\in\t\cap [\g_0,\g_0]$ that satisfy \eqref{ordertwo} and such that $e^{2\pi \sqrt{-1}\a_{i_0}(t'')}\in\R$ gives a list of all real forms.
\par
(3) In all the remaining cases the list of all  $t''\in\t\cap [\g_0,\g_0]$ that satisfy \eqref{ordertwo} and such that $e^{2\pi \sqrt{-1}\a_{i_0}(t'')}\in\sqrt{-1}\R$ gives the list of all real forms.
 
 The set of $t''\in\t\cap [\g_0,\g_0]$ that satisfy \eqref{ordertwo} is finite, so we need only to check the above conditions on this finite set. 
  The outcome of this computation is given in the following tables. 
 In all cases except $A(m,n)$ and $C(n+1)$ define
 $\varpi_i\in\h$ by $\a_j(\varpi_i)=\d_{ij}$ where the simple roots of $\g$ are enumerated from left to right. If $\g$ is  of type $A(m,n)$ or $C(n+1)$ 
 we define $\varpi_i$ for $i\ne i_0$ to be the unique element in $\h\cap[\g_0,\g_0]$ such that $\a_i(\varpi_j)=\d_{ij}$ for all $i\ne i_0$.
 In the following tables we assume that 
 $\Dp$ is a very distinguished set of positive roots. For types $D,C,F$ we choose $\Dp_{D1},\,\Dp_{C2},\,\Dp_{F1}$, respectively.
 Table I displays the real forms corresponding to antilinear involutions with $\eta_{|\g_0}=Id_{\g_0}$ and Table II covers the remaining cases. 
 In Table I we also list    the parameters  $(\l_1,\ldots,\l_n)$ occurring in \eqref{omega}, setting, only once in the paper, $i=\sqrt{-1}$; the $\l_j$ which are not listed 
 are equal to $1$. 
 \par\newpage
{\scriptsize
 \begin{equation*}\begin{tabular}{ l | l | l | l}
$\g$ &\text{real form of $\g_0$}&$t''$ & $(\l_1,\ldots,\l_n)$ \\\hline
$gl(m,n)$ & $u(p,m-p)\times u(q,n-q)$& $\frac{1}{2}\varpi_p+\frac{1}{2}\varpi_{n+q}$, &$\l_m=i$\\
& & {\tiny $(0\le p\le m-1,1\leq q\leq n)$} &$\l_j=\l_{m+q}=-1$\\
& & &{\tiny $(1\leq j\leq m,j\ne p)$}\\
 $A(m-1,n-1),$& $su(p,m-p)\times su(q,n-q)\times\R$& $\frac{1}{2}\varpi_p+\frac{1}{2}\varpi_{n+q}$,&\text{same as above;}\\
$m\ne n$& & {\tiny $(0\le p\le m-1,1\leq q\leq n)$}\\
 $A(n-1,n-1)$& $su(p,n-p)\times su(q,n-q)$& $\frac{1}{2}\varpi_p+\frac{1}{2}\varpi_{n+q}$,&\text{same as above}\\
& & {\tiny $(0\le p\le n-1,1\leq q\leq n)$}&\text{with $m=n$;}\\
 $B(0,n)$& $sp(2n,\mathbb R)$&$\frac{1}{4}\varpi_n$&$\l_n=i$\\
 & & &$\l_j=-1$\\
 & & &{\tiny $(1\leq j\leq n-1)$}\\
  $B(m,n)$& $so(2m+1)\times sp(2n,\mathbb R)$ &$\frac{1}{4}\varpi_n$&$\l_n=i$\\
 & & &$\l_{n+k}=-1$ \\
 & & &{\tiny $(1\leq k \leq m)$}\\
 $B(m,n)$& $so(2p,2m+1-2p)\times sp(2n,\mathbb R)$ &$-\frac{1}{4}\varpi_n+\frac{1}{2}\varpi_{n+p}$, & $\l_n=-i$\\
 & & {\tiny $(1\le p\le m)$} &$\l_{n+k}=-1$\\
  & & &{\tiny $(1\leq k \leq m,  k\ne p)$}\\
   $D(m,n) $& $so(2m)\times sp(2n,\mathbb R)$& $\frac{1}{4}\varpi_{n}$,&$\l_n=i$\\ 
 & & & $\l_{n+k}=-1$\\
  & & &{\tiny $(1\leq k \leq m)$}\\
 $D(m,n) $& $so(2p,2m-2p)\times sp(2n,\mathbb R)$& $-\frac{1}{4}\varpi_n+\frac{1}{2}\varpi_{n+p}$,& $\l_n=-i$\\
 & & {\tiny $(1\le p\le m-2)$}&$\l_{n+k}=-1$\\
  & & &{\tiny $(1\leq k \leq m, k\ne p)$}\\
$D(m,n)$&$sp(n)\times so^*(2m)$&$\frac{1}{4}\varpi_n+\frac{1}{2}\varpi_{m+n}$&$\l_n=i$
\\ 
& & & $\l_{n+k}=-1$ \\
 & & &{\tiny $(1\leq k \leq m-1)$}\\
$D(m,n)$&$sp(q,n-q)\times so^*(2m)$&$\frac{1}{2}\varpi_q-\frac{1}{4}\varpi_n+\frac{1}{2}\varpi_{m+n}$,& $\l_n=-i$
\\ & &  {\tiny $(1\leq q\le n-1)$} &$\l_q=\l_{n+k}=-1$  \\
 & & &{\tiny $(1\leq k \leq m-1)$}\\
 $C(n+1) $ & $sp(n)\times \R$&$0$&$\l_1=i$\\
 $C(n+1) $ & $ sp(2n,\mathbb R)\times \R$&$\frac{1}{2}\varpi_{n+1}$&$\l_1=i,\l_{n+1}=-1$ \\
$C(n+1)$ &$sp(q,n-q)\times \R$ &$\frac{1}{2}\varpi_{q+1}$,  {\tiny $1\le q\le n-1$} &$\l_1=i,\l_q=-1$ \\
$D(2,1,\a)$&$(sl(2,\R))^{\times 3}$& $-\frac{1}{4}\varpi_1+\frac{1}{2}\varpi_2+\frac{1}{2}\varpi_3$&$(-i,1,1)$\\
$D(2,1,\a)$&$ su(2)\times su(2) \times sl(2,\R)$&$\frac{1}{4}\varpi_1$&$(i,-1,-1)$\\
$F(4)$&$so(7)\times sl(2,\R)$& $\frac{1}{4}\varpi_1$&$(i,-1,-1,-1)$\\
$F(4)$&$su(2)\times so(1,6)$&$\frac{1}{2}\varpi_2-\frac{3}{4}\varpi_1$&$(i,1,-1,-1)$\\
$F(4)$&$su(2)\times so(2,5)$ &$\frac{1}{2}\varpi_4-\frac{1}{4}\varpi_1$&$(-i,-1,1,-1)$\\
$F(4)$&$so(3,4)\times sl(2,\R)$ &$\frac{1}{2}\varpi_3-\frac{1}{4}\varpi_1$&$(-i,-1,-1,1)$\\
$G(3)$&$ \mathcal G_{2,0}\times sl(2,\R)$&$\frac{1}{4}\varpi_1$&$(i,-1)$\\
$G(3)$&$\mathcal G_{2,2}\times sl(2,\R)$ &$-\frac{1}{4}\varpi_1+\frac{1}{2}\varpi_3$&$(-i,1)$\\
 \end{tabular}\end{equation*}
 }
\centerline{{Table I}} 
\vskip5pt
\begin{equation*}\begin{tabular}{ l | l | l | l }
$\g$&$\eta_{|\g_0}$  &\text{real form of $\g_0$}&$t''$\\\hline
$gl(m,n)$&$-st$ & $gl(m,\R))\times gl(n,\R)$&$0$ \\
$A(m,n), m\ne n$&$-st$ & $sl(m,\R))\times sl(n,\R)\times \R$&$0$ \\
$A(n,n)$ &$-st$ & $sl(m,\R))\times sl(n,\R)$&$0$ \\
$gl(2m,2n)$  &$-st$ & $u^*(2m))\times u^*(2n)$&$\frac{1}{2}\varpi_m+\frac{1}{2}\varpi_{2m+n}$\\
$A(2m,2n), m\ne n$&$-st$ & $su^*(2m))\times su^*(2n)\times \R$&$\frac{1}{2}\varpi_m+\frac{1}{2}\varpi_{2m+n}$\\
$A(2n,2n)$ &$-st$ & $su^*(2m))\times su^*(2n)$&$\frac{1}{2}\varpi_m+\frac{1}{2}\varpi_{2m+n}$\\
$gl(n,n)$&$F$ & $gl(n,\C)$&$0$\\
$A(n,n)$&$F$ & $sl(n,\C)$&$0$\\
$D(m,n) $&$Ad(J)$ & $sp(2n,\R)\times so(2m-2p-1,2p+1)$&$-\frac{1}{4}\varpi_n+\frac{1}{2}\varpi_{n+p}$\\
& & &{\tiny $(0\leq p\leq m-2)$}\\
\end{tabular}\end{equation*}
\centerline{{Table II}}
\par\newpage
In Table II, we have not  considered the case when $\g$ is of type $D(2,1,-\frac{1}{2})$ with non trivial $\eta_1$, because $D(2,1,-\frac{1}{2})$ is isomorphic to  $D(2,1)$. The corresponding real form is, from Table II, $sp(2,\R)\times so(3,1)$. The special isomorphism $so(3,1)\cong sl(2,\C)$ allows to recover the classification  as stated in  \cite[Theorem 2.5]{parker}.

\vskip5pt
Suppose now that $\g$ is simple of type $A(1,1)$. It needs special care since in this case there is a large group of
automorphisms acting identically on the even part (this was apparently
missed in  previous discussions on real forms in the literature). First observe that $\g_0\cong sl(2,\C)\times sl(2,\C)$, hence there are four non isomorphic real forms of $\g_0$: $sl(2,\C), sl(2,\R)\times sl(2,\R), su(2)\times sl(2,\R), su(2)\times su(2)$.  All these real forms are obtained from real forms of $\g$ 
by restricting an antilinear involution $\omega$ of $gl(2,2)$ to $sl(2,2)$. Since this restriction clearly stabilizes the center of $sl(2,2)$, it can be pushed down to $A(1,1)$. Hence for any real form of $\g_0$ there is at least one real form of $A(1,1)$ which induces it. The only non-trivial thing to prove is that also in this case the real form of $\g_0$ determines the real form of $\g$. Suppose that $\omega$ and $\omega'$ are antilinear involutions that restrict in the same way to $\g_0$. Then there is an automorphism $g$ of $\g$ such that $g_{|\g_0}=Id_{\g_0}$ and $\omega'=\omega g$. We now identify the group $\mathcal S$ of automorphisms of $\g$ that restrict to the identity of $\g_0$ with $SL(2,\C)$: indeed, if $g$ is such an automorphism, then necessarily 
\begin{align}\label{xsl2}g(X_{\a_2})=aX_{\a_2}+bX_{-\theta},\, g(X_{-\a_2})=cX_{\theta}+dX_{-\a_2}\\\label{xsl3} g(X_{\theta})=aX_{\theta}+bX_{-\a_2},\, g(X_{-\theta})=cX_{\a_2}+dX_{-\theta}.\end{align} The fact that $g$ is an homomorphism implies that $M_g=\left(\begin{array}{cc}
a&b\\c&d\end{array}\right)\in SL(2,\C)$. Hence we have a map 
to $SL(2,\C)$. To prove that this map is bijective, we consider the local Lie superalgebra $G_{-1}\oplus G_0\oplus G_1$ with
$G_1=\C X_{\a_1}\oplus \C X_{\a_2} \oplus \C X_{\a_3}\oplus \C X_{-\theta}$, $G_{-1}=\C X_{-\a_1}\oplus \C X_{-\a_2} \oplus \C X_{-\a_3}\oplus \C X_{\theta}$, $G_0=\h$ and whose bracket is obtained by restricting the bracket of $A(1,1)$. The corresponding minimal $\ganz$-graded Lie superalgebra is $\g$. A direct check shows that the map 
$g$ defined   by \eqref{xsl2}, \eqref{xsl3} and such that $g_{|\g_0}=Id$ defines an automorphism of the local algebra, hence it extends uniquely to an automorphism of $\g$.\par
Suppose now that $\omega=\eta_1 e^{ad(t)}\omega_0$. It is easy to see that for any $g\in \mathcal S$
there exists a unique $g' \in\mathcal  S$ such that
$$M_{g'}= A \overline{M_g} A,
$$
with $A=\left(\begin{array}{cc}
0&1\\1&0\end{array}\right)$ if $\g_0^{\omega}=su(2)\times su(2)$ or $sl(2,\R)\times sl(2,\R)$, with $A=\left(\begin{array}{cc}
0&i\\-i&0\end{array}\right)$ if $\g_0^{\omega}=sl(2,\R)\times su(2)$, and with $A=\left(\begin{array}{cc}
1&0\\0&1\end{array}\right)$ if $\g_0^{\omega}=sl(2,\C)$.
If 
$\omega'=\omega g$, then the condition that $(\omega')^2=Id_\g$ is equivalent to $\omega=g\omega g$, hence we must have
$$A \overline{M_g}A M_g=\left(\begin{array}{cc}
1&0\\0&1\end{array}\right).$$
Moreover we have that $(g')^{-1}\omega gg'=
\omega g''$ with 
$M_{g''}=A\overline{M^{-1}_{g'}}A M_gM_{g'}$.

Consider the antilinear antiinvolution $\s$ on $SL(2,\C)$ defined by $\s(M)=A \overline{M}^{-1} A$. We have just shown that the set of $M_g$ such that $(\omega\circ g)^2=Id$ is the fixed point set $S$ of  $\s$. Let us also consider the action of $SL(2,\C)$ on $S$ by $M\cdot s=\s(M)sM$. If this action has a unique orbit then, as shown above, $\omega$ and $\omega\circ g$ are conjugated, so the real forms $\g^\omega$ and $\g^{\omega'}$ are isomorphic.  To check when this action has a unique orbit we observe that the stabilizer of any point $s\in S$ is the fixed point set of the antilinear involution on $SL(2,\C)$ given by $g\mapsto s^{-1}\s(g)^{-1}s$ hence it is a real form of $SL(2,\C)$. In particular all the orbits have the same dimension so the orbits are the connected components of $S$.
More explicitly 
$$
S=\{\left(\begin{array}{cc}
x&z\\-\bar z&y\end{array}\right) \in SL(2,\C)\mid x,y\in\R,\, z\in\C\}
$$
if $\g_0^\omega=su(2)\times su(2)$ or $\g_0^\omega=sl(2,\R)\times sl(2,\R)$, while
$$
S=\{\left(\begin{array}{cc}
z&\sqrt{-1}x\\\sqrt{-1}y&\bar z\end{array}\right) \in SL(2,\C)\mid x,y\in\R,\, z\in\C\}
$$
if $\g_0^\omega=sl(2,\C)$.
In both cases $S$ is the quadric in $\R^4$ of equation $xy+a^2+b^2=1$ where $z=a+\sqrt{-1}b$, which is homemorphic to $\R\times S^2$. 

It remains to deal with the last case, i.e. $\g_0^{\omega}=sl(2,\R)\times su(2)$,  where
$$
S=\{\left(\begin{array}{cc}
x&z\\\bar z&y\end{array}\right) \in SL(2,\C)\mid x,y\in\R,\, z\in\C\},
$$
which is the quadric of equation $xy-a^2-b^2=1$. This is  homeomorphic  to $\R^3\times S^0$, hence disconnected. However, the two orbits are the orbit of $M_{Id}$ and of $M_{\d_{-1}}$ (cf. \eqref{deltal}), so we need only to check if $\g^\omega$ and $\g^{\omega\delta_{-1}}$ are isomorphic, but the argument used in the other cases shows that $\omega_\R$  provides an isomorphism between $\g^\omega$ and $\g^{\omega\d_{-1}}$. Thus, $A(1,1)$ has no real forms other than those listed in Tables I and II.

\begin{rem} Our treatment doesn't cover the cases when the Cartan matrix is not real,
which happens only if $\g$ is of type $D(2,1,\a)$, where $\a\in \C\setminus\R,
,\a+\bar\a=-1$. In this case one has one extra real form, with
even part $sl(2,\C)\times sl(2,\R)$ (see \cite{parker}).
\end{rem}
\begin{rem} Our approach to classify real forms started by fixing a suitable set 
of positive roots and henceforth the antilinear involution $\omega_0$. All antilinear involutions are then gotten as $\eta_1e^{ad(h)}\omega_0$, by letting $h$ and $\eta_1$ vary.
 
One can, in a different perspective, consider just the  antilinear involutions $\eta_1\omega_0$ and let the choice of $\Dp$ vary. In this way one can associate to any set of positive roots a set of real forms indexed by the outer automorphism $\eta_1$. It can be checked that all real forms can be obtained in this way. The real forms discussed in Section \ref{distinguishedrealforms} are precisely those corresponding to the distinguished sets of positive roots with $\eta_1=Id_\g$. In Section \ref{cinquesei} we will discuss the real forms corresponding to the sets of positive roots of depth $4$ when $\g$ is $gl(m,n)$ or of type $B(m,n)$ and $D(m,n)$. It will actually turn out that, in these cases, all real forms of $\g$ already appear as real forms corresponding to positive sets of roots of depth at most $4$.
\end{rem} 
\begin{rem} It is possible to deal with the problem of classifying antilinear involutions in a slightly different way. Consider marked sets of positive roots. For this we mean the datum $(\Pi,L)$ where $\Pi=\{\a_1,\ldots,\a_n\}$ is the set of simple roots of a set $\Dp$ of positive roots and 
$L=\{l_i,\ldots,l_n\}$ is a set of labels satisfying
$l_i\in\{\pm 1\}$ if $p(\a_i)=0$, $l_i\in\{\pm \sqrt{-1}\}$ if $p(\a_i)=1$. Given a marked set of positive roots, we can define an antilinear involution $\omega_L$ using formulas
\eqref{omega}--\eqref{lambdaalfa} with $\l_i=-\xi_{\a_i}l_i$, so that $\l_{\a_i}=l_i$. We have  a natural  action of odd reflections on marked set of positive roots: if $r_{\a_i}$ denotes the odd reflection w.r.t.
the simple root $\a_i$ (see \eqref{oddref}), we set $r_{\a_i}(\Pi,L)=(r_{\a_i}(\Pi),L')$, with $L'=\{l'_i,\ldots,l'_n\},\,
l'_j=l_j$ if $(\a_i,\a_j)=0,\,i\ne j,\,l'_i=l_i^{-1}$ and $l'_j=l_jl_i$ if $(\a_i,\a_j)\ne0$.
This action has  the following property: if $r_\a(\Pi,L)=(\Pi',L')$, then, since $\zeta_{\omega_L}=\zeta_{\omega_{L'}}$, 
$\omega_L$ and $\omega_{L'}$ define isomorphic real forms. Using this action, one  has a 
combinatorial recipe to detect the real form corresponding to a given marked set of positive roots:
since any two sets of positive roots can be obtained from each other through a sequence of odd reflections, it suffices to consider the marked sets  $(\Pi_{vd},L_{vd})$, where $\Pi_{vd}$ is one of  the sets of simple roots corresponding to a very distinguished set of positive roots we used for the classification. Given any marked set $(\Pi,L)$ one can compute combinatorially the marked set $(\Pi_{vd},L_{vd})$ in the same equivalence class and apply our classification to $\omega_{L_{vd}}$.
\end{rem}

 \vskip5pt
\subsection{Compact dual pairs coming from distinguished sets of positive roots} 
Let $\g=gl(m,n)$ or  a Lie superalgebra of type  $B,C,D$. Fix   a distinguished set $\Dp$  of positive roots. Let $\omega$ be as in Subsection \ref{distinguishedrealforms}
and let $V$ be the set of $\omega$-fixed points in $\g_1$. Also recall from \eqref{si} the definition of $\mathfrak s_1, \mathfrak s_2$.\par In this Subsection we prove the following proposition. 
\begin{prop}\label{B} Let $\g$ be $gl(m,n)$ or  a Lie superalgebra of type  $B,C,D$. If $\Dp$ is a distinguished set of positive roots, then there is  a compact dual pair $(G_1,G_2)$ in $Sp(V,\langle\cdot\,,\cdot\rangle)$ with $Lie(G_i)=\mathfrak s_i,\,i=1,2,$ such that 
 the action of $\g_0$ on $M^{\Dp}(\g_1)$ gives   the Theta correspondence 
for  $(G_1,G_2)$ at the level of Lie algebras.
The compact dual pairs 
$(G_1,G_2)$ are listed  in the following table (in which $m,n$ are positive integers):
\begin{equation}\label{table}\begin{tabular}{l|l|l}
$\g$&$\Dp$  & $(G_1,G_2)$ \\\hline
$gl(m,n)$&$\D_{gl}^{(p,q)}$ & $(U(n), U(p,q))$\\
 $B(0,n)$&$\Dp_B$ & $(O(1), Sp(2n,\mathbb R))$\\
 $B(m,n)$&$\Dp_B$ & $(O(2m+1), Sp(2n,\mathbb R))$\\
$D(m,n) $&$\Dp_{D1}$ & $(O(2m),Sp(2n,\mathbb R))$\\
$D(m,n)$&$\Dp_{D2}$\ &$(Sp(n),SO^*(2m))$ \\
 $D(m,n)$&$\Dp_{D2'}$ &$(Sp(n),SO^*(2m))$\\
 $C(n+1) $&$\Dp_{C1}$ & $(O(2),Sp(2n,\mathbb R))$\\
$C(n+1)$&$\Dp_{C2}$\ &$(Sp(n),SO^*(2))$ \\
$C(n+1)$&$\Dp_{C2'}$\ &$(Sp(n),SO^*(2))$ 
 \end{tabular}\end{equation}
\end{prop}
We shall prove  Proposition \ref{B}  by realizing explicitly the superalgebras as subalgebras of some $gl(r,s)$ and then checking the conditions of Proposition \ref{structuredual} in a case by case fashion.
For shortness we will give the details only in type $D(m,n)$.

\subsubsection{$\Dp_{D1}$} 

\vskip5pt
Endow the superspace $\mathbb C^{2m|2n}$ with the bilinear form given by the matrix $\left(\begin{array}{cc} I_{2m} & 0\\ 0& J_{2n}\end{array}\right),$ where
$J_{2n}=\left(\begin{array}{cc} 0& I_n\\ -I_n& 0\end{array}\right)$.\par The  Lie  superalgebra $osp(2m,2n)$ of type $D(m,n)$ is  the set of linear transformations
which are skewsupersymmetric w.r.t. a supersymmetric nondegenerate bilinear form. If we choose the form as above, then $osp(2m,2n)$ can be realized as the set of matrices
$$\left(\begin{array}{ccc} A & B_1 & B_2\\ -B_2^t & C_1 & C_2 \\ B_1^t & C_3 & -C_1^t\end{array}\right)$$ with $A$ skewsymmetric $2m\times 2m$ and $C_2,C_3$ symmetric $n\times n$ matrices.
Denoting by $E_{i,j}$ the matrix units, set
$$h_i=
\begin{cases}-\sqrt{-1}E_{i,m+i}+\sqrt{-1}E_{m+i,i}, &1\leq i\leq m,\\
-\sqrt{-1}E_{2m+i,2m+n+i}+\sqrt{-1}E_{2m+n+i,2m+i}, &m+1\leq i\leq m+n.\end{cases}
$$
Then $\h=\oplus_{i=1}^{m+n}\C h_i$ is a Cartan subalgebra of $\g$.\par Set $( X,Y )=\frac{1}{2} str(XY)$;
then 
$$( h_i,h_j )=\begin{cases}\d_{ij}\quad&1\leq i\leq m,\\-\d_{ij}\quad& m+1\leq i\leq m+n.\\
\end{cases}$$
In the chosen distinguished set of positive roots
we have $\Dp_1=\{\d_i\pm\e_j\mid 1\leq j\leq m, \, 1\leq i\leq n\}$. We can choose for each $\a\in\Dp_1$ root vectors $X_\a$, $X_{-\a}$ in such a way that  $$V=\left\{\left(\begin{array}{cc}
0&\sqrt{-1}A\\-\sqrt{-1}J_{2n}{}^tA& 0\end{array}\right)\mid A\in
M_{2m,2n}(\R)\right\}.$$
\par
The map $\Psi:\g_1\to M_{2m,2n}(\C),$
\begin{equation}\label{psi}
\Psi(\left(\begin{array}{cc}
0&A\\-J_{2n}{}^tA& 0\end{array}\right))=A
\end{equation} intertwines the adjoint action of $\g_0$ on $\g_1$ with its action  on $M_{2m,2n}(\C)$ given by 
\begin{equation}\label{azioneg0}\left(\begin{array}{cc}
E&0\\0&F \end{array}\right)\cdot A=EA-AF.
\end{equation} 
It is then clear that $ad\left(\begin{array}{cc}
E&0\\0&F \end{array}\right)(V)\subset V$ if and only if $E$ and $F$ are real matrices, i. e. $E\in so(2m)$ and $F\in sp(2n,\R)$.\par 

Consider now the map $\Phi:V\to \R^{2m}\otimes (\R^{2n})^*$ given by
\begin{align*}
&\Phi(\sqrt{-1}(E_{i,2m+j}+E_{2m+n+j,i}))=e_i\otimes f^j\quad && 0\leq i\leq 2m,\, 1\leq j\leq n\\
&\Phi(\sqrt{-1}(E_{i,2m+n+j}-E_{2m+j,i})=e_i\otimes f^{n+j}\quad && 0\leq i\leq 2m,\, 1\leq j\leq n
\end{align*}
where $\{e_i\}$ is the standard basis of $\R^{2m}$ and $\{f^j\}$ is the standard  basis of $(\R^{2n})^*$. Note that $\Phi$ intertwines the action of $ad_{|\g_1}(\g_0)\cap sp(V)$ on $V$ with the standard action of $so(2m)\times sp(2n,\R)$ on  $\R^{2m}\otimes (\R^{2n})^*$.

Let $(\cdot\, ,\cdot)_{2m}$ be the standard symmetric bilinear form   on $\R^{2n}$, $(\cdot\, ,\cdot)_{2n}$ the standard symplectic form on $(\R^{2n})^*$ and let $\langle\cdot\,,\cdot\rangle$ be the tensor product $(\cdot\, ,\cdot)_{2m}\otimes(\cdot\, ,\cdot)_{2n}$. It is easy to check that, for $j,k\in\{0,\ldots,2m\},\,i,r\in\{1,\ldots,n\}$
\begin{align*}&\langle E_{j,2m+i}+E_{2m+n+i,j},E_{k,2m+r}+E_{2m+n+r,k}\rangle=0,\\
&\langle E_{j,2m+n+i}-E_{2m+n+i,j},E_{k,2m+n+r}+E_{2m+r,k}\rangle=0,\\
&\langle E_{j,2m+i}+E_{2m+n+i,j},E_{k,2m+n+r}-E_{2m+r,k}\rangle=-\d_{jk}\d_{ir}.\end{align*}
It follows that $\langle\Phi(X),\Phi(Y)\rangle=\langle X,Y\rangle $ for any $X,Y\in V$.
According to the classification given in \cite[\S\ 5]{thetaseries}, the pair   $(O(2m), Sp(2n,\R))$ is a type I dual pair in  
$Sp(\R^{2m}\otimes (\R^{2n})^*,\langle\cdot\,,\cdot\rangle)$. Since $\Phi$ maps   $ad_{|\g_1}(\g_0)\cap sp(V)\subset sp(V)$  exactly on the Lie algebras of this dual pair, we have proven Proposition \ref{B} in this case.

\subsubsection{$\Dp_{D2}$} 

\vskip5pt
Endow the superspace $\mathbb C^{2m|2n}$ with the bilinear form given by the matrix 
$
\left(\begin{array}{cc} I_{m,m} & 0\\ 0& J_{2n}\end{array}\right),
$
where
 $I_{m,m}=\left(\begin{array}{cc} 0& I_m\\ I_m& 0\end{array}\right)$. \par With this choice of the form, the superalgebra $osp(2m,2n)$ can be realized as the set of matrices
$$\left(\begin{array}{cccc} A_1&A_2 & B_1&B_2 \\ A_3&-{}^tA_1&B_3&B_4\\{}^t B_4&{}^t B_2& C_1 & C_2 \\- {}^tB_3 &-{}^tB_1& C_3 & -{}^tC_1\end{array}\right)$$ with $A_2,A_3$ skewsymmetric $m\times m$ and $C_2,C_3$ symmetric $n\times n$ matrices.
In this case we set
$$h_i=
\begin{cases}E_{i,i}-E_{m+i,m+i}, &1\leq i\leq m,\\
E_{m+i,m+i}-E_{m+n+i,m+n+i}, &m+1\leq i\leq m+ n.\end{cases}
$$
Then $\h=\oplus_{i=1}^{m+n}\C h_i$ is a Cartan subalgebra of $\g$ and 
$$( h_i,h_j)=\begin{cases}\d_{ij}\quad&1\leq i\leq m,\\-\d_{ij}\quad& m+1\leq i\leq m+n.\\
\end{cases}$$
In the chosen distinguished set of positive roots
we have $\Dp_1=\{\e_i\pm\d_j\mid 1\leq i\leq m, \, 1\leq j\leq n\}$. We can choose for each $\a\in\Dp_1$ root vectors $X_\a$, $X_{-\a}$ in such a way that

\begin{equation}\label{V2}V=\left\{\left(\begin{array}{cc}
0&A\\-J_{2n}{}^tA& 0\end{array}\right)\mid A=\left(\begin{array}{cc}
A_1&A_2\\\sqrt{-1}\bar{A_2}& -\sqrt{-1}\bar{A_1}\end{array}\right),\ A_1,A_2\in M_{m,n}(\C)\right\}.\end{equation}
\par
Again, the map $\Psi$, given by \eqref{psi}, intertwines the adjoint action of $\g_0$ on $\g_1$ with its action on $M_{2m,2n}(\C)$, given by 
\eqref{azioneg0}. It is then clear that $ad\left(\begin{array}{cc}
E&0\\0&F \end{array}\right)(V)\subset V$ if and only if 
\begin{equation}\label{gzeroreal}
E=\left(\begin{array}{cc}
A_1&A_2\\-\bar A_2&\bar A_1 \end{array}\right),\quad F=\left(\begin{array}{cc}
B_1&B_2\\-\bar B_2&\bar B_1 \end{array}\right)
\end{equation}
with $ A_1,A_2\in M_{m,m}(\C)$, $B_1,B_2\in M_{n,n}(\C)$, $A_1,B_1$ skew-Hermitian, $A_2$ antisymmetric, and $B_2$ symmetric.
\par
Let $\mathbb H$ be the  skew field of real quaternions. Set $(\H^n)^*=End_\H(\H^n,\H)$ where $\H^n,\H$ are viewed as left $\H$-spaces. Endow $(\H^n)^*$ with a right $\H$-action by setting $(\l q)(v)=\l(v) q$. Identify $V$ with $(\H^n)^*\otimes_\H\H^m$ as follows. The map $\Psi$ above followed by left multiplication by $L_m=\left(\begin{array}{cc}
I_m&0\\0&\sqrt{-1}I_m \end{array}\right)$ maps $V$ into the subspace of $M_{2m,2n}(\C)$ of matrices of the form 
\begin{equation}\label{quaternionic}
\left(\begin{array}{cc}
A&-\bar B\\B&\bar A \end{array}\right).
\end{equation} 
Identifying $\H^r$ with $\C^r\times\C^r=\C^{2r}$ by  $x+yj\leftrightarrow(x,y)$  we see that the matrices as in \eqref{quaternionic} are precisely those commuting with the left action of $j$. Thus $V$ gets identified with $End_\H(\H^n,\H^m)$ (here $\H^n,\H^m$ are seen as left $\H$-spaces). The natural map $\l\otimes v\mapsto T_{\l,v}$ with $T_{\l,v}(u)=\l(u)v$ identifies $End_\H(\H^n,\H^m)$ and $(\H^n)^*\otimes_\H\H^m$.
Let now $\Phi:V\to(\H^n)^*\otimes_\H\H^m$ be the ($\R$-linear) isomorphism described above. 
Note that, if $E,F$ are as in \eqref{gzeroreal}, then both $L_m E L_m^{-1}$ and $F$ are of the form \eqref{quaternionic}. Thus we can identify them with elements $T_E, T_F$ in $End_\H(\H^m,\H^m)$ and $End_\H(\H^n,\H^n)$ respectively. Unwinding all the identifications we see that if $X=\left(\begin{array}{cc}
E&0\\0&F \end{array}\right)\in\g_0$ and $ad(X)(V)\subset V$ then $\Phi ad(X)\Phi^{-1}$ acts on $(\H^n)^*\otimes_\H\H^m$ via the natural action of $T_F\otimes I +I\otimes T_E$.

Let $\tau:\H\to\H$ be the quaternionic conjugation:  $\tau(a+bj)=\bar a-bj$ ($a,b\in\C$). It induces the complex conjugation in the
above identification $\H^r=\C^{2r}$.
Let $(\cdot\, ,\cdot)'_{m}$ be the skew-Hermitian  form on $\H^{m}$ given by $(v ,w )'_{m}=\sum_iv_i\sqrt{-1}\tau(w_i)$ and $(\cdot\, ,\cdot)_{n}$ the  Hermitian form on $\H^{n}$ given by $(v ,w )_{n}=\sum_iv_i\tau(w_i)$. Let $(\cdot\, ,\cdot)_{n}$ denote also the form induced on $(\H^n)^*$. Note that, since $E,F$ are as in \eqref{gzeroreal}, then $T_E\in u(\H^m,(\cdot\, ,\cdot)'_m)\simeq so^*(2m)$ and $T_F\in u((\H^n)^*, (\ , \ )_n)\simeq sp(n)$.

Let $\langle\cdot\,,\cdot\rangle$ be the real symplectic bilinear form on $(\H^n)^*\otimes \H^m$ given by $$\langle\cdot\,,\cdot\rangle=Tr_{\H/\R}(\cdot\, ,\cdot)_{n}\otimes\tau\circ(\cdot\, ,\cdot)_{m}.$$ For $1\le i\le m$, $1\le j\le n$, set
\begin{align}\label{psi1}
&X_{ij}=\Psi^{-1}(\sqrt{-1}E_{ij}-E_{m+i,n+j}),\\\label{psi2}
&Y_{ij}=\Psi^{-1}(E_{ij}-\sqrt{-1}E_{m+i,n+j}),\\\label{psi3}
&Z_{ij}=\Psi^{-1}(\sqrt{-1}E_{i,n+j}+E_{m+i,j}),\\\label{psi4}
&W_{ij}=\Psi^{-1}(E_{i,n+j}+\sqrt{-1}E_{m+i,j}).
\end{align}
If $X,Y\in\{X_{ij},Y_{ij},Z_{ij},W_{ij}\}$, then it is easy to check that $\langle X,Y \rangle=0$ except in the following cases:
$$2=\langle X_{ij},Y_{ij}\rangle=-\langle Y_{ij},X_{ij}\rangle=
\langle Z_{ij},W_{ij}\rangle=-\langle W_{ij},Z_{ij}\rangle. 
$$

On the other hand, letting $\{e_i\}$ be the canonical basis of $\H^m$ and $\{e^i\}$ the basis dual to the canonical basis of $\H^n$, we see that
\begin{align}\label{phi1}
&\Phi(X_{ij})=e^j\otimes \sqrt{-1}e_i,\\\label{phi2}
&\Phi(Y_{ij})=e^j\otimes e_i,\\\label{phi3}
&\Phi(Z_{ij})=e^j\otimes\sqrt{-1}j e_i,\\\label{phi4}
&\Phi(W_{ij})=-e^j\otimes j e_i.
\end{align}
It follows that $\langle\Phi(X),\Phi(Y)\rangle=-2\langle X,Y\rangle$ for any $X,Y\in V$.

According to the classification given in \cite[\S\ 5]{thetaseries}, the pair 
$$(U((\H^n)^*,(\cdot\, ,\cdot)_{n}),U(\H^m,(\cdot\, ,\cdot)_{m}))=(Sp(n),SO^*(2m))
$$ is a Type I dual pair in  
$Sp((\H^n)^*\otimes\H^{m},\langle\cdot\,,\cdot\rangle)$. Since the map $\Phi$ maps   $ad_{|\g_1}(\g_0)\cap sp(V)\subset sp(V)$  exactly on the Lie algebras of this dual pair, we have proven Proposition \ref{B} in this case.

\subsubsection{$\Dp_{D2'}$} 
This case can be reduced to the $\Dp_{D2}$ case as follows: by a careful choice of the root vectors, the space $V$  is 
the set of matrices
$$\left(\begin{array}{cc}
0&A\\-J_{2n}{}^tA& 0\end{array}\right)
$$
such that
$$ A=\left(\begin{array}{c | c}
A_1&A_2\\v&w\\ \hhline{-|-}
\sqrt{-1}\bar{A_2}& -\sqrt{-1}\bar{A_1}\\
-\sqrt{-1}\bar w&\sqrt{-1}\bar v\end{array}\right),
$$
with $\ A_1,A_2\in M_{m-1,n}(\C)$, $v,w\in\C^n$.
Define $S_{m,2m}\in M_{2m,2m}(\C)$ as
$$
S_{m,2m}=\left(\begin{array}{cc|cc}
I_{m-1}&0&0&0\\0&0&0& 1\\
\hline
0& 0&I_{m-1}&0\\
0& 1&0&0\end{array}\right).
$$

Let $\s:\g\to \g$ be the map  
$$
\left(\begin{array}{cc}
A&B\\-J_{2n}{}^tB& C\end{array}\right)\mapsto \left(\begin{array}{cc}
S_{m,2m}AS_{m,2m}&S_{m,2m}B\\-J_{2n}{}^tBS_{m,2m}& C\end{array}\right).
$$
It is easy to check that $\s$ is an automorphism of $\g$.
Since  $\si(V)$ is the space described in \eqref{V2} we see that the map $X\mapsto \Phi\circ ad(\s(X))\circ\Phi^{-1}$ identifies $ad_{|\g_1}(\g_0)\cap sp(V)$ with  $sp(n)\times so^*(2m)$.
Finally observe that 
$str(\si(A)\si(B))=str(AB)$
hence   $$\langle\Phi(\si(X)),\Phi(\si(Y))\rangle=-2( X ,Y)$$ for any $X,Y\in V.$
This concludes the proof of Proposition \ref{B} in this case.
\subsection{Noncompact dual pairs and gradings of depth 4}\label{cinquesei}
In this subsection we classify, up to $W_\g$-action,  the sets of positive roots such that 
the grading \eqref{grading} has depth at most 4 (i.e., $\mathfrak q_i=\{0\}$ for $|i|>4$) and we show 
that these sets of positive roots are related to  the noncompact dual pairs 
$(U(p,q), U(r,s))$, $(O(p,q), Sp(2n,\R))$, $(Sp(p,q),\,SO^*(2n))$. These pairs exhaust the type I dual pairs, with the exception of  the pair $(O(m,\C),Sp({2n},\C))$. The latter pair  cannot arise in our picture, since 
$O(m,\C)\times Sp({2n},\C)$  is not a real form of $\g_0$ for any choice of $\g$.  It occurs as the real form $\g_0\cap \g_\R$ with   $\g=osp(m,2n)\times osp(m,2n)$ and $\g_\R$ the diagonal copy of $osp(m,2n)$ in $\g$. Since we are chiefly interested in the application to dual pairs, we confine to study 
$gl(m,n),\,B(m,n),$ $D(m,n),\,C(n+1)$.\par
In the following we  choose $\omega_0$ (see \eqref{omega0}) as antilinear involution and let $\l^0_\a$ be the corresponding multipliers as defined in  \eqref{ox}. Recall that the corresponding  involution $\zeta_{\omega_0}$ of $\g_0$ is  defined by setting $\zeta_{\omega_0}(X_\a)=sgn(\l^0_\a)X_\a$.
\vskip5pt
\noindent {\it Case $\g=gl(m,n)$.} It is clear that diagrams of type $A$ with three or four grey nodes support sets of positive roots whose associated grading has depth 
3 or 4, respectively. These sets of positive roots correspond to the following sets of  simple roots:  

\begin{align}\label{A3}
\{&\e_ 1-\e_ 2,\ldots,\e_ p-\d_1,\d_1-\d_2,\ldots,\d_r-\e_ {p+1},\e_ {p+1}-\e_ {p+2},\ldots,\\\notag
 &\e_ {p+1}-\d_{r+1},\d_{r+1}-\d_{r+2},\ldots,\d_{r+s-1}-\d_{r+s}\},\\
\label{A4}
\{&\e_ 1-\e_ 2,..,\e_ h-\d_1,\d_1-\d_2,..,\d_r-\e_ {h+1},\e_ {h+1}-\e_ {h+2},..,
 \e_ {h+q}-\d_{r+1},\\\notag &\d_{r+1}-\d_{r+2},..,\d_{r+s-1}-\d_{r+s},\d_{r+s}-\e_ {h+q+1},..,\e_ {h+k+q-1}-\e_ {h+k+q}\},
\end{align}
and are clearly the only ones with the required property about the grading, up to switching the role of $\e$ and $\d$. Here $p+q=m, r+s=n$ in  \eqref{A3} and
$h+k=p, p+q=m, r+s=n$ in  \eqref{A4}. 
 Fix \eqref{A4} as a set of simple roots. To detect the real form corresponding to  $\omega_0$, we need only to classify $(\zeta_{\omega_0})_{|\g_0^i},\,i=1,2$ via the  parameters 
 $s^i_j,\,i=1,2,\,1\leq j\leq k_i,\,,\,k_1=m-1,\,k_2=n-1$ as explained in Section \ref{realforms}. In this case the roots 
 $\beta^i_j$ are the simple roots of $\g_0^i$:
 \begin{align*}
 &\beta^1_{j}=\a_j,\,1\leq j\leq h-1,  &&\beta^2_{j}=\a_{h+j},\,1\leq j\leq r-1,\\
&\beta^1_h=\a_h+\a_{h+1}+\ldots+\a_{h+r}, &&\beta^2_r=\a_{h+r}\ldots+\a_{h+r+q},\\
&\beta^1_{h+j}=\a_{h+r+j},\,1\leq j\leq q-1,&&\beta^2_{r+j}=\a_{h+r+q+j},\,1\leq j\leq s-1,\\
&\beta^1_{h+q}=\a_{h+r+q}+\ldots+\a_{h+r+q+s},\\
&\beta^1_{h+q+j}=\a_{h+r+q+s+j},\,1\leq j\leq k-1.
 \end{align*}
while $\theta_i$ is the highest root of $\g_0^i$.
In this case $\zeta_{\omega_0}=e^{2\pi\sqrt{-1}t}$ with $\beta_j^i(t)=\frac{1}{2}$ if  $sgn(\l^0_{\beta^i_j}) =-1$ and $\beta_j^i(t)=0$ otherwise. Note that, letting $s^i_j=\beta_j^i(t)$ and $s^i_0=1-\theta_i(t)$, then $t$ satisfies \eqref{ordertwo}. It follows that, since  $s^1_j=0$ for $j\ne h,h+q$, $s^1_h=s^1_{h+q}=\frac{1}{2}$, $s^2_i=0$ 
for $i\ne 0,r$, $s^2_0=s^2_r=\frac{1}{2}$, the corresponding real form of $\g_0$ is  $ u(p,q)\times u(r,s)$.
 
\vskip5pt
\noindent {\it Case $\g=B(m,n)$.} Recall that in type $B$ all diagrams have shape \eqref{diaB}.  Assume  $|\pi|=2$. Suppose first $\a_{n+m}\in\pi$: then $\a_{n+m}$ is odd and non isotropic. The grading has depth 4. The simple roots of $\g_0$ that are not simple in $\g$ are both of degree $2$, while the largest roots are of degree $4$. Arguing as for $gl(m,n)$ it is easy to see that the corresponding real form is $so(2m,1)\times sp(2n,\R)$. The corresponding dual pair is $(O(2m,1),Sp(2n,\R))$.
Otherwise the set of positive roots  is necessarily of the form
$$\{\e_ 1-\e_ 2,\ldots,\e_ p-\d_1,\d_1-\d_2,\ldots,\d_n-\e_ {p+1},\ldots,\e_ {m-1}-\e_ m,\e_ m\}$$
with $p+q=m$.
This grading has depth 4, since the highest root is $\e_ 1+\e_ 2$. Arguing as above, we can show that  it gives rise to the dual pair $(O(2p,2q+1), Sp(2n,\R))$. It is easily seen that if $|\pi|\geq3$ then the grading has depth strictly greater than 4.
\vskip5pt
\noindent {\it Case $\g=D(m,n)$.} 
If  the diagram of the set of positive roots is  like \eqref{diaD1}, then arguing as in the previous case we deduce that the only 
sets of positive roots with $|\pi|=2$ and depth at most 4 (indeed exactly 4) are
\begin{align*}&\{\d_1-\d_2,\ldots,\d_p-\e_ 1,\e_ 1-\e_ 2,\ldots,\e_{m-1}-\e_m,\e_ m-\d_{p+1},\ldots,\d_{p+q-1}-\d_{p+q},2\d_{p+q}\},\\
&\{\d_1-\d_2,\ldots,\d_p-\e_ 1,\e_ 1-\e_ 2,\ldots,\e_{m-1}+\e_m,-\e_ m-\d_{p+1},\ldots,\d_{p+q-1}-\d_{p+q},2\d_{p+q}\}.
\end{align*}
with $p+q=n$, which gives rise to the dual pair $(Sp(2p,2q),\,SO^*(2m))$.
If instead the diagram is like \eqref{diaD3}, we again obtain a unique set of positive roots:
\begin{equation}\label{DDDD}\{\e_ 1-\e_ 2,\ldots,\e_ p-\d_1,\d_1-\d_2,\ldots,\d_n-\e_ {p+1},\ldots,\e_ {m-1}-\e_ m,\e_ {m-1}+\e_ m\}\end{equation}
with $p+q=m$, which gives rise to the dual pair $(O(2p,2q), Sp(2n,\R))$.
In both cases, if  $|\pi|\geq 3$ then  the grading has depth strictly greater than 4.

The pairs  $(O(2p+1,2q-1),$ $Sp(2n,\R))$ are gotten by considering the antilinear involution $\omega=\eta_1\omega_0$ when the positive sets of roots are those in \eqref{DDDD} and $\eta_1=Ad(J)$. We need to classify  $\zeta_\omega=\eta_1\circ\zeta_{\omega_0}$. The extended Dynkin diagrams of $\g_0^1,\g_0^2$ are of type $D^{(2)}_m,\,C^{(1)}_n$ respectively.
 The corresponding roots $\beta^i_j$ are 
 \begin{align*}
 &\beta^1_{j}=\a_j,\,1\leq j\leq p-1,  &&\beta^2_{j}=\a_{p+j},\,1\leq j\leq n-1,\\
&\beta^1_p=\a_p+\a_{p+1}+\ldots+\a_{p+n}, &&\beta^2_{2n}=2(\a_{p+n}\ldots+\a_{m+n-2})+\a_{m+n-1}+\a_{m+n},\\
&\beta^1_{p+j}=\a_{p+n+j},\,1\leq j\leq q-2,\\
&\beta^1_{m-1}=\tfrac{1}{2}(\a_{m+n}+\a_{m+n-1}).
 \end{align*}
 
 The parameters for $\zeta_\omega$ are $s^1_p=\frac{1}{2}$  and $s^1_j=0$ for $j\ne p$, $s^2_n=\frac{1}{2}$ and $s^2_j=0$ for $j\ne n$. It follows that the real forms of $\g_0^i$ are respectively $so(2p+1, 2q-1)$ and $sp(2n,\R)$, which give rise to the dual pairs $(O(2p+1,2q-1), Sp(2n,\R))$. 
 
 Note that the pair $(O(1,2m-1), Sp(2n,\R))$ is obtained from the distinguished set of positive roots whose set of simple roots is 
 \begin{equation}\label{D1odd}\{\d_1-\d_2,\ldots,\d_n-\e_ {1},\e_ 1-\e_ 2,\ldots,\e_ {m-1}-\e_ {m},\e_ {m-1}+\e_ {m}\}.
 \end{equation}
\vskip5pt
\noindent {\it Case $\g=C(n+1)$.} According to \cite[page 52]{Kacsuper}, the only sets of simple roots which correspond to non-distinguished sets of positive roots are
\begin{align*}&\{\d_1-\d_2,\ldots,\d_i-\e_ 1,\e_ 1-\d_{i+1},\d_{i+1}-\d_{i+2},\ldots,2\d_n\},\\
&\{\d_1-\d_2,\ldots,\d_i+\e_ 1,-\e_ 1-\d_{i+1},\d_{i+1}-\d_{i+2},\ldots,2\d_n\}.
&\end{align*}
where $i$ ranges from $1$ to $n-1$. The associated gradings have depth $4$.
\section{Theta correspondence for the pair $(O(2m+1),Sp(2n,\R))$}\label{Bnm}

In this section we use the denominator identity developed in the previous sections to derive the Theta correspondence for the  compact dual pair $(O(2m+1),Sp(n,\R))$. 
This pair, according to Proposition \ref{B}, corresponds to  the distinguished set of positive roots $\Dp_B$ in a superalgebra $\g$ of type $B(m,n)$. Recall that
\begin{align}
\label{B1}\D_0&=\pm\{\e_i\pm\e_j,\e_i,\d_k\pm\d_l,2\d_k\mid 1\leq i\ne j\leq m,\, 1\leq k\ne l\leq n\},\\\notag
\D_1&=\pm\{\d_k\pm\e_i,\d_k\mid 1\leq i\leq m,\, 1\leq k\leq n\};\\
\label{B2}\overline{\D}_0&=\pm\{\e_i\pm\e_j,\e_i,\d_k\pm\d_l\mid 1\leq i\ne j\leq m,\, 1\leq k\ne l\leq n\},\\\notag
\overline\D_1&=\pm\{\d_k\pm\e_i\mid 1\leq i\leq m,\, 1\leq k\leq n\}.
\end{align}
Also recall that  the defect of $\g$ is $d=\min(n,m)$.  In this Section we will use the following notation
\begin{align}
\D(B_r)&=\pm\{\e_i\pm\e_j,\e_i\mid 1\le i\ne j\le r\},\\
\D(C_r)&=\pm\{\d_k\pm\d_l,2\d_k\mid n-r+1\le k\ne l\le n\},\\
\D(A_{r-1})&=\pm\{\d_k-\d_l\mid n-r+1\leq k\ne l\leq n\},\\
W(A_{r-1})&= \text{Weyl group of $\D(A_{r-1})$},\\
W(B_r)&= \text{Weyl group of $\D(B_r)$},\\\label{x}
{\mathcal W_r} &= \text{subgroup of $W_\g$ generated by $\{s_{2\d_i}\mid i=1,\dots, r\}$},\\\label{x+}
{\mathcal W_r}^+&= \{w\in {\mathcal W_r}\mid w=s_{2\d_{i_1}}\dots s_{2\d_{i_k}},\,\text{$k$ even}\},\\\label{x-}
{\mathcal W_r}^-&= \{w\in {\mathcal W_r}\mid w=s_{2\d_{i_1}}\dots s_{2\d_{i_k}},\,\text{$k$ odd}\}.
\end{align}

\par 
Now fix the set of positive roots $\Dp_B$ (cf. \eqref{B3}). 
Then $2\rho_1=(2m+1)(\d_1+\ldots+\d_n)$. Set $\Dp(C_n)=\Dp_0\cap\D(C_n)$, $\Dp(B_m)=\Dp_0\cap\D(B_m)$, and $\Dp(A_{n-1})=\Dp_0\cap\D(A_{n-1})$. We denote by $\rho^C$, $\rho^B$, $\rho^A$ the corresponding $\rho$-vectors.
With notation as in  Section  \ref{ED}, the total order corresponding to this choice of $\Dp$ is 
$$
\d_1>\dots >\d_n>\e_1>\dots>\e_m.
$$
There is only one arc diagram associated to this total order, namely $$X=\{
\stackrel{\frown}{\d_n\e_1},\dots, \stackrel{\frown}{\d_{n-d+1}\e_d}\}.$$ The corresponding maximal isotropic set of roots is $S(X)=\{\gamma_1,\dots,\gamma_d\}$, where
\begin{equation}\label{gammab}
\gamma_1=\d_{n}-\e_1,\,\gamma_2=\d_{n-1}-\e_2,\ldots,\gamma_d=\d_{n-d+1}-\e_d.
\end{equation}
We apply Proposition \ref{migliore} with $\mathcal B'=Supp(X)$.  Note that, with notation as in \ref{mmm}, $\D(Supp(X))=\D(B_d)\times \D(C_d)$, and we can choose  $\D^\sharp(Supp(X))=\D(B_d)$, so that 
$W_{Supp(X)}W^\sharp(Supp(X))=W(A_{d-1})W(B_d)$. With notation as in \eqref{x}, set
$$
W=W(A_{n-1}) {\mathcal W_{n-d}}W(B_m).
$$
We can choose $$Z=(W(A_{n-1})/W(A_{d-1}) ){\mathcal W_{n-d}}(W(B_m)/W(B_d)),$$ so that
$
W_0=W.
$
Since $\prod_{i=1}^d\frac{ht(\gamma_i)+1}{2}=d!$, Proposition \ref{migliore} Êgives
\begin{equation}\label{seconda}
e^{\rho}\check R=
{\check \cF}_{W}\bigl(
\frac{e^{\rho}}{\prod_{i=1}^d(1-e^{-\llbracket\gamma_i\rrbracket})}\bigr).
\end{equation}

Dividing \eqref{seconda} by $\mathcal{D}_0=e^{\rho_0}\prod_{\a\in\Dp_0}(1-e^{-\a})$, we find
\begin{align}\label{terza}
chM^{\Dp}(\g_1)&=\frac{e^{-\rho_1} }{\prod_{\a\in \Dp_1}(1-e^{-\a})}=\frac{1}{\mathcal{D}_0} {\check \cF}_{W}\bigl(\frac{e^{\rho}}{\prod_{i=1}^d(1-e^{-\llbracket\gamma_i\rrbracket})}\bigr).
\end{align}

Let $\PP_d=\{\ba =(a_1,\dots,a_d)\in(\ganz^+)^d\mid a_1 \ge a_2\ge\dots \ge a_d\}$ the set of partitions with at most $d$ parts. 
Recall from Section \ref{5} that there is an element $H\in \h$ such that $\a(H)=1$ for all $\a\in\Dp_1(\g)$, thus the domain $\h^+=\{h\in\h\mid \a(h)>0\,   \forall \a\in \Dp_1\}$ is nonempty. 
Since, by our choice of $W$, in the denominators  of \eqref{terza} only sums of roots in $-\Dp_1$ occur, we can expand the r.h.s.  of \eqref{terza} in a product of geometric series on $\h^+$, thus obtaining the following  equality of  formal power series:
\begin{equation}\label{quarta}
chM^{\Dp}(\g_1)=\sum_{\mathbf a\in \PP_d}\frac{1}{\mathcal{D}_0} {\check \cF}_{W}\bigl(e^{\rho-\sum_{i=1}^d a_i\gamma_i}\bigr).
\end{equation}
We can write $\mathcal{D}_0=e^{\rho^{C}}\prod_{\a\in\Dp(C_n)}(1-e^{-\a})e^{\rho^{B}}\prod_{\a\in\Dp(B_m)}(1-e^{-\a})$.
For $\ba\in  \PP_d$, define
$$
\mu(\ba)= -\sum_{r=1}^d(a_{d+1-r})\d_{n-d+r},\ 
\varepsilon(\ba)=\sum_{r=1}^m a_r \epsilon_r.
$$
 Using \eqref{quarta}, we obtain that
\begin{align*}
chM^{\Dp}(\g_1)&=\sum_{\mathbf a\in \PP_d}\left( \sum\limits_{w\in W(A_{n-1}){\mathcal W_{n-d}}}sgn'(w) \frac{e^{w(\rho^C-\rho_1+\mu(\ba))-\rho^C}}{\prod_{\a\in\Dp(C_n)}(1-e^{-\a})}\right)\\&\times\left( \sum\limits_{w\in W(B_{m})}sgn(w) \frac{e^{w(\rho^B+\varepsilon(\ba))-\rho^B}}{\prod_{\a\in\Dp(B_m)}(1-e^{-\a})}\right).
\end{align*}
Here we used the fact that $sgn(w)=sgn'(w)$ if $w\in W(B_m)$.
By the Weyl character formula, we see that 
$$
 \sum\limits_{w\in W(B_{m})}sgn(w) \frac{e^{w(\rho^B+\varepsilon(\ba))-\rho^B}}{\prod_{\a\in\Dp(B_m)}(1-e^{-\a})}=ch F_B(\varepsilon(\ba)),
$$
 $F_B(\varepsilon(\ba))$ being the finite-dimensional irreducible $so(n)$-module with highest weight $\varepsilon(\ba)$. We therefore obtain that
\begin{align}\label{intermedia}
chM^{\Dp}(\g_1)=\sum_{\mathbf a\in \PP_d}\sum\limits_{w\in W(A_{n-1}){\mathcal W_{n-d}}}sgn'(w) \frac{e^{w(-\rho_1+\mu(\ba)+\rho^C)-\rho^C}}{\prod_{\a\in\Dp(C_n)}(1-e^{-\a})}chF_B(\varepsilon(\ba)).
\end{align}
If $\l\in span(\d_i)$, set 
\begin{equation}\label{FAL}
ch F_A(\l)= \sum_{w\in W(A_{n-1})}sgn(w)\frac{e^{w(\l+\rho^A)-\rho^A}}{\Pi_{\a\in\Dp(A_{n-1})}(1-e^{-\a})}.
\end{equation}
As the notation is suggesting, when $\l+\rho^A$ is regular and dominant integral  for $\D(A_{n-1})$, then $ch F_A(\l)$ is the character of a finite-dimensional irreducible representation of $sl(n,\C)$, while, if $\l+\rho^A$ is singular, then $ch F_A(\l)=0$. 

Using the fact that $sgn(w)=1$ if $w\in {\mathcal W_{n-d}}$, 
we can rewrite \eqref{intermedia} as 
\begin{align}\label{hower}
chM^{\Dp}(\g_1)=\sum_{\mathbf a\in \PP_d} \sum\limits_{w\in {\mathcal W_{n-d}}} \frac{ch F_A(w(-\rho_1+\mu(\ba)+\rho^C)-\rho^C)}{\prod_{\a\in\Dp(C_n)\backslash\Dp(A_{n-1})}(1-e^{-\a})}chF_B(\varepsilon(\ba)).
\end{align}

 Recall from the previous section that we constructed a real symplectic subspace $V$ of $\g_1$ and a map $\Phi:V\to \R^{2m+1}\otimes (\R^{2n})^*$ such that there is a dual pair $(G_1,G_2)$  in $sp(V,(\ , \ ))$ having the properties described in Proposition \ref{B}.
 We now want to describe explicitly the set $\Sigma$ and the map $\tau$ occurring in Theorem \ref{howe}. In what follows we show that all these information can be read off from formula \eqref{hower}.

First of all we  need to parametrize the finite-dimensional irreducible representations  of $\widetilde G_1$. To accomplish this we start by describing $\widetilde G_1$.  
With  notation as in  the previous section, let $J':(\R^{2n})^*\to (\R^{2n})^*$ be the complex structure such that $J'(f_i)=-f_{n+i}$ and $J'(f_{n+i})=f_i$ for $i=1, \dots, n$. A direct computation shows that, if $J$ is the compatible complex structure on $V$ introduced in the previous section, then $\Phi J\Phi^{-1}=Id\otimes J'$. It follows that, if we let $W'$ be  the space $(\R^{2n})^*$ seen as a complex space via the complex structure $J'$, then, obviously, $\Phi$ is a $\C$-linear isomorphism between $W$ and the complex space $\R^{2m+1}\otimes W'$. 
 We look upon  $O(2m+1)$ as a subgroup of $Sp(\R^{2m+1}\otimes (\R^{2n})^*)$ via its action on the first factor. Then $G_1=\{\Phi^{-1}g \Phi\mid g\in  O(2m+1)\}$. It follows that, if  $\Phi^{-1}g \Phi\in G_1$, then 
 $$
 det_W(\Phi^{-1}g \Phi)=det_{\R^{2m+1}\otimes W'}(g)=(det_{\R^{2m+1}}(g))^n.
 $$
 Therefore
 $$ \widetilde G_1=
 \begin{cases}\ 
\{(\Phi^{-1}g \Phi,\pm 1)\mid g\in O(2m+1)\}&\text{if $n$ is even,}\\
\begin{array}{l}\{(\Phi^{-1}g \Phi,\pm 1)\mid g\in SO(2m+1)\}\\
\cup\{(\Phi^{-1}g \Phi,\pm \sqrt{-1})\mid -g\in SO(2m+1)\}\end{array}&\text{if $n$ is odd.}
 \end{cases}
 $$
 Let $\widetilde G_1^0\simeq SO(2m+1)$ be the connected component of the identity of $\widetilde G_1$. From the description of $\widetilde G_1$ we see that $\widetilde G_1=\widetilde G_1^0\times Z$ with $Z$ isomorphic to $\ganz/2\ganz\times \ganz /2\ganz$ if $n$ is even and isomorphic to $\ganz/4\ganz$ if $n$ is odd. The generators of $Z$ are $(-Id,1),(Id,-1)$ in the first case, while $Z$ is generated by $(-Id, \sqrt{-1})$ in the second case. It follows that the finite-dimensional irreducible representations of $\widetilde G_1$ are pairs $(F,\chi)$ where $F$ is a  finite-dimensional irreducible representation of $\widetilde G_1^0$ (hence of its complexified Lie algebra $\mathfrak{s}_1^\C$) and $\chi$ is a character of $Z$.

 If $n$ is even and $\e\in\{1,-1\}$, we let $\chi_\e$  be the character of $Z$ such that $\chi_\e(-Id,1)=\e$ and $\chi_\e(Id,-1)=-1$. If $n$ is odd we define $\chi_\e$ to be the character of $Z$ such that $\chi_\e(-Id,\sqrt{-1})=-\sqrt{-1}\e$.
 If $\ba\in \PP_d$ set $\e(\ba)=(-1)^{\sum_{i=1}^d a_i}$.

In our case the set of roots of $\mathfrak{s}_2^\C$ is $\D(C_n)$. In the identification of $\h$ with $\h^*$ given by $(\ , \ )$, we see that the element $H\in \h$ that corresponds to $-\sum_{i=1}^n\d_i$ has the property that $H_{|V_\C^\pm}=\pm I$. Thus the parabolic subalgebra $\p_2$ defined by $H$ is
$$
\p_2=\h\oplus \sum_{\a\in\D(A_{n-1})}(\g_0)_\a\oplus \n,
$$
where 
$
\n=\sum\limits_{\a\in\Dp(C_{n})\backslash \D(A_{n-1})}(\g_0)_\a
$ is the nilradical.
\begin{nota}\label{nota}

If $\l\in (\h\cap\mathfrak{s}_2^\C)^*$ is such that $\l+\rho^A$ is regular and integral  for $\D(A_{n-1})$, we let $L^2(\l)$ be the irreducible quotient of the $\mathfrak p_2$-parabolic Verma module $V^2(\l)$ for $\mathfrak{s}_2^\C$. \par
 If $\mu\in (\h\cap\mathfrak{s}_2^\C)^*$ is such that $\mu$ is regular  for $\D(A_{n-1})$, we let $\{\mu\}$  be the unique element in the $W(A_{n-1})$-orbit of $\mu$ that is dominant with respect to $\Dp(A_{n-1})$ .\par
Finally,  we set $c_\mu=sgn(v)$, where $v$ is the unique element of $W(A_{n-1})$ such that $v(\mu) =\{\mu\}$. 
\end{nota}
Note that
$$
ch V^2(\l)=c_{\l+\rho^A}\frac{ch F_A(\l)}{\prod_{\a\in\Dp(C_n)\backslash\Dp(A_{n-1})}(1-e^{-\a})}.
$$

 \begin{prop}\label{theta}With  notation as in  Theorem \ref{howe}, we have that, if $(F,\chi)\in \Sigma$, then there exists $\ba\in \PP_d$ and $\e\in\{1,-1\}$ such that 
 $$F=F_B(\varepsilon(\ba)),\quad \chi=\chi_\e.
 $$
Furthermore, the $\h$-character of the isotypic  component of $(F_B(\varepsilon(\ba)),$ $ \chi_{\pm\e(\ba)})$ in  $M^{\Dp}(\g_1)$ is
\begin{equation}\label{isotypic}
  \sum\limits_{w\in \mathcal W^\pm_{n-d}}\frac{ch F_A(w(\rho^C-\rho_1+\mu(\ba))-\rho^C)}{\prod_{\a\in\Dp(C_n)\backslash\Dp(A_{n-1})}(1-e^{-\a})}chF_B(\varepsilon(\ba)).
\end{equation}
($\mathcal W^\pm_{n-d}$ are defined in  \eqref{x+}, \eqref{x-}).
 \end{prop}
 \begin{proof}
 If $(F,\chi)$ occurs in $M^{\Dp}(\g_1)$ then, by \eqref{howe}, $F=F_B(\varepsilon(\ba))$ for some $\ba\in \PP_d$. As in the previous section, we identify $M^{\Dp}(\g_1)$ with $P(W)$. Let $P(W)^+$ and $P(W)^-$ be the subspaces of homogeneous  polynomials of even and odd degree respectively. By the explicit expression for the action of $\widetilde K$ given in \eqref{actionk}, we see that $Z$ acts by $\chi_1$ on $P(W)^+$ and by $\chi_{-1}$ on $P(W)^-$. This proves the first assertion.
 
Let $M(\ba,\chi_\e)$ be the isotypic component of $(F_B(\varepsilon(\ba)), \chi_\e)$.  By Theorem \ref{howe}, if  $M(\ba,\chi_\e)\ne \{0\}$, then
$$
ch M(\ba,\chi_\e)=ch L^2(\l)ch F_B(\varepsilon(\ba))
$$
for some $\l\in(\h\cap\mathfrak s_2)^*$.

 Observe that $H$ acts on a homogeneous polynomial  $p$ as $H\cdot p=(-\rho_1(H)-deg(p))p$. Since
 $M(\ba,\chi_\e)= L^2(\l)\otimes F_B(\varepsilon(\ba))$, we see that we must have 
 $$
 \e=(-1)^{(\l+\varepsilon(\ba)+\rho_1) (H)}.
 $$
It is well known that 
$
ch L^2(\l)=\sum_{\mu}c_{\l,\mu}ch V^2(\mu),
$
where  $\l-\mu=\sum_{\a\in\D(C_n)} n_\a\a$. Since $\a(H)$ is even for any root in $\D(C_n)$ we deduce that
\begin{equation}\label{characteriso}
ch M(\ba,\chi_\e)=\sum_{\mu}  c_{\l,\mu}ch V^2(\mu)ch F_B(\varepsilon(\ba))
\end{equation}
with $\e=(-1)^{(\mu+\varepsilon(\ba)+\rho_1) (H)}$. 

Hence, if  we set 
\begin{equation}\label{wreg}
\mathcal W^{reg}_{n-d}=\{ w\in {\mathcal W_{n-d}}\mid \text{$w(-\rho_1+\mu(\ba)+\rho^C)$ is regular for $\D(A_{n-1})$}\},
\end{equation}
 we derive from \eqref{hower} that
\begin{align}\label{characterba}
&ch M(\ba,\chi_1)+ ch M(\ba,\chi_{-1})=\\&\sum\limits_{w\in \mathcal W^{reg}_{n-d}}c_wch V^2(\{w(-\rho_1+\mu(\ba)+\rho^C)\}-\rho^C)chF_B(\varepsilon(\ba)),\notag
\end{align}
where $c_w=c_{w(-\rho_1+\mu(\ba)+\rho^C)}$.

We need to compute 
$$
(-1)^{(\{w(-\rho_1+\mu(\ba)+\rho^C)\}-\rho^C+\varepsilon(\ba)+\rho_1)(H)}.
$$
For this observe that $\varepsilon(\ba)(H)=0$ and that $w(\mu(\ba))=\mu(\ba)$ for all $\ba\in \PP_d$ and $w\in {\mathcal W_{n-d}}$,
hence
\begin{align*}
(w(-\rho_1+&\mu(\ba)+\rho^C)-\rho^C+\varepsilon(\ba)+\rho_1)(H)\\
&=\mu(\ba)(H)+(-w(\rho_1)+\rho_1)(H)+(w(\rho^C)-\rho^C)(H)\\&\equiv \mu(\ba)(H)+(-w(\rho_1)+\rho_1)(H) \mod 2.
\end{align*}
Since $(-1)^{\mu(\ba)(H)}=\e(\ba)$, $\rho_1-s_{2\d_i}(\rho_1)=(2m+1)\d_i$, and $v(H)=H$ for any $v\in W(A_{n-1})$, we see that 
$$
(-1)^{(\{w(-\rho_1+\mu(\ba)+\rho^C\}-\rho^C+\varepsilon(\ba)+\rho_1)(H)}=\pm \e(\ba)\text{ if $w\in \mathcal W^\pm_{n-d}$}.
$$
If $w,w'\in {\mathcal W_{n-d}}$ with $w\ne w'$ then  $W(A_{n-1})w\cap W(A_{n-1})w'=\emptyset$, hence the set of  characters $\{ch V^2(\{w(-\rho_1+\mu(\ba)+\rho^C)\}-\rho^C)\mid w\in \mathcal W^{reg}_{n-d}\}$ is linearly independent. It follows that, if $M(F_B(\varepsilon(\ba)) ,\chi_{\e(\ba)})\ne \{0\}$, then, comparing \eqref{characteriso} with \eqref{characterba}, we find
\begin{align*}
 ch M&(F_B(\varepsilon(\ba)) ,\chi_{\e(\ba)})\\&=\sum\limits_{w\in \mathcal W^{reg}_{n-d}\cap \mathcal W^+_{n-d}}c_wch V^2(\{w(-\rho_1+\mu(\ba)+\rho^C)\}-\rho^C)chF_B(\varepsilon(\ba)).
\end{align*}
 Likewise, if $M(F_B(\varepsilon(\ba)) ,\chi_{-\e(\ba)})\ne \{0\}$, then 
 \begin{align*}
 ch M&(F_B(\varepsilon(\ba)) ,\chi_{-\e(\ba)})\\&=\sum\limits_{w\in \mathcal W^{reg}_{n-d}\cap \mathcal W^-_{n-d}}c_wch V^2(\{w(-\rho_1+\mu(\ba)+\rho^C)\}-\rho^C)chF_B(\varepsilon(\ba)).
\end{align*}
Since $ch F_A(\l)=0$ if $\l+\rho^A$ is singular, \eqref{isotypic} follows. 

Finally if $M(F_B(\varepsilon(\ba)) ,\chi_{\pm\e(\ba)})= \{0\}$, then the linear independence of the characters involved implies that $\mathcal W^{reg}_{n-d}\cap \mathcal W^\pm_{n-d}=\emptyset$, hence \eqref{isotypic} holds as well.
\end{proof}

Let $\PP^*_j$  denote the set  of all partitions with exactly $j$ parts and set
 \begin{equation}\label{P}\PP=\bigcup_{j=\max(0,m+1-(n-d))}^m \PP^*_j.
 \end{equation}
  For $\max(0,m+1-(n-d))\le j\le m$ and $\ba\in  \PP_j^*$, we set 
\begin{equation}\label{nu}
\nu(\ba)=-\sum_{r=n-d-m+j}^{n-j}\d_{r}-\sum_{r=n-j+1}^na_{n-r+1}\d_{r}.
\end{equation}
\begin{cor}[Theta correspondence] \label{cortheta}With  notation as in  Theorem \ref{howe}, we have that
$$
\Sigma=\{(F_B(\varepsilon(\ba)),\chi_\e(\ba))\mid \ba\in \PP_d\}\cup\{(F_B(\varepsilon(\ba)),\chi_{-\e(\ba)})\mid \ba\in \PP\} .
$$
Moreover 
$$
\tau(F_B(\varepsilon(\ba)),\chi_{\e(\ba)}))=L^2(-\rho_1+\mu(\ba))
$$ 
and, if $\ba\in \PP$, then
$$
\tau(F_B(\varepsilon(\ba)),\chi_{-\e(\ba)}))=L^2(-\rho_1+\nu(\ba)).
$$
\end{cor}
\begin{proof}
By Proposition \ref{theta}, $(F_B(\varepsilon(\ba),\chi_{\pm\e(\ba)})\in\Sigma$ if and only if $\mathcal W^\pm_{n-d}\cap \mathcal W^{reg}_{n-d}\ne \emptyset$. Clearly $1\in \mathcal W^+_{n-d}$ and $-\rho_1+\mu(\ba)=1(-\rho_1+\mu(\ba)+\rho^C)-\rho^C$ is dominant for $\Dp(A_{n-1})$, hence $-\rho_1+\mu(\ba)+\rho^A$ is regular. This implies that $(F_B(\varepsilon(\ba),\chi_{\e(\ba)})\in \Sigma$ for all $\ba\in \PP_d$.

We now show that, if $ \ba\in \PP$, then $(F_B(\varepsilon(\ba)),\chi_{-\e(\ba)}))\in \Sigma$. First observe that, if $\PP\ne\emptyset$, then $d=m<n$. Consider $s=s_{2\d_{n-d-m+j}}$. Then 
\begin{align*}
&s(-\rho_1+\mu(\ba)+\rho^C)=\\
&\sum_{i=1}^{n-d-m+j-1}(n-m-i+\frac{1}{2})\d_i+(-m-\frac{1}{2}+j)\d_{n-d-m+j}\\&+\sum_{i=n-d-m+j+1}^{n-j}(n-m-i+\frac{1}{2})\d_i+\sum_{i=n-j+1}^n(n-m-i+\frac{1}{2}-a_{n-i+1})\d_{i}.
\end{align*}
Let $\s$ be the permutation (in cycle notation) $\s=(n-j\ n-j-1\ \cdots\ n-d-m+j)$, and let $w$ be the element of $W(A_{n-1})$ such that $w(\d_i)=\d_{\s(i)}$. Then 
\begin{align*}
&ws(-\rho_1+\mu(\ba)+\rho^C)=\\
&\sum_{i=1}^{n-d-m+j-1}(n-m-i+\frac{1}{2})\d_i+\sum_{i=n-d-m+j}^{n-j}(n-m-i-\frac{1}{2})\d_i+\\
&\sum_{i=n-j+1}^n(n-m-i-\frac{1}{2}-(a_{n-i+1}-1))\d_{i}=-\rho_1+\nu(\ba)+\rho^C,
\end{align*}
and, since $a_r\ge1$ for $r\ge j$, we see that $ws(-\rho_1+\mu(\ba)+\rho^C)$ is regular for $\D(A_{n-1})$. It follows that $s\in \mathcal W^{reg}_{n-d}\cap \mathcal W^-_{n-d}$.

Suppose now that $\ba\in \PP_d\backslash \PP$. Then $a_r=0$ for $r> m-(n-d)$. In this case 
\begin{align*}
-\rho_1+\mu(\ba)+\rho^C&=\sum_{i=1}^{2(n-d)}(n-m-i+\frac{1}{2})\d_i\\&+\sum_{i=2(n-d)+1}^{n}(n-m-i+\frac{1}{2}-a_{n-i+1})\d_{i}.
\end{align*}
If $w\in {\mathcal W_{n-d}}$, $w\ne 1$, then write $w=s_{2\d_{i_1}}\dots s_{2\d_{i_k}}$ with $1\le i_1<\dots< i_k \le n-d$. It is clear that
\begin{align*}
&(w(-\rho_1+\mu(\ba)+\rho^C),\d_{i_1}-\d_{2(n-d)-i_1+1})=\\&(-\rho_1+\mu(\ba)+\rho^C,-\d_{i_1}-\d_{2(n-d)-i_1+1})=0,
\end{align*}
showing that $\mathcal W^{reg}_{n-d}=\{1\}$ in this case. This proves the first statement.

For the second statement recall that, if $\ba\in \PP_j^*$, then
$-\rho_1+\nu(\ba)=\{s(-\rho_1+\mu(\ba)+\rho^C)\}-\rho^C$. 
Note that, if $\tau(F_B(\varepsilon(\ba)),\chi_{\pm\e(\ba)})=L^2(\l)$ then, by \eqref{isotypic}, $\l(H)\ge(\{w(-\rho_1+\mu(\ba)+\rho^C)\}-\rho^C)(H)$ for all $w\in \mathcal W^\pm_{n-d}\cap \mathcal W^{reg}_{n-d}$. Hence if we show that, for $w\in \mathcal W^+_{n-d}\cap \mathcal W^{reg}_{n-d}$, $w\ne 1$, 
$$
(-\rho_1+\mu(\ba))(H)>(\{w(-\rho_1+\mu(\ba)+\rho^C)\}-\rho^C )(H),
$$ 
and that, for $\ba\in \PP^*_j$, $w\in \mathcal W^-_{n-d}\cap \mathcal W^{reg}_{n-d}$, $w\ne s_{2\d_{n-d-m+j}}$, 
$$
(-\rho_1+\nu(\ba))(H)>(\{w(-\rho_1+\mu(\ba)+\rho^C)\}-\rho^C )(H),
$$ 
we are done.

Since $v(H)=H$ if $v\in W(A_{n-1})$, it is enough to check that, for $w\in \mathcal W^+_{n-d}\cap \mathcal W^{reg}_{n-d}$, $w\ne 1$, 
\begin{equation}\label{rhopiumu}
(-\rho_1+\mu(\ba)+\rho^C)(H)>w(-\rho_1+\mu(\ba)+\rho^C)(H),
\end{equation}
and that, for $\ba\in \PP^*_j$, $w\in \mathcal W^-_{n-d}\cap \mathcal W^{reg}_{n-d}$, $w\ne s_{2\d_{n-d-m+j}}$, 
\begin{equation}\label{rhopiunu}
s_{2\d_{n-d-m+j}}(-\rho_1+\mu(\ba)+\rho^C)(H)>w(-\rho_1+\mu(\ba)+\rho^C)(H).
\end{equation}
If $w=s_{2\d_{i_1}}\dots s _{2\d_{i_k}}$, then $(\l,H-w(H))=(\l,-2\d_{i_1}-\dots-2\d_{i_k})$, hence
$$
(-\rho_1+\mu(\ba)+\rho^C)(H-w(H))=2(n-d-i_1+\frac{1}{2})+\dots+2(n-d-i_k+\frac{1}{2})>0,
$$ 
so, using the fact that $w=w^{-1}$, we obtain \eqref{rhopiumu}. To obtain \eqref{rhopiunu} we set $j_0=n-d-m+j$ and observe that, since $\ba\in \PP^*_j$, if $j_0< i_k\le n-d$, then $w(-\rho_1+\mu(\ba)+\rho^C)$ is singular for $\D(A_{n-1})$. We can therefore assume that $i_k\le j_0$. In this case we need to evaluate $(-\rho_1+\mu(\ba)+\rho^C)(s_{2\d_{j_0}}(H)-w(H))$, which is equal to 
$$
(-\rho_1+\mu(\ba)+\rho^C,2\d_{j_0}-2\d_{i_1}-\dots-2\d_{i_k}).
$$
If $i_k<j_0$ then this is 
$2(n-d-i_1+\frac{1}{2})+\dots+2(n-d-i_k+\frac{1}{2})-2(n-d-j_0+\frac{1}{2})>0$. If $i_k=j_0$ then, since $w\ne s_{2\d_{j_0}}$, we have $k>1$ so $(-\rho_1+\mu(\ba)+\rho^C)(s_{2\d_{j_0}}(H)-w(H))$ is equal to 
$2(n-d-i_1+\frac{1}{2})+\dots+2(n-d-i_{k-1}+\frac{1}{2})>0$.
The proof is now complete.

\end{proof}

We conclude this section by observing that, as a byproduct, we have computed the character of $\tau(\eta)$ for $\eta\in\Sigma$. In fact, combining Corollary \ref{cortheta} with Proposition \ref{theta}, we find (cf. \eqref{wreg} for notation)
\begin{equation}\label{charactermu}
ch(L^2(-\rho_1+\mu(\ba))=\sum_{w\in \mathcal W^{reg}_{n-d}\cap \mathcal W^+_{n-d}}c_wch V^2(\{w(-\rho_1+\mu(\ba)+\rho^C)\}-\rho^C)
\end{equation}
and
\begin{equation}\label{characternu}
ch(L^2(-\rho_1+\nu(\ba))=\sum_{w\in \mathcal W^{reg}_{n-d}\cap \mathcal W^-_{n-d}}c_wch V^2(\{w(-\rho_1+\mu(\ba)+\rho^C)\}-\rho^C),
\end{equation}
where $c_w=c_{w(-\rho_1+\mu(\ba)+\rho^C)}$.

 If $\l\in(\h\cap \mathfrak s_2^\C)^* $,  let $W_{\lambda}$ be the subgroup of $W(C_n)$ defined by Enright in \cite[Definition 2.1]{En}. In our context, $W_{\l}$ is the subgroup, generated by reflections in roots  $\a=\d_i+\d_j$, such that:
\begin{enumerate}
\item[(i)] $(\l,\a^\vee)\in \mathbb Z^{>0}$;
\item[(ii)] $\be\in\D,\, (\l,\be)=0\implies (\a,\be)=0$.
\end{enumerate}
Following \cite[Definition 2.1]{En}, we introduce the root system $\D_{\l}$ as the set of the roots $\a$ such that $s_\a\in W_\l$. Set $\D_\l(A_{n-1})=\D_\l\cap \D(A_{n-1})$, $\Dp_\l=\Dp\cap\D_\l$. Let also  $W_\l(A_{n-1})$ be the Weyl group of $\D_\l(A_{n-1})$. We denote by $\ell_\l$ the length function in $W_\l$ with respect to $\Dp_\l$ and let $W^A_\l$ be the set of minimal length coset representatives for $W_\l(A_{n-1})\backslash W_\l$.
 In \cite[Corollary 2.3]{En} Enright proved a character formula that holds for any unitary highest weight module of a classical real reductive group.  This 
 formula is also implied by a stronger result proved in \cite{EW}. In the following we show how Enright's formula for $L^2(-\rho_1+\mu(\ba))$ and $L^2(-\rho_1+\nu(\ba))$ can be derived combinatorially from \eqref{charactermu} and \eqref{characternu} respectively.

\begin{cor}\label{enrighteven}If $\ba\in \PP_d$ and $\l_0=-\rho_1+\mu(\ba)+\rho^C$, then
$$
ch(L^2(-\rho_1+\mu(\ba)))=\sum_{w\in W_{\l_0}^A}(-1)^{\ell_{\l_0}(w)}V^2(\{w(\l_0)\}-\rho^C).
$$
\end{cor}
\begin{proof}
Fix $\mathbf a \in A$ and consider $\l_0$, whose coordinates in the basis
$\{\d_i\}$ are of the form 
\begin{align*}
v&=(v_1,\ldots,v_n)=(n-m+\tfrac{1}{2},n-m-\tfrac{1}{2},\ldots,\tfrac{1}{2},-\tfrac{1}{2}-b_1,\ldots,-\tfrac{1}{2}-b_m),
\end{align*}
where $b_i=a_{m-i+1}+i-1$.
We say that a positive entry $a$  in $v$ is singular or regular according to whether   $-a$ appears in $v$ or not. The group $W_{\l_0}$ acts on $\l_0$ as follows: it  is the identity if the number of regular entries is less  or equal than one; if there are exactly two regular entries, then the only nontrivial action is given by exchanging positions and changing signs of both regular entries; if there are more than two regular entries, it acts   by an even number  of sign changes of the regular entries followed by a permutation of the result. All other entries are fixed. 

If $x\in  \mathcal W^{reg}_{n-d}$, then it acts on $\l_0$ by changing signs of some regular entries. Let $w_x$ be the permutation that arranges the result in decreasing order. Since the set $\mathcal W^+_{n-d}\cap \mathcal W^{reg}_{n-d}$ is a set of coset representatives for  $W_{\l_0}(A_{n-1})$ in $W_{\l_0}$, we see that $W_{\l_0}^A=\{w_x x\mid x\in \mathcal W^+_{n-d}\cap \mathcal W^{reg}_{n-d}\}$ (notice that this is true also when there are only two regular entries). Let $w'_x$ be the unique element of $W(A_{n-1})$ such that $w'_x w_x x(\l_0)$ is $\Dp(A_{n-1})$-dominant.
Since $sgn(x)=1$, it follows from \eqref{charactermu} that
$$
ch(L^2(-\rho_1+\mu(\ba)))=\sum_{x\in \mathcal W^+_{n-d}\cap \mathcal W^{reg}_{n-d}}sgn(w'_xw_xx)V^2(\{w_xx(\l_0)\}-\rho^C).
$$
To obtain our result we need only to prove that $sgn(w'_x)=1$.
We shall prove this claim by induction on $n-m$. Let $s$ be the number of singular entries in $v$ and let $t\geq 0$ be maximal for which
$v_{n-m-t},\ldots,v_{n-m}$ are singular. Set $v'=w_x x(v)$. We want to prove that arranging $v'$ in decreasing order involves  an even number of simple transpositions.  If $x$ does not change the sign to  $v_1$, then induction applies in a straightforward way. Otherwise we have  $v'=(v'_1,\ldots,-v_1,v_{n-m-t},\ldots,v_{n-m},\ldots)$. In making $v'$ dominant for $A_{n-1}$ we should pass $-v_1$ through the $t$ singular values 
 and the $s$ negative values corresponding to negatives of  the singular entries. So we obtain after  $s+t$ simple transpositions 
 the element $$v''=(v'_1,\ldots,v_{n-m-t},\ldots,v_{n-m},\underbrace{\dots}_{\text{$s$ entries}},-v_1,\ldots\ldots).
 $$
  Consider now the vector
$$u=(u_1,\ldots,u_n)=(n-m-\tfrac{1}{2},n-m-\tfrac{3}{2},\ldots,\tfrac{1}{2},-\tfrac{1}{2}-b'_1,\ldots,-\tfrac{1}{2}-b'_{m+1})$$
with $b_i=b'_i$ for $i=1,\dots, s$, $b'_{s+1}=-x_1$, and $b'_i=b_{i-1}$ for $i=s+2\dots , m+1$. Let $y=s_{2\d_1}x$ and set $u'=w_{y}(u)$. Observe that $v''$ is obtained from $u'$ by performing 
exactly $s-t$ simple transpositions. Hence $v'$ differs from  $u'$ by an even number ($2s$) of simple transpositions, and we can apply induction.
\end{proof}

\begin{cor}If $\ba\in \PP$ and  $\l_1=-\rho_1+\nu(\ba)+\rho^C$, then
$$
ch(L^2(-\rho_1+\nu(\ba)))=\sum_{w\in W_{\l_1}^A}(-1)^{\ell_{\l_1}(w)}V^2(\{w(\l_1)\}-\rho^C).
$$
\end{cor}
\begin{proof}
Assume $\ba\in A_j^+$ and let $\s$ be the permutation described in the proof of Corollary \ref{cortheta}. Let $w$ be the element of $W_{A_{n-1}}$ such that $w(\d_i)=\d_{\s(i)}$ and set $s=s_{2\d_{n-d-m+j}}$. Recall that 
$
ws(\l_0)=\l_1$.

We now show that $W_{\l_1}=wW_{\l_0} w^{-1}$. For this it is enough to check that $\a$ satisfies conditions (i) and (ii) for $\l_0$ if and only if $w(\a)$ satisfies both conditions for $\l_1$. 
This is clear if $(\a,\delta_{n-d-m+j})=0$, so that $s(\a)=\a$. If $\a=  \delta_{n-d-m+j}+\delta_i$, with $n-d-m+j<i\le n-j$, then  $\a$ does not satisfy condition (ii) and, if $i>n-j$ then it does not satisfy condition (i). On the other hand $(w(\a)^\vee,ws(\l_0))=(s(\a)^\vee,\l_0)<0$ so $w(\a)$ does not satisfy condition (i). 
It remains to check the case when $\a = \delta_i + \delta_{n-d-m+j}$ with $i< n-d-m+j$. In this case one checks readily that  $\a$ satisfies condition (i). On the other hand $s(\a)\in\D^+(A_{n-1})$, so $(s(\a)^\vee,\l_0))>0$, hence $(w(\a),ws(\l_0))>0$. If $\be$ is a root such that $(\be, \l_1)=0$, then $(s w^{-1}(\be),\l_0)=0$, so $(\a,s w^{-1}(\be))=0$. If $(\a,w^{-1}(\be))\ne 0 $ the only possibility is that $w^{-1}(\be)=\pm\a$, but this implies $(s(\a),\l_0)=0$. This contradiction implies $(\a,w^{-1}(\be))=  (w(\a),\be)=0$. This shows that if $\a$ satisfies condition (ii) then also $w(\a)$ does. Reversing this argument we obtain that if $w(\a)$ satisfies condition (ii), then also $\a$ does. This proves our claim.

Notice now that $w(\Dp_{\l_0})=\Dp_{\l_1}$, thus $W_{\l_1}^A=wW_{\l_0}^Aw^{-1}$.
By applying \eqref{characternu} we can write
$$
ch(L^2(-\rho_1+\nu(\ba))=\!\!\!\!\sum_{x\in \mathcal W^{reg}_{n-d}\cap \mathcal W^-_{n-d}}\!\!\!\!sgn(w'_xw_x)ch V^2(\{ww_{xs}xsw^{-1}(\l_1)\}-\rho^C).
$$
As shown in Corollary \ref{enrighteven}, $sgn(w'_x)=1$, so we need only to show that $sgn(w_x)=sgn(ww_{xs}xsw^{-1})$, or, equivalently, that $sgn(w_x)=sgn(w_{xs})$. Recall that, for $y\in \mathcal W^{reg}$, $w_y$ is the unique element of $W_{\l_0}(A_{n-1})$ such that $w_y y(\l_0)$ is dominant for $\Dp_{\l_0}(A_{n-1})$. Since $s(\l_0)$ is dominant for $\Dp_{\l_0}$, it follows that $w_x=w_{xs}$ and we are done.
\end{proof}

\begin{rem}
It should be noted that, even though the proofs of Proposition \ref{theta} and its corollaries do not use classical invariant theory, they depend on Theorem \ref{howe}, which uses in a crucial way the first fundamental theorem of classical invariant theory for $O(2m+1)$.
\end{rem}

\section{Theta correspondence for the pair $(Sp(n),SO^*(2m))$}\label{7}

In this section we use the denominator identity developed in the previous sections to derive the Theta correspondence for the  compact dual pair $(Sp(n),SO^*(2m))$. This dual pair, according to Proposition \ref{B}, corresponds to  the distinguished sets of positive roots $\Dp_{D2}$ and $\Dp_{D2'}$ in a superalgebra of type $D(m,n)$. 

We will develop the theory only  for $\Dp_{D2}$:  the formulas corresponding to $\Dp_{D2'}$ are obtained by simply applying the reflection $s_{\e_m}$ to the formulas corresponding to $\Dp_{D2}$.

 We have, for $1\leq i\ne j\leq m,\,1\leq k\ne l\leq n$, 
\begin{alignat}{2}\label{positivoD0}
&\Dp_0= & &\{\e_i+\e_j,\d_k+\d_l,2\d_k\mid 1\leq i\ne j\leq m,\,1\leq k\ne l\leq n
\}\cup\\\notag
&   & &\{\e_i-\e_j,\d_k-\d_l,\mid 1\leq i< j\leq m,\,1\leq k< l\leq n \},\\
 &\Dp_1=& &\{\e_i \pm\d_k \mid 1\leq i\leq m,\,1\leq k \leq n\}.\notag
\end{alignat}
We have
$
2\rho_1=(2n)(\e_1+\ldots+\e_m)
$
and, as in type $B(m,n)$, the defect of $\g$ is $d=\min(m,n)$. Set also 
\begin{align*}
\D(D_r)&=\pm\{\e_i\pm\e_j\mid m-r+1\le i\ne j\le m\},\\
\D(C_r)&=\pm\{\d_k\pm\d_l,2\d_k\mid 1\le k\ne l\le r\},\\
\D(A_{r-1})&=\pm\{\e_k-\e_l\mid m-r+1\leq k\ne l\leq m\}.\end{align*}
Let $W(A_{r-1}), W(D_r), W(C_r)$ be the Weyl groups of $\D(A_{r-1}), \D(D_r), \D(C_r)$, respectively. Set $\Dp(C_n)=\Dp_0\cap\D(C_n)$, $\Dp(D_m)=\Dp_0\cap\D(D_m)$, and $\Dp(A_{m-1})=\Dp_0\cap\D(A_{m-1})$. We denote by $\rho^C$, $\rho^D$, $\rho^A$ the corresponding $\rho$-vectors.
Set
\begin{equation}\label{wr}{\mathcal W_r}= \text{subgroup of $W_\g$ generated by $\{s_{\e_i}s_{\e_j}\mid 1\leq i<j\leq  r\}$}.\end{equation}
\par 
With notation as in  Section  \ref{ED}, the total order corresponding to this choice of $\Dp$ is 
$$
\e_1>\dots >\e_m>\d_1>\dots>\d_n.
$$
There is only one arc diagram associated to this total order, namely $$X=\{\stackrel{\frown}{\e_{m} \d_1},\dots, \stackrel{\frown}{\e_{m-d+1}\d_d}\}.$$ 
The corresponding maximal isotropic set of roots is $S(X)=\{\gamma_1,\dots,\gamma_d\}$, where
\begin{equation}
\gamma_1=\e_{m}-\d_1,\,\gamma_2=\e_{m-1}-\d_2,\ldots,\gamma_d=\e_{m-d+1}-\d_d.
\end{equation}
We want to apply Proposition \ref{migliore} with $\mathcal B'=Supp(X)$.
 Note that, with notation as in \ref{mmm}, $\D(\mathcal B')=\D(D_{d})\times \D(C_d)$, and we have that  $\D^\sharp(\mathcal B')=\D(C_d)$, so that 
$W_{\mathcal B'}W^\sharp(\mathcal B')=W(A_{d-1})W(C_d)$. 
With notation as in \eqref{wr}, set $$Z_0=(W(A_{n-1})/W(A_{d-1}) ){\mathcal W_{n-d}}(W(C_m)/W(C_d))$$ and, if $m>n$,  $$Z_1=(W(A_{n-1})/W(A_{d-1}) ){\mathcal W_{m-d}}s_{\e_{m-d}}s_{\e_{m-d+1}}(W(C_m)/W(C_d)).$$ Then we can choose $Z=Z_0$ if $m\le n$ and $Z=Z_0\cup Z_1$ otherwise.

  We define
$$
W=W(A_{m-1}) {\mathcal W_{m-d}}W(C_n)
$$
so that we have $W_0=W$ if $m\le n$ and while
$W_0=W\cup W_1$ with
$$
W_1=(W(A_{m-1})/W(A_{d-1})) {\mathcal W_{m-d}}s_{\e_{m-d}}s_{\e_{m-d+1}}W(A_{d-1})W(C_n)$$
 otherwise.

If $m\le n$, by Proposition \ref{migliore}, we have
\begin{align}
\label{secondaDtipodue}
e^\rho \check R&= \check\cF_{W}\bigl(\frac{e^{\rho
}}{\prod_{j=1}^d(1-e^{-\llbracket\gamma_j\rrbracket})}\bigr).
\end{align}
If $m>n$, Proposition \ref{migliore} gives 
\begin{align*}
2e^\rho \check R&= \check\cF_{W}\bigl(\frac{e^{\rho
}}{\prod_{j=1}^d(1-e^{-\llbracket\gamma_j\rrbracket})}\bigr)+\check\cF_{W_1}\bigl(\frac{e^{\rho
}}{\prod_{j=1}^d(1-e^{-\llbracket\gamma_j\rrbracket})}\bigr).
\end{align*}

\begin{lemma}If $m>n$ then
\begin{equation}\label{WequalWuno}
 \check\cF_{W}\bigl(\frac{e^{\rho
}}{\prod_{j=1}^d(1-e^{-\llbracket\gamma_j\rrbracket})}\bigr)=\check\cF_{W_1}\bigl(\frac{e^{\rho
}}{\prod_{j=1}^d(1-e^{-\llbracket\gamma_j\rrbracket})}\bigr).
\end{equation}
In particular \eqref{secondaDtipodue} holds for any $m,n$.
\end{lemma}
\begin{proof}
We need only to check that $
 \check\cF_{W}(\cP(X))=\check\cF_{W_1}(\cP(X))$. Since right multiplication  by $W_{Supp(X)}$ stabilizes both $W$ and $W_1$, we can apply Corollary \ref{cor1} to both sides of \eqref{WequalWuno}. Since odd reflections only change the sign of $\cP(X)$, by applying a series of odd and interval reflection to both sides of \eqref{WequalWuno}, we are reduced to checking that $
 \check\cF_{W}(\cP(X'))=\check\cF_{W_1}(\cP(X'))$, where $X'$ is the arc diagram  
 $
 X'=\{
\stackrel{\frown}{\e_{m-n+i}\d_i}\mid i=1,\dots n\}$ whose underlying order on $\mathcal B$ is $\e_1>\dots>\e_{m-d+1}>\d_1>\e_{m-d+2}>\d_2>\dots>\e_m>\d_n$. Let $Y'$ be the arc diagram $X'$ seen as a diagram on $\mathcal B''=\{\e_{m-d+i}\mid 1\le i\le d\}\cup\{\d_i\mid i=1,\dots, d\}$. Since all simple roots of $\mathcal B''$ are isotropic, we have $(\rho_{X'},\a)=(\rho_{Y'},\a)=0$ for all $\a\in\D(\mathcal B'')$ so
$$
 \check\cF_{W}(\cP(X))=\check\cF_{(W(A_{m-1})/W(A_{d-1})) {\mathcal W_{m-d}}}(e^{\rho_{X'}-\rho_{Y'}}\check\cF_{W(A_{d-1})W(C_d)}(\cP(Y'))
$$
Since all simple roots of $\mathcal B''\cup\{\e_{m-d}\}$ are all isotropic, we see that $(\rho_{X'},\e_{m-d})=0$. Clearly $(\rho_{Y'},\e_{m-d})=(\gamma,\e_{m-d})=0$ for all $\gamma\in S(Y')$. Thus, setting $s=s_{\e_{m-d}}s_{\e_{m-d+1}}$, we have
$$
 \check\cF_{W_1}(\cP(X))=\check\cF_{(W(A_{m-1})/W(A_{d-1})) {\mathcal W_{m-d}}}(e^{\rho_{X'}-\rho_{Y'}}\check\cF_{sW(A_{d-1})W(C_d)}(\cP(Y')),
$$
hence it is enough to check that
$\check\cF_{W(A_{d-1})W(C_d)}(\cP(Y'))=\check\cF_{sW(A_{d-1})W(C_d)}(\cP(Y'))$. 
It is enough to verify  the $s_{\epsilon_{m-d+1}}$-invariance.
However,
$$\check{\cF}_{W(A_{d-1})W(C_d)}(\cP(Y'))=d! \check{R}_{Y'} e^{\rho_{Y'}}$$
and this, up to a sign, does not depend on the choice of simple roots for $\Delta(Y')$.
Taking a system of simple roots containing $\delta_1\pm \epsilon_{m-d+1}$,
we realize $s_{\epsilon_{m-d+1}}$ as an automorphism of the corresponding Dynkin diagram
so $\check{R}_{Y'} e^{\rho_{Y'}}$ is $s_{\epsilon_{m-d+1}}$-invariant.

\end{proof}
\vskip5pt
For $\ba\in \PP_d$, define
$$
\mu(\ba)= -\sum_{r=1}^d(a_{d+1-r})\e_{m-d+r},\quad
\varepsilon(\ba)=\sum_{r=1}^m a_r \d_r.
$$

If $\l\in span(\e_i)$, set 
\begin{equation}
ch F_A(\l)= \sum_{w\in W(A_{m-1})}sgn(w)\frac{e^{w(\l+\rho^A)-\rho^A}}{\Pi_{\a\in\Dp(A_{m-1})}(1-e^{-\a})}.
\end{equation}
Arguing as in the previous section, it follows from  \eqref{secondaDtipodue} that 
\begin{align}\label{formulaperDsecondo}
chM^{\Dp}(\g_1)=\sum_{\mathbf a\in \PP_d} \sum\limits_{w\in {\mathcal W_{m-d}}} \frac{ch F_A(w(-\rho_1+\mu(\ba)+\rho^D)-\rho^D)}{\prod_{\a\in\Dp(D_m)\backslash\Dp(A_{m-1})}(1-e^{-\a})}chF_C(\varepsilon(\ba)),
\end{align}
$F_C(\varepsilon(\ba))$ being the finite-dimensional irreducible $sp(n)$-module with highest weight $\varepsilon(\ba)$.

 Recall from section \ref{5} that we constructed a real symplectic subspace $V$ of $\g_1$ and a map $\Phi:V\to (\H^n)^*\otimes_\H\H^m$, such that there is a dual pair $(G_1,G_2)$  in $sp(V,\langle\cdot\,,\cdot\rangle)$ having the properties described in Proposition \ref{B}.
 We look upon  $Sp(n)$ as a subgroup of $Sp((\H^n)^*\otimes_\H
\H^m)$ via its action on the first factor. Then $G_1=\{\Phi^{-1}g \Phi\mid g\in  Sp(n)\}$. It follows that, if  $\Phi^{-1}g \Phi\in G_1$, then, since $Sp(n)$ is compact, connected and simple,
 $$
 det_W(\Phi^{-1}g \Phi)=1.
 $$
 Therefore
 $$ \widetilde G_1=G_1\times\{\pm 1\}.$$
 
 Given a representation $(\pi,V_\pi)$ of $G_1=\Phi^{-1}Sp(n) \Phi$, we let $(\tilde \pi,V_\pi)$ be the representation of $\tilde G_1$ on the same representation space such that 
\begin{equation}\label{definizionetildaD}
\tilde \pi(g_1,z)( v)=z^{-1}\pi(g_1)( v).
\end{equation}
 If $\Omega$ is the set of finite-dimensional irreducible representations of $G_1$, then we know from \eqref{actionk} that the set $\Sigma$, occurring in Theorem \ref{howe}, is a subset of $\tilde \Omega=\{\tilde \pi\mid \pi\in \Omega\}$.
  
 \begin{prop}\label{thetaD2} With  notation as in  Theorem \ref{howe}, we have that, if $\tilde \pi\in \Sigma$, then there is $\ba\in \PP_d$  such that 
 $$\pi=F_C(\varepsilon(\ba)).
 $$
Furthermore, the $\h$-character of the isotypic  component of $\tilde F_C(\varepsilon(\ba))$ in  $M^{\Dp}(\g_1)$ is
\begin{equation}
  \sum\limits_{w\in {\mathcal W_{m-d}}}\frac{ch F_A(w(\rho^D-\rho_1+\mu(\ba))-\rho^D)}{\prod_{\a\in\Dp(D_m)\backslash\Dp(A_{m-1})}(1-e^{-\a})}chF_C(\varepsilon(\ba)).
\end{equation}
 \end{prop}
 \begin{proof} Clear from \eqref{formulaperDsecondo}.
\end{proof}

In the case in question the set of roots of $\mathfrak{s}_2^\C$ is $\D(D_m)$. In the identification of $\h$ with $\h^*$, given by $(\ , \ )$, we see that the element $H\in \h$ that corresponds to $\sum_{i=1}^m\e_i$ has the property that $H_{|V_\C^\pm}=\pm I$. Thus the parabolic subalgebra $\p_2$, defined by $H$, is
$$
\p_2=\h\oplus \sum_{\a\in\D(A_{m-1})}(\g_0)_\a\oplus \n,
$$
where 
$
\n=\sum\limits_{\a\in\Dp(D_{m})\backslash \D(A_{m-1})}(\g_0)_\a
$ is the nilradical.


We use the notation \ref{nota}. Again  we remark  that
$$
ch V^2(\l)=c_{\l+\rho^A}\frac{ch F_A(\l)}{\prod_{\a\in\Dp(D_m)\backslash\Dp(A_{m-1})}(1-e^{-\a})}.
$$

\begin{cor}[Theta correspondence] \label{corthetaD2} With  notation as in  Theorem \ref{howe}, we have that
$$
\Sigma=\{\tilde F_C(\varepsilon(\ba))\mid \ba\in \PP_d\}$$
Moreover, if $\ba\in  \PP_d$, then
$$
\tau(\tilde F_C(\varepsilon(\ba)))=L^2(-\rho_1+\mu(\ba)).
$$ 
\end{cor}
\begin{proof}
By Proposition \ref{thetaD2}, we need only to check that $\tilde F_C(\varepsilon(\ba))\in\Sigma$ for all $\ba\in \PP_d$.  Clearly $1\in {\mathcal W_{m-d}}$ and $-\rho_1+\mu(\ba)=1(-\rho_1+\mu(\ba)+\rho^D)-\rho^D$ is dominant for $\Dp(A_{m-1})$, hence $-\rho_1+\mu(\ba)+\rho^A$ is regular. By Proposition \ref{thetaD2}, this implies that  the isotypic component of $\tilde F_C(\varepsilon(\ba))$ has a nonzero  character, thus $\tilde F_C(\varepsilon(\ba))$ occurs in $\Sigma$ as wished.

For the second statement, arguing as in the proof of Corollary \ref{cortheta}, we need only to show that, for $w\in {\mathcal W_{m-d}}$ with $w(-\rho_1+\mu(\ba)+\rho^D)$ regular for $\Dp(A_{m-1})$ and $w\ne 1$, we have that
$$
(-\rho_1+\mu(\ba))(H)>(\{w(-\rho_1+\mu(\ba)+\rho^D)\}-\rho^D )(H).
$$

Since $v(H)=H$ if $v\in W(A_{m-1})$, it is enough to check that, for $w\in {\mathcal W_{m-d}}$, $w\ne 1$, 
$$
(-\rho_1+\mu(\ba)+\rho^D)(H)>w(-\rho_1+\mu(\ba)+\rho^D)(H).
$$
This is clear because,
if $w=s_{\e_{i_1}}\dots s _{\e_{i_k}}$, then $(\l,H-w^{-1}(H))=(\l,2\e_{i_1}+\dots+2\e_{i_k})$, hence
$$
(-\rho_1+\mu(\ba)+\rho^D)(H-w^{-1}(H))=2(m-d-i_1)+\dots+2(m-d-i_k)>0.
$$ 
The proof is now complete.
\end{proof}

\noindent As already observed in Section \ref{Bnm}, we have also computed the character of $\tau(\tilde F_C(\varepsilon(\ba))$.  In fact, letting $\mathcal W^{reg}_{m-d}$ be the set of $w\in {\mathcal W_{m-d}}$ such that $w(-\rho_1+\mu(\ba)+\rho^D)$ is regular for $\D(A_{m-1})$, then, combining Corollary \ref{corthetaD2} with Proposition \ref{thetaD2}, we find
\begin{equation}\label{charactermuD2}
ch(L^2(-\rho_1+\mu(\ba))=\sum_{w\in \mathcal W^{reg}_{m-d}}c_wsgn(w)ch V^2(\{w(-\rho_1+\mu(\ba)+\rho^D)\}-\rho^D),
\end{equation}
where $c_w=c_{w(-\rho_1+\mu(\ba)+\rho^D)}$.

The argument given in  Corollary \ref{enrighteven} works also in the present case thus, with notation as in  Section \ref{Bnm}, we obtain

\begin{cor}\label{enrightevenD}If $\ba\in \PP_d$ and $\l_0=-\rho_1+\mu(\ba)+\rho^D$, then
$$
ch(L^2(-\rho_1+\mu(\ba)))=\sum_{w\in W_{\l_0}^A}(-1)^{\ell_{\l_0}(w)}V^2(\{w(\l_0)\}-\rho^D).
$$
\end{cor}

\section{Theta correspondence for the pair $(O(2m),Sp(2n,\R))$}

This is the most difficult case. Complications arise from the representation theory of even orthogonal groups, which we now discuss. Let $V$ be the real symplectic subspace  of $\g_1$ and $\Phi:V\to \R^{2m}\otimes (\R^{2n})^*$ be the map such that there is a dual pair $(G_1,G_2)$  in $sp(V,\langle\cdot\,,\cdot\rangle)$ having the properties described in Proposition \ref{B}.
We first parametrize the representations  of $\widetilde G_1$: the same argument given in Section \ref{Bnm} implies that $\widetilde G_1$ is the group of pairs $(\Phi^{-1}g \Phi,z)$ with $g\in O(2m)$ and $z^2=(det_{\R^{2m}}(g))^n$. If $\Omega$ is the set of finite-dimensional irreducible representations of $G_1$, then, by \eqref{actionk}, we know that 
 $\Sigma\subset \{\tilde \pi\mid \pi\in\Omega\}$ where $\tilde \pi$ is as in \eqref{definizionetildaD}.

 To describe $\Omega$ we need to recall that $G_1\simeq O(2m)$ is isomorphic to the semidirect product of $\ganz/2\ganz$ with $SO(2m)$. The generator $x$ of $\ganz/2\ganz$ in $G_1$ can be chosen so that  $Ad(x)$ induces on $\h\cap \mathfrak s_1^\C$ the reflection $s_{\e_m}$. The finite-dimensional irreducible representations of $G_1$ are determined by their restriction to the  connected component of the identity  $G_1^0\simeq SO(2m)$ of $G_1$ and by the action of $x$. First of all, we parametrize the irreducible finite-dimensional representations of $G_1^0$ as irreducible finite-dimensional representations of $\mathfrak s_1^\C$:  having chosen $\h\cap \mathfrak s_1^\C$ as a Cartan subalgebra, we choose $\Dp_0\cap\D(\mathfrak s_1^\C)$ as a  set of positive roots for $\mathfrak s_1^\C$, where $\Dp_0$ is as in \eqref{positivoD0}, and let $\rho^D$ be the corresponding $\rho$-vector. Given $\l\in (\h\cap \mathfrak s_1^\C)^*$ dominant integral, we let $F_D(\l)$ be the irreducible finite-dimensional $\mathfrak s_1^\C$-module with highest weight $\l$.
 
 If $\la=\sum_{i=1}^n a_i\e_i$ with  $a_i\in\ganz$ and $a_1\ge a_2\ge\dots\ge a_m> 0$, then there is a unique irreducible representation of $G_1$ that restricted to $G_1^0$ contains $F_D(\l)$. We denote this representation by $F^+(\l)$; it can be checked that the restriction of $F^+(\l)$ to $G_1^0$ is 
 $$F_D^+(\l)=F_D(\l)\oplus F_D(s_{\e_m}(\l)).
 $$
  If instead $a_m=0$, then there are two irreducible representations of $G_1$ whose restriction to $G_1^0$ contains $F_D(\l)$. Both restrict to $G_1^0$ as  $F_D(\l)$, and we let $F^+(\l)$  be the one such that $x$ acts trivially on the highest weight vector, and $F^-(\l)$ the one such that $x$ acts by $det_{\R^{2m}}(x)=-1$. Note that this parametrization depends on the choice of $x$.

We also need Kostant's generalization to nonconnected groups of the Weyl character formula, given in \cite[Theorem 7.5]{Kostant}. In the particular case at hand, we let $T$ be the torus in $G_1$ such that the Lie algebra of $T$ is contained in  $\h$. Set $H^+=T\cup xT$. The group $H^+$ is the normalizer of the torus $T$ and of the Borel subgroup of $G_1^0$ corresponding to our choice of positive roots. Then Kostant's formula in this case reads
\begin{equation}\label{a++}
ch_{ F^+(\l)}(x^st)=\frac{1+(-1)^s}{2}ch F^+_D(\l)(t)
\end{equation}
if $\l=\sum_{i=1}^ma_i\e_i$ with $a_m>0$, while, if $a_m=0$,
\begin{align}\label{azero}
ch_{ F^\pm(\l)}(x^st)&=\frac{1+(-1)^s}{2}ch F_D(\l)(t)\\
&\pm \frac{1-(-1)^s}{2}\frac{\sum_{w\in W(C_{m-1})}sgn(w)e^{w(\l+\rho^D)}}{\sum_{w\in W(C_{m-1})}sgn(w)e^{w(\rho^D)}}(t),\notag
\end{align}
where $W(C_{m-1})$ is the Weyl group of the root system 
\begin{align*}\D(C_{m-1})&=\{\pm(\e_i-\e_j)\mid 1\le i<j\le m-1\}\\
&\cup\{\pm(\e_i+\e_j)\mid 1\le i\le j\le m-1\}.
\end{align*}
Note that, since $\rho^D=\sum_{i=1}^{m-1}(m-i)\e_i$,
$$
\sum_{w\in W(C_{m-1})}sgn(w)e^{w(\rho^D)}=e^{\rho^D}\prod_{\a\in \Dp(C_{m-1})}(1-e^{-\a}),
$$
where 
\begin{align*}\Dp(C_{m-1})&=\{(\e_i-\e_j)\mid 1\le i<j\le m-1\}\\
&\cup\{(\e_i+\e_j)\mid 1\le i\le j\le m-1\}.
\end{align*}
\vskip10 pt
Set $\h_2=\h\cap s_2^\C$. 
We need to compute the $H^+\times \h_2$-character of $M^{\Dp}(\g_1)$. 
In this case
$$
\Dp_1=\{\d_k\pm\e_i\mid 1\leq i\leq m,\,1\leq k \leq n\},
$$
and 
$
2\rho_1=(2m)(\d_1+\ldots+\d_n).
$

Let $ad$ be the adjoint action of $\g_0$ on $\g_1$. 
Recall from Section \ref{5} that, as $\tilde G_1$-module, $M^{\Dp}(\g_1)$ is given, up to the action of the center, by the action of $G_1$ on a polynomial algebra $P(W)$. The differential of this action on $W$ coincides with $ad^*_{|s_1^\C}$ on $(\g_1^-)^*$. Hence we can identify $P(W)$ with the symmetric algebra $S(\g_1^-)$ with $s_1^\C$ acting by $ad$. Let us denote by $Ad$ the action of $G_1$ on $S(\g_1^-)$ coming from its action on $P(W)$.
We normalize the choice of $x$ by assuming that we can choose the root vectors $X_{-\d_i\pm\e_j}$ in $\g_1^-$ in such a way that 
$$
Ad(x)(X_{-\d_i\pm \e_j})=X_{-\d_i\pm\e_j}
$$
if $j<m$, and
$$
Ad(x)(X_{-\d_i+ \e_m})=X_{-\d_i-\e_m}.
$$
Clearly, if $t\in T\times \h_2$, then 
$$ch_{H^+\times \h_2} M^{\Dp}(\g_1)(t)=\frac{e^{-\rho_1}}{\prod_{\a\in \Dp_1}(1-e^{-\a})}(t).
$$

In order to compute $ch_{H^+\times \h_2}M^{\Dp}(\g_1)(xt)$ we simply compute the trace of the matrix of the action of $xt$ in the basis given by monomials
$$
\prod_{i,j} X_{-\d_i+\e_j}^{a_{ij}}X_{-\d_i-\e_j}^{b_{ij}}.
$$
If $\l$ is the $\h$-weight of this monomial, then 
\begin{align*}
Ad(xt)(\prod_{i,j} X_{-\d_i+\e_j}^{a_{ij}}X_{-\d_i-\e_j}^{b_{ij}})=
e^{\l}(t)(\prod_{\genfrac{}{}{0pt}{2}{1\le i\le n }{1\le j<m} }X_{-\d_i+\e_j}^{a_{ij}}X_{-\d_i-\e_j}^{b_{ij}}\prod_{1\le i\le n} X_{-\d_i+\e_m}^{b_{im}}X_{-\d_i-\e_m}^{a_{im}}).
\end{align*}
Thus the only contribution to the trace is given by the monomials
$$
\prod_{\genfrac{}{}{0pt}{2}{1\le i\le n}{ 1\le j<m} }X_{-\d_i+\e_j}^{a_{ij}}X_{-\d_i-\e_j}^{b_{ij}}\prod_{1\le i\le n} (X_{-\d_i+\e_m}X_{-\d_i-\e_m})^{a_{im}}.
$$
It follows that
$$
ch_{H^+\times \h_2}M^{\Dp}(\g_1^-)(xt)=\frac{e^{-\rho_1}}{\prod_{\genfrac{}{}{0pt}{3}{1\le i\le n}{1\le j<m}}(1-e^{-\d_i\pm\e_j})\prod_{1\le i\le n}(1-e^{-2\d_i})}(t).
$$
Putting everything together we find that
\begin{align}\label{charpari}
ch_{H^+\times \h_2}&M^{\Dp}(\g_1^-)(x^st)=\frac{1+(-1)^s}{2}\frac{e^{-\rho_1}}{\prod_{\a\in \Dp_1}(1-e^{-\a})}(t)\\
&+\frac{1-(-1)^s}{2}\frac{e^{-\rho_1}}{\prod_{\genfrac{}{}{0pt}{3}{1\le i\le n}{1\le j<m}}(1-e^{-\d_i\pm\e_j})\prod_{1\le i\le n}(1-e^{-2\d_i})}(t).\notag
\end{align}

We now apply our denominator formulas to the two summands above. We start with the first one: again the defect of $\g$ is $d=\min(n,m)$. We set 
\begin{align*}
\D(D_r)&=\pm\{\e_i\pm\e_j,\mid 1\le i\ne j\le r\},\\
\D(C_r)&=\pm\{\d_k\pm\d_l,2\d_k\mid n-r+1\le k\ne l\le n\},\\
\D(A_{r-1})&=\pm\{\d_k-\d_l\mid n-r+1\leq k\ne l\leq n\}.\end{align*}
Let $W(A_{r-1})$, $W(D_r)$, $W(C_r)$ be the corresponding Weyl groups. In this case $2\rho_1=2m\sum_{i=1}^n\d_i$. Set $\Dp(C_n)=\Dp_0\cap\D(C_n)$, $\Dp(B_m)=\Dp_0\cap\D(B_m)$, and $\Dp(A_{n-1})=\Dp_0\cap\D(A_{n-1})$. We denote by $\rho^C$, $\rho^B$, $\rho^A$ the corresponding $\rho$-vectors.

With notation as in  Section  \ref{ED}, the total order corresponding to this choice of $\Dp$ is
$$
\d_1>\dots >\d_n>\e_1>\dots>\e_m.
$$
There is only one arc diagram associated to this total order, namely $$X=\{\stackrel{\frown}{\d_{n}\e_1},\dots,\stackrel{\frown}{\d_{n-d+1}\e_d}\}.$$ 
The corresponding maximal isotropic set of roots is $S=\{\gamma_1,\dots,\gamma_d\}$, where
\begin{equation}
\gamma_1=\d_{n}-\e_1,\,\gamma_2=\d_{n-1}-\e_2,\ldots,\gamma_d=\d_{n-d+1}-\e_d.
\end{equation}

Applying formula \eqref{mermaid} we can write 
\begin{align}\label{primaD2}
e^{\rho}\check R&= {\check \cF}_{W^\sharp}\bigl( \frac{e^{\rho+\sum_{i=1}^d \rrbracket \gamma_i\llbracket }}{\prod_{i=1}^d(1-e^{-\gamma_i})}\bigr).
 \end{align}

Recall from \eqref{x} that ${\mathcal W_r}$ is the subgroup of $W_\g$ generated by $\{s_{2\d_i}\mid i=1,\dots, r\}$. Define
\begin{align*}
W_\g^{ext}&=\text{reflection group generated by $W_\g$ and the reflections $s_{\e_i}$},\\
W(B_m)&=\text{subgroup of $W_\g^{ext}$ generated by $W(D_m)$ and $s_{\e_i}$.}\end{align*}   
Extend $sgn'$ to $W_\g^{ext}$ by setting $sgn'(s_{\e_i})=1$. We define
$$
W=W(A_{n-1}) {\mathcal W_{n-d}}W(B_m), \  W_D=W(A_{n-1}) {\mathcal W_{n-d}}W(D_m)
$$

If $n< m$ then $W^\sharp=W(D_m)$ hence, summing over $W(A_{n-1})$,  we find 
\begin{equation}\label{parzDnm}
d!e^{\rho}\check R= {\check \cF}_{W_D}\bigl(\frac{e^{\rho+\sum_{i=1}^d \rrbracket \gamma_i\llbracket }}{\prod_{i=1}^d(1-e^{-\gamma_i})}\bigr).
\end{equation}
Since $s_{\e_m}$ fixes $\frac{e^{\rho+\sum_{i=1}^d \rrbracket \gamma_i\llbracket }}{\prod_{i=1}^d(1-e^{-\gamma_i})}$ we can rewrite \eqref{parzDnm} as
$$
d!e^{\rho}\check R= {\check \cF}_{W}\bigl(\frac{e^{\rho+\sum_{i=1}^d \rrbracket \gamma_i\llbracket }}{\prod_{i=1}^d(1-e^{-\gamma_i})}\bigr)-{\check \cF}_{W_D}\bigl(\frac{e^{\rho+\sum_{i=1}^d \rrbracket \gamma_i\llbracket }}{\prod_{i=1}^d(1-e^{-\gamma_i})}\bigr).
$$

The same argument given to prove \eqref{seconda} gives
\begin{equation}\label{secondaDmn}
e^{\rho}\check R={\check \cF}_{W}\bigl(\frac{e^{\rho}}{\prod_{i=1}^d(1-e^{-\llbracket\gamma_i\rrbracket})}\bigr)-{\check \cF}_{W_D}\bigl(\frac{e^{\rho }}{\prod_{i=1}^d(1-e^{-\llbracket\gamma_i\rrbracket})}\bigr).
\end{equation}
If $n\ge m$ then $W^\sharp=W(C_n)$. By an explicit computation we see that 
$
\rho+\sum_{i=1}^d \rrbracket \gamma_i\llbracket=\sum_{i=1}^{n-m}(n-m-i+1)\d_i.
$\par
Set 
\begin{equation}\label{y}\mathcal Y_d=\text{subgroup of $W^{ext}_\g$ generated by $\{s_{2\d_{n-i+1}}s_{\e_i}\mid 1\leq i\leq d\}$}.
\end{equation}
We have that $s_{2\d_{n-i+1}}s_{\e_i}(\rho+\sum_{i=1}^d \rrbracket \gamma_i\llbracket)=\rho+\sum_{i=1}^d \rrbracket \gamma_i\llbracket$ hence 
\begin{align*}
&sgn'(s_{2\d_{n-i+1}}s_{\e_i})s_{2\d_{n-i+1}}s_{\e_i}\frac{e^{\rho+\sum_{i=1}^d \rrbracket \gamma_i\llbracket}}{\prod_{j=1}^d(1-e^{-\gamma_j})}=
&\frac{e^{\rho+\sum_{i=1}^d \rrbracket \gamma_i\llbracket}}{\prod_{j=1}^d(1-e^{-\gamma_j})}-\frac{e^{\rho+\sum_{i=1}^d \rrbracket \gamma_i\llbracket}}{\prod_{j\ne i}^d(1-e^{-\gamma_j})}.
\end{align*}
It follows that, if $w=\prod_{r=1}^k s_{2\d_{n-d+i_r}}s_{\e_{i_r}}\in \mathcal Y_d$, then 
\begin{align*}
sgn'(w)w\frac{e^{\rho+\sum_{i=1}^d \rrbracket \gamma_i\llbracket}}{\prod_{j=1}^d(1-e^{-\gamma_j})}=\sum_{J\subset \{i_1,\dots,i_k\}}(-1)^{|J|}\frac{e^{\rho+\sum_{i=1}^d \rrbracket \gamma_i\llbracket}}{\prod_{j\not\in J}(1-e^{-\gamma_j})}.
\end{align*}
Note that if $|J|\ge2$ then there is a reflection in $W(A_{m-1})$ that fixes the element
$\frac{e^{\rho+\sum_{i=1}^d \rrbracket \gamma_i\llbracket}}{\prod_{j\not\in J}(1-e^{-\gamma_j})}$, hence
\begin{align*}
{\check \cF}_{W}\bigl(sgn'(w)w\frac{e^{\rho+\sum_{i=1}^d \rrbracket \gamma_i\llbracket}}{\prod_{j=1}^d(1-e^{-\gamma_j})}\bigr)={\check \cF}_{W}\bigl(\frac{e^{\rho+\sum_{i=1}^d \rrbracket \gamma_i\llbracket}}{\prod_{j=1}^d(1-e^{-\gamma_j})}-\sum_{h=0}^k\frac{e^{\rho+\sum_{i=1}^d \rrbracket \gamma_i\llbracket}}{\prod_{j\ne j_h}(1-e^{-\gamma_j})}\bigr).
\end{align*}
Note that 
$$s_{\d_{n-j_h+1}-\d_{n-d+1}}s_{\e_{j_h}-\e_d}(\frac{e^{\rho+\sum_{i=1}^d \rrbracket \gamma_i\llbracket}}{\prod_{j\ne j_h}(1-e^{-\gamma_j})})=\frac{e^{\rho+\sum_{i=1}^d \rrbracket \gamma_i\llbracket}}{\prod_{j\ne m}(1-e^{-\gamma_j})},$$
hence, if $w\in \mathcal Y_d$, we have
\begin{align*}
&{\check \cF}_{W}\bigl(sgn'(w)w\frac{e^{\rho+\sum_{i=1}^d \rrbracket \gamma_i\llbracket}}{\prod_{j=1}^d(1-e^{-\gamma_j})}\bigr)\\ &={\check \cF}_{W}\bigl(\frac{e^{\rho+\sum_{i=1}^d \rrbracket \gamma_i\llbracket}}{\prod_{j=1}^d(1-e^{-\gamma_j})}-k\frac{e^{\rho+\sum_{i=1}^d \rrbracket \gamma_i\llbracket}}{\prod_{j= 1}^{d-1}(1-e^{-\gamma_j})}\bigr).
\end{align*}
Using the fact that $W^{ext}_\g=W\mathcal Y_d$ and that $s_{\e_i}e^{\rho}\check R= e^{\rho}\check R$, it follows that
\begin{align*}
&2^dd!e^{\rho}\check R=
\\&{\check \cF}_{W^{ext}_\g}\bigl(\frac{e^{\rho+\sum_{i=1}^d \rrbracket \gamma_i\llbracket}}{\prod_{j=1}^d(1-e^{-\gamma_j})}\bigr)
 ={\check \cF}_{W\mathcal Y_d}\bigl( \frac{e^{\rho+\sum_{i=1}^d \rrbracket \gamma_i\llbracket}}{\prod_{j=1}^d(1-e^{-\gamma_j})}\bigr)=\\
 &2^d {\check \cF}_{W}\bigl(\frac{e^{\rho+\sum_{i=1}^d \rrbracket \gamma_i\llbracket}}{\prod_{j=1}^d(1-e^{-\gamma_j})}\bigr)-2^{d-1}d {\check \cF}_{W}\bigl(\frac{e^{\rho+\sum_{i=1}^d \rrbracket \gamma_i\llbracket}}{\prod_{j=1}^{d-1}(1-e^{-\gamma_j})}
 \bigr).
\end{align*}
Hence
\begin{align}\label{nmaggioremD}
d!e^{\rho}\check R ={\check \cF}_{W}\bigl(\frac{e^{\rho+\sum_{i=1}^d \rrbracket \gamma_i\llbracket}}{\prod_{j=1}^d(1-e^{-\gamma_j})}\bigr)-\half d {\check \cF}_{W}\bigl(\frac{e^{\rho+\sum_{i=1}^d \rrbracket \gamma_i\llbracket}}{\prod_{j=1}^{d-1}(1-e^{-\gamma_j})}
 \bigr).\end{align}
 Set $W'=W_{[\d_{n-d+1},\e_d]}$. Using the fact that ${\mathcal W_{n-d}}$ and $W'$ commute, we can choose $$W/W'=\left((W(A_{n-1})\times W(B_m))/W'\right){\mathcal W_{n-d}}.$$ Let $Y$ be the arc subdiagram corresponding to the interval $[\d_{n-d+1},\e_d]$. Applying \eqref{mermaid}, we have
\begin{align*}
{\check \cF}_{W}\bigl(\frac{e^{\rho+\sum_{i=1}^d \rrbracket \gamma_i\llbracket}}{\prod_{j=1}^d(1-e^{-\gamma_j})}\bigr)&=
{\check \cF}_{W/W'}\bigl(e^{\rho_X-\rho_Y}\cF_{W'}\bigl(
\frac{e^{\rho_Y+\sum_{i=1}^d\rrbracket \gamma_i\llbracket}}{\prod_{i=1}^d(1-e^{-\gamma_i})}\bigr)\bigr)\\
&={\check \cF}_{W/W'}\bigl(e^{\rho_X-\rho_Y}d!e^{\rho_Y}\check R_Y\bigr)
\end{align*}

Observe that $ht(\gamma_i)=2i-1$ so $\prod_{i=1}^d \frac{ht(\gamma_i)+1}{2} =d!$.
 Applying \eqref{formula1} to $Y$ we find
\begin{align*}
{\check \cF}_{W}\bigl(\frac{e^{\rho+\sum_{i=1}^d \rrbracket \gamma_i\llbracket}}{\prod_{j=1}^d(1-e^{-\gamma_j})}\bigr)=
{\check \cF}_{W/W'}\bigl(e^{\rho_X-\rho_Y}d!\cF_{W'}\bigl(
\frac{e^{\rho_Y}}{\prod_{i=1}^d(1-e^{-\llbracket\gamma_i\rrbracket})}\bigr)\bigr).
\end{align*}
hence
\begin{equation}\label{bigfactor}
{\check \cF}_{W}\bigl(\frac{e^{\rho+\sum_{i=1}^d \rrbracket \gamma_i\llbracket}}{\prod_{j=1}^d(1-e^{-\gamma_j})}\bigr)=d!{\check \cF}_{W}\bigl(\frac{e^{\rho}}{\prod_{j=1}^d(1-e^{-\llbracket\gamma_j\rrbracket})}\bigr).
\end{equation}
Analogously one checks that
\begin{equation}\label{smallfactor}
 {\check \cF}_{W}\bigl(\frac{e^{\rho+\sum_{i=1}^d \rrbracket \gamma_i\llbracket}}{\prod_{j=1}^{d-1}(1-e^{-\gamma_j})}\bigr)=(d-1)! {\check \cF}_{W}\bigl(\frac{e^{\rho+ \rrbracket \gamma_d\llbracket}}{\prod_{j=1}^{d-1}(1-e^{-\llbracket\gamma_j\rrbracket})}\bigr).
\end{equation}
Observe in \eqref{smallfactor} that $\rrbracket \gamma_d\llbracket=\llbracket \gamma_{d-1} \rrbracket$, hence
\begin{equation}\label{eliminagammad}
 {\check \cF}_{W}\bigl(\frac{e^{\rho+ \rrbracket \gamma_d\llbracket}}{\prod_{j=1}^{d-1}(1-e^{-\llbracket\gamma_j\rrbracket})}\bigr)={\check \cF}_{W}\bigl(\frac{e^{\rho}}{\prod_{j=1}^{d-1}(1-e^{-\llbracket\gamma_j\rrbracket})}\bigr)+{\check \cF}_{W}\bigl(\frac{e^{\rho+\llbracket \gamma_{d-1}\rrbracket}}{\prod_{j=1}^{d-2}(1-e^{-\llbracket\gamma_j\rrbracket})}\bigr)
 \end{equation}
 and the second summand in \eqref{eliminagammad} is zero since the reflection $s_{\e_{d-1}-\e_d}$ fixes it. Plugging these formulas in \eqref{nmaggioremD} we find
\begin{align}\label{secondanmaggioremD}
e^{\rho}\check R ={\check \cF}_{W}\bigl(\frac{e^{\rho}}{\prod_{j=1}^d(1-e^{-\llbracket\gamma_j\rrbracket})}\bigr)-\half  {\check \cF}_{W}\bigl(\frac{e^{\rho}}{\prod_{j=1}^{d-1}(1-e^{-\llbracket\gamma_j\rrbracket})}
 \bigr).\end{align}

Finally, observe that $W=W_D\cup W_Ds_{\e_m}$ and $s_{e_m}$ fixes 
$
\frac{e^{\rho}}{\prod_{j=1}^{d-1}(1-e^{-\llbracket\gamma_j\rrbracket})},
$
so
\begin{align}\label{secondaDunoaux}
e^{\rho}\check R ={\check \cF}_{W}\bigl(\frac{e^{\rho}}{\prod_{j=1}^d(1-e^{-\llbracket\gamma_j\rrbracket})}\bigr)-  {\check \cF}_{W_D}\bigl(\frac{e^{\rho}}{\prod_{j=1}^{d-1}(1-e^{-\llbracket\gamma_j\rrbracket})}
 \bigr).\end{align}
 To give a uniform treatment of the cases $n<m$ and $n\ge m$, we set $d_1=\min(n,m-1)$ so that  \eqref{secondaDunoaux}  and \eqref{secondaDmn} combine to give
\begin{align}\label{secondaDuno}
e^{\rho}\check R ={\check \cF}_{W}\bigl(\frac{e^{\rho}}{\prod_{j=1}^d(1-e^{-\llbracket\gamma_j\rrbracket})}\bigr)-  {\check \cF}_{W_D}\bigl(\frac{e^{\rho}}{\prod_{j=1}^{d_1}(1-e^{-\llbracket\gamma_j\rrbracket})}
 \bigr).\end{align}
 
Set $\PP_d^{+}=\{\ba \in \PP_d\mid a_m>0\}$  and $\PP_d^{0}=\PP_d\backslash \PP_d^{+}$. Clearly,  if $d<m$ then $\PP_d^+=\emptyset$ and $\PP_d^0=\PP_d$. 
Dividing \eqref{secondaDuno} by $\mathcal{D}_0=e^{\rho_0}\prod_{\a\in\Dp_0}(1-e^{-\a})$, we find that, arguing as for type $B(m,n)$,
\begin{align*}
\frac{e^{-\rho_1}}{\prod_{\a\in \Dp_1}(1-e^{-\a})}&=\frac{1}{\mathcal {D}_0}\left(\sum_{\mathbf a\in \PP_d^{+}} {\check \cF}_{W}\bigl( e^{\rho-\sum_{i=1}^d a_i\gamma_i}\bigr)+\sum_{\mathbf a\in  \PP_d^0} {\check \cF}_{W_D}\bigl( e^{\rho-\sum_{i=1}^d a_i\gamma_i}\bigr)\right).
\end{align*}
Since  $\mathcal{D}_0=e^{\rho^{C}}\prod_{\a\in\Dp(C_n)}(1-e^{-\a})e^{\rho^{D}}\prod_{\a\in\Dp(D_m)}(1-e^{-\a})$, we can rewrite the above formula as
\begin{align}\label{formulaD}
&\frac{e^{-\rho_1}}{\prod_{\a\in \Dp_1}(1-e^{-\a})}\\&=\sum_{\mathbf a\in \PP_d^+} \sum\limits_{w\in {\mathcal W_{n-d}}} sgn'(w)\frac{ch F_A(w(-\rho_1+\mu(\ba)+\rho^C)-\rho^C)}{\prod_{\a\in\Dp(C_n)\backslash\Dp(A_{n-1})}(1-e^{-\a})}chF^+_D(\varepsilon(\ba))\notag\\
&+\sum_{\mathbf a\in  \PP_d^0} \sum\limits_{w\in {\mathcal W_{n-d}}} sgn'(w)\frac{ch F_A(w(-\rho_1+\mu(\ba)+\rho^C)-\rho^C)}{\prod_{\a\in\Dp(C_n)\backslash\Dp(A_{n-1})}(1-e^{-\a})}chF_D(\varepsilon(\ba)).\notag
\end{align}
Here we recall that, if $\ba\in  \PP_d$, then
$\mu(\ba)= -\sum_{r=1}^d(a_{d+1-r})\d_{n-d+r}
$, $\varepsilon(\ba)=\sum_{r=1}^d a_r \varepsilon_r$, and $chF_A(\l)$ is given by \eqref{FAL}.
\vskip 10 pt

We now take care of the second summand in \eqref{charpari}. Note that
\begin{align*}
&\frac{e^{-\rho_1}e^{\rho^{C}}\prod_{\a\in\Dp(C_n)}(1-e^{-\a})e^{\rho^{D}}\prod_{\a\in\Dp(C_{m-1})}(1-e^{-\a})}{\prod_{\genfrac{}{}{0pt}{3}{1\le i\le n}{1\le j<m}}(1-e^{-\d_i\pm\e_j})\prod_{1\le i\le n}(1-e^{-2\d_i})}\\
&=\frac{e^{-\rho_1}e^{\rho^{C}}\prod_{\a\in\Dp(D_n)}(1-e^{-\a})e^{\rho^{D}}\prod_{\a\in\Dp(C_{m-1})}(1-e^{-\a})}{\prod_{\genfrac{}{}{0pt}{3}{1\le i\le n}{1\le j<m}}(1-e^{-\d_i\pm\e_j})},
\end{align*}
and that 
$$\rho^C+\rho^D-\rho_1=\sum_{i=1}^n (n-(m-1)-i)\d_i+\sum_{i=1}^{m-1}(m-i)\e_i,
$$
so the formula above is precisely the denominator for the distinguished Borel subalgebra of type $\Dp_{D2}$ for a superalgebra of type $D(n,m-1)$.

We can therefore apply the results of Section \ref{7} and find that 
\begin{align*}
&\frac{e^{-\rho_1}e^{\rho^{C}}\prod_{\a\in\Dp(C_n)}(1-e^{-\a})e^{\rho^{D}}\prod_{\a\in\Dp(C_{m-1})}(1-e^{-\a})}{\prod_{\genfrac{}{}{0pt}{3}{1\le i\le n}{1\le j<m}}(1-e^{-\d_i\pm\e_j})\prod_{1\le i\le n}(1-e^{-2\d_i})}\\
&= {\check \cF}_{W(m-1)}\bigl(\frac{e^{\rho}}{\prod_{i=1}^{d_1}(1-e^{-\llbracket\gamma_i\rrbracket})}\bigr),
\end{align*}
where $W(m-1)=W(A_{n-1}){\mathcal W_{n-d_1}} W(C_{m-1})$. Recall from \eqref{x+},  \eqref{x-}, that  ${\mathcal W_r}^+$  (resp. ${\mathcal W_r}^-$) is  the set of $w\in {\mathcal W_r}$ such that $w=s_{2\d_{i_1}}\dots s_{2\d_{i_k}}$ with $k$ even (resp. odd).

By dividing off $e^{\rho^{C}}\prod_{\a\in\Dp(C_n)}(1-e^{-\a})e^{\rho^{D}}\prod_{\a\in\Dp(C_{m-1}}(1-e^{-\a})$ and expanding the denominator, we find
\begin{align}\label{formulaxt}
&ch_{H^+\times\h_2}(M^{\Dp}(\g_1))(xt)=\\
&\sum_{\ba\in  \PP_d^0}(\sum_{w\in \mathcal W^+_{n-d_1}}sgn(w)\frac{ch F_A(w(-\rho_1+\mu(\ba)+\rho^C)-\rho^C)}{\prod_{\a\in\Dp(C_n)\backslash\Dp(A_{n-1})}(1-e^{-\a})}\times\notag\\
&\times \frac{\sum_{w\in W(C_{m-1})}sgn(w)e^{w(\varepsilon(\ba)+\rho^D)}}{\sum_{w\in W(C_{m-1})}sgn(w)e^{w(\rho^D)}}(t)).\notag
\end{align}

Recall that, if $\pi$ is a finite-dimensional representation of $G_1$, then $\tilde \pi$ is a representation of $\tilde G_1$ whose definition is given in  \eqref{definizionetildaD}. We can now state
 \begin{prop}\label{thetaD}With  notation as in  Theorem \ref{howe}, we have that, if $\tilde \pi\in \Sigma$, then there is $\ba\in  \PP_d$ such that 
 $$
 \pi = F^\pm(\varepsilon(\ba)).
 $$

Furthermore, if $\ba\in \PP_d^+$, then the $H^+\times\h_2$-character of the isotypic  component of $\tilde F^+(\varepsilon(\ba))$ in  $M^{\Dp}(\g_1)$ is
\begin{equation}\label{isotypicDplus}
  \sum\limits_{w\in {\mathcal W_{n-d}}}sgn(w)\frac{ch F_A(w(\rho^C-\rho_1+\mu(\ba))-\rho^C)}{\prod_{\a\in\Dp(C_n)\backslash\Dp(A_{n-1})}(1-e^{-\a})}ch_{H^+}F^+(\varepsilon(\ba)),
\end{equation}
while, for $\ba\in  \PP_d^0$,  the $H^+\times\h_2$-character of the isotypic  component  of $\tilde F^\pm(\varepsilon(\ba))$ is
\begin{equation}\label{isotypicDzero}
  \sum\limits_{w\in \mathcal W^\pm_{n-d}}sgn(w)\frac{ch F_A(w(\rho^C-\rho_1+\mu(\ba))-\rho^C)}{\prod_{\a\in\Dp(C_n)\backslash\Dp(A_{n-1})}(1-e^{-\a})}ch_{H^+}F^\pm(\varepsilon(\ba)),
\end{equation}
 \end{prop}
\begin{proof}
It follows from \eqref{formulaD} that $\tilde F^\pm(\sum_{i=1}^m a_i\e_i)$ can occur  in $M^{\Dp}(\g_1)$ only if $a_i=0$ for $i>d$. This proves the first  statement.

It is also clear from \eqref{formulaD} and Kostant's formula \eqref{a++} that, if $\ba\in \PP_d^+$, then the isotypic component of $\tilde F^+(\varepsilon(\ba))$ is given by \eqref{isotypicDplus}.

If $\ba\in  \PP_d^0$, let $M_\ba$ be the sum of the isotypic components of $\tilde F^+(\varepsilon(\ba))$ and  $\tilde F^-(\varepsilon(\ba))$. 
By substituting \eqref{formulaD} and \eqref{formulaxt} in \eqref{charpari}, we find that
\begin{align*}
&ch_{H^+\times \h_2}(M_\ba)(x^st)\\
&=\frac{1+(-1)^s}{2}\sum_{w\in {\mathcal W_{n-d}}}sgn(w)\frac{ch F_A(w(-\rho_1+\mu(\ba)+\rho^C)-\rho^C)}{\prod_{\a\in\Dp(C_n)\backslash\Dp(A_{n-1})}(1-e^{-\a})}
\times ch_{T}(F_D(\varepsilon(\ba))(t)
\\&+\frac{1-(-1)^s}{2}(\sum_{w\in \mathcal W^+_{n-d_1}}sgn(w)\frac{ch F_A(w(-\rho_1+\mu(\ba)+\rho^C)-\rho^C)}{\prod_{\a\in\Dp(C_n)\backslash\Dp(A_{n-1})}(1-e^{-\a})}\\
&\times \frac{\sum_{w\in W(C_{m-1})}sgn(w)e^{w(\varepsilon(\ba)+\rho^D)}}{\sum_{w\in W(C_{m-1})}sgn(w)e^{w(\rho^D)}}(t)).
\end{align*}
Observe now that, since $\ba\in  \PP_d^0$,  then $s_{-2\d_{n-d_1}}$ fixes $-\rho_1+\mu(\ba)+\rho^C$ so, if $w=w's_{-2\d_{n-d_1}}$ with $w'\in \mathcal W^-_{n-d}$ then $sgn(w)=-sgn(w')$ and 
$w(-\rho_1+\mu(\ba)+\rho^C)=w'(-\rho_1+\mu(\ba)+\rho^C)$. Combining this observation with Kostant's character formula \eqref{azero}, we can rewrite the formula above as
\begin{align*}
&ch_{H^+\times \h_2}(M_\ba)(x^st)\\
&=\sum_{w\in \mathcal W^+_{n-d}}sgn(w)\frac{ch F_A(w(-\rho_1+\mu(\ba)+\rho^C)-\rho^C)}{\prod_{\a\in\Dp(C_n)\backslash\Dp(A_{n-1})}(1-e^{-\a})} ch_{H^+}(F^+(\varepsilon(\ba))(x^st)\\
&+\sum_{w\in \mathcal W^-_{n-d}}sgn(w)\frac{ch F_A(w(-\rho_1+\mu(\ba)+\rho^C)-\rho^C)}{\prod_{\a\in\Dp(C_n)\backslash\Dp(A_{n-1})}(1-e^{-\a})} ch_{H^+}(F^-(\varepsilon(\ba))(x^st),
\end{align*}
from which \eqref{isotypicDzero} follows readily.
\end{proof}

In our case the set of roots of $\mathfrak{s}_2^\C$ is $\D(C_n)$. The element $H\in \h$ that corresponds to $-\sum_{i=1}^n\d_i$ has the property that $H_{|V_\C^\pm}=\pm I$. Thus the parabolic subalgebra $\p_2$ defined by $H$ is
$$
\p_2=\h\oplus \sum_{\a\in\D(A_{n-1})}(\g_0)_\a\oplus \n,
$$
where 
$
\n=\sum\limits_{\a\in\Dp(C_{n})\backslash \D(A_{n-1})}(\g_0)_\a
$ is the nilradical.

Recall that, if $\PP$ is as in \eqref{P} and $\ba\in\PP$, then we can define the weight $\nu(\ba)$ as in \eqref{nu}. Using Notation \ref{nota}, we can now state
\begin{cor}[Theta correspondence] \label{corthetapari}With  notation as in  Theorem \ref{howe}, we have
$$
\Sigma=\{\tilde F^+(\varepsilon(\ba))\mid \ba\in \PP_d\}\cup\{\tilde F^-(\varepsilon(\ba))\mid \ba\in \PP\} .
$$
Moreover 
$$
\tau(\tilde F^+(\varepsilon(\ba)))=L^2(-\rho_1+\mu(\ba)),
$$ 
and, if $\ba\in \PP$, then
$$
\tau(\tilde F^-(\varepsilon(\ba)))=L^2(-\rho_1+\nu(\ba)).
$$
\end{cor}
\begin{proof}
Using Proposition \ref{thetaD}, the proof follows exactly as in the proof of Corollary \ref{cortheta}.
\end{proof}

As already observed in previous sections,  we have also computed the character of $\tau(\tilde F^\pm(\varepsilon(\ba))$.  We set, as before, $\mathcal W^{reg}_{n-d}$ to be the set of $w\in {\mathcal W_{n-d}}$ such that $w(-\rho_1+\mu(\ba)+\rho^C)$ is regular for $\D(A_{n-1})$ and, for $w\in \mathcal W^{reg}_{n-d}$, $c_w=c_{w(-\rho_1+\mu(\ba)+\rho^C)}$. Then, combining Corollary \ref{corthetapari} with Proposition \ref{thetaD}, we find that, if $\ba\in \PP_d^+$, then 
\begin{equation}\label{charactermuD++}
ch(L^2(-\rho_1+\mu(\ba))=\!\!\!\!\sum_{w\in \mathcal W^{reg}_{n-d}}\!\!\!c_wsgn(w)ch V^2(\{w(-\rho_1+\mu(\ba)+\rho^C)\}-\rho^C),
\end{equation}
while, if $\ba\in \PP^+_0$, then
\begin{equation}\label{charactermuDzero}
ch(L^2(-\rho_1+\mu(\ba))=\!\!\!\!\!\!\sum_{w\in \mathcal W^+_{n-d}\cap  \mathcal W^{reg}_{n-d}}\!\!\!c_wch V^2(\{w(-\rho_1+\mu(\ba)+\rho^C)\}-\rho^C),
\end{equation}
and, if $\ba\in \PP$, then
\begin{equation}\label{charactermuDzeromeno}
ch(L^2(-\rho_1+\nu(\ba))=-\!\!\!\!\!\!\!\!\!\sum_{w\in \mathcal W^-_{n-d}\cap  \mathcal W^{reg}_{n-d}}\!\!\!\!\!\!\!\!\!c_wch V^2(\{w(-\rho_1+\mu(\ba)+\rho^C)\}-\rho^C).
\end{equation}

The argument given in  Corollary \ref{enrighteven} works also in the present case, thus, we obtain

\begin{cor}\label{enrightevenpari}If $\ba\in \PP_d$ and $\l_0=-\rho_1+\mu(\ba)+\rho^C$, then
$$
ch(L^2(-\rho_1+\mu(\ba)))=\sum_{w\in W_{\l_0}^A}(-1)^{\ell_{\l_0}(w)}V^2(\{w(\l_0)\}-\rho^C).
$$
If $\ba\in \PP$ and  $\l_1=-\rho_1+\nu(\ba)+\rho^C$, then
$$
ch(L^2(-\rho_1+\nu(\ba)))=\sum_{w\in W_{\l_1}^A}(-1)^{\ell_{\l_1}(w)}V^2(\{w(\l_1)\}-\rho^C).
$$
\end{cor}
\section{Theta correspondence for the pair $(U(n),U(p,q))$}
In this section we use the combinatorial machinery  developed in the previous sections to derive the Theta correspondence for the  compact dual pair $(U(n),U(p,q))$. This dual pair, according to Proposition \ref{B}, corresponds to  the distinguished sets of positive roots $\D^{(p,q)}_{gl}$ in a superalgebra of type $gl(m,n),\,p+q=m$. 
\par
Let $\g=gl(m,n)$. Its Weyl group is 
$$W_\g=\text{group generated by $s_{\e_1-\e_2},\ldots,s_{\e_{m-1}-\e_m}, s_{\d_1-
\d_2},\ldots,s_{\d_{n-1}-\d_n}$},$$ which we identify with $S_m\times S_n$. 
Recall that the defect of $\g$ is $d=\min(n,m)$. Set 
\begin{align*}
\D(A_{m-1})&=\{\e_i-\e_j\mid 1\leq i\ne j\leq m\},\\
\D(A_{n-1})&=\{\d_k-\d_l\mid 1\leq k\ne l\leq n\},\\
\D(A_{p-1}\times A_{q-1})&=\{\e_k-\e_l\mid 1\le k\ne l\le p\text{ or } p+1\le k\ne l\le m\},\\
\D_c&=\D(A_{p-1}\times A_{q-1})\cup\D(A_{n-1}).\end{align*}

Let $W(A_{n-1})$ be the Weyl group of $\D(A_{n-1})$ and $W(A_{m-1})$ the Weyl group of $\D(A_{m-1})$. 
 Let $W_c$ be the  Weyl group of $\D_c$ (the subscript ``$c$'' stands for compact).
Set $\Dp(A_{n-1})=\Dp_0\cap\D(A_{n-1})$, $\Dp(A_{m-1})=\Dp_0\cap\D(A_{m-1})$, and $\Dp(A_{p-1}\times A_{q-1})=\Dp_0\cap\D(A_{p-1}\times A_{q-1})$. We denote by $\rho^{A_{n-1}}$, $\rho^{A_{m-1}}$, $\rho^{p,q}$ the corresponding $\rho$-vectors.\vskip5pt
Set
$$L=\{i\in\{0,\ldots,p\}| 0\leq  d-i\leq q\}.$$
Set $i_{max}=\min(n,p)$ and $j_{min}=d-i_{max}$.  
For $\s\in S_m$ set
 \begin{align*}
n_1(\sigma)&=|\{\sigma(i)\mid p- i_{max}+1\leq i\leq p+j_{min}\}\cap \{1,\ldots,p\}|,\\
U_i&=\{w\in S_m|\ n_1(\sigma)=i\}.\end{align*}
Then  $S_m=\coprod_{i\in L} U_i.$
\par
With notation as in Section \ref{df}, the order corresponding to  $\D^{(p,q)}_{gl}$ is
$$\e_1>\dots>\e_p>\d_1>\dots>\d_n>\e_{p+1}>\dots>\e_m.$$
Define, for $i\in L$, $\be^{i}_t=\e_{p-t+1}-\d_t$ for $1\le t\le i$ and $\gamma^{i}_t=\d_{n-t+1}-\e_{t+p}$ for $1\le t\le d-i$. 
For $i\in L$, let $X_i$ be  the arc diagram 
$$
X_{i}=\{\stackrel{\frown}{\e_{p-t+1}\d_t}\mid 1\le t\le i\}\cup\{\stackrel{\frown}{\d_{n-t+1}\e_{p+t}}\mid 1\le t\le j\}.
$$
Then
\begin{equation}\label{sij}
S(X_i)=
\{\be^{i}_t\mid 1\le t\le i\}\cup \{\gamma^{i}_{t}\mid 1\le t\le d-i\}.
\end{equation}
 \vskip5pt
 We start by writing formula \eqref{mermaid} corresponding to $S(X_{i_{max}})$, which, for simplicity, we denote from now on by $\underline S$. 
Let $W_1^{i}$ (resp. $W_2^{i}$)  be the 
subgroup of $S_m$ that fixes $1,\dots,p- i$ and $p+d-i+1,\dots,m$ (resp.
$p-i+1,\dots, p+d-i$). 
Set $W_1=W_1^{i_{\max}}$ and $W_2=W_2^{i_{\max}}$.
Then $U_i$ is stable under the right action of $W_1$, $W_2$, and under the left action of $S_p\times S_q$. Choose $V_i\subset U_i$ such that $U_i=V_iW_1$. Let $W_i'$ be the subgroup of $W_1$ stabilizing $V_i$ and set $V'_i=V_i/W_i'$.

 Assume that $n<m$. By the above, using notation of Section \ref{mm},
$$\check{R}e^{\rho}=\check{\cF}_{S_m}\bigl(\cQ(X_{i_{max}})\bigr)=\sum_{i\in L}
\check{\cF}_{U_i}\bigl(\cQ(X_{i_{max}})\bigr)=\sum_{i\in L}\check{\cF}_{V'_i}\bigl(\check{\cF}_{W_1}\bigl(
\cQ(X_{i_{max}})\bigr).$$

 Set $v=\epsilon_{p-i_{max}+1},
w=\e_{p-i_{max}+d}$; let $Y$ be the subdiagram   in
 $X_{i_{max}}$ which corresponds to the interval $[v,w]$; note that $S(Y)=S(X_{i_{max}})$
and each vertex of $Y$ is an end of some arc. Moreover, $W_{[v,w]}=S_n\times  W_1$ and we may choose $W^{\sharp}$ for $Y$ to be either $S_n$ or  $W_1$. One has
\begin{align*}
\check{\cF}_{W_1}\bigl(\cQ(X_{i_{max}})\bigr)=
e^{\rho-\rho_Y}\check{\cF}_{W_1}\bigl(\cQ(Y)\bigr)=
e^{\rho-\rho_Y}\check{\cF}_{S_n}\bigl(\cQ(Y)\bigr),
\end{align*}
hence
\begin{align*}
\check{R}e^{\rho}&=\sum_{i\in L}\check{\cF}_{V'_i}\bigl( e^{\rho-\rho_Y}\check{\cF}_{S_n}\bigl(\cQ(Y)\bigr)\bigr)=\frac{1}{|W_i'|}\bigl(\check{\cF}_{V'_i}\bigl( \check{\cF}_{W_i'S_n}\bigl( e^{\rho-\rho_Y}\check{\cF}_{S_n}\bigl(\cQ(Y)\bigr)\bigr)\\&=\frac{1}{|W_i'|}\bigl(\check{\cF}_{V_i}\bigl( \check{\cF}_{S_n}\bigl(\cQ(X_{i_{max}})\bigr)\bigr).
\end{align*}
Since the elements of $V_i\subset S_m$ commute with the elements of $S_n$ we conclude that
\begin{equation}\label{mermaidupq}
\check{R}e^{\rho}=\frac{1}{|W_i'|}\bigl(\check{\cF}_{S_n}\bigl( \check{\cF}_{V_i}\bigl(\cQ(X_{i_{max}})\bigr)\bigr).
\end{equation}
Note that, if $m\le n$, then $L=\{p\}$, $U_p=S_m=W_1$. Thus \eqref{mermaidupq} holds as well with $V_p$ any subset of $S_m$, being just \eqref{mermaid} for this case.
\vskip5pt
We now make a careful choice of $V_i$ by exploiting the action of $W_1\times W_2$ on $U_i$.  Set $\s_{i}= \prod_{t=1}^{i_{max}-i}s_{\e_{p-i_{max}+t}-\e_{p+j_{min}+t}}$. Note that $\s_i$ is an element of $U_i$.
Let $\mathcal O$ be the orbit of $(S_p\times S_q)\s_{i}$ under the right action of $W_2$ in 
$(S_p\times S_q)\backslash U_{i}$. First we observe that, if $x,y$ are distinct elements of 
$\mathcal O$, then $xW_1\cap yW_1=\emptyset$. Indeed, if for $\s\in W_2$  we have $(S_p\times S_q)\s_{i}\s\ne (S_p\times S_q)\s_{i}$, then 
$(S_p\times S_q)\s_{i}\eta\s\ne (S_p\times S_q)\s_{i}$ for any $\eta\in W_1$. This can be checked as 
follows: assume $\s_{i}\s \s_{i}\not\in S_p\times S_q$, hence there is $i_0\in\{1,\dots,p\}$ such that 
$\s_{i}\s \s_{i}(i_0)>p$. If $i_0\le p- i_{max}$ then, since $\s_{i}(i_0)=i_0$, we have 
$\s(i_0)>p+d-i$, so $\s_{i}\eta\s \s_{i}(i_0)>p$ for any $\eta\in W_1$. 
If $p-i<i_0\le p$ then $\s_{i}\s \s_{i}(i_0)=i_0$ so we can assume $i_0=p-i_{max}+t$ with $t\le i-i_{max}$. In this case $ \s_{i}(i_0)=p+j_{min}+t$ and we must have that 
$\s(p+j_{min}+t)>p+d-i$ for otherwise we would have $\s_{i}\s \s_{i}(i_0)\le p$. 
But in this case $\eta\s \s_{i}(i_0)> p+d-i$ for $\eta\in W_1$ and we are done.

Next we show that the action of $W_1\times W_2$ on  $(S_p\times S_q)\backslash U_{i}$ by right multiplication is transitive. 
For this it is enough to show that for any $w\in S_m\cap U_{i}$, there is $\s\in 
W_1$ and $\tau\in W_2$, $\eta\in (S_p\times S_q)$ such that $\s_{i}=\eta w\s\tau$. Since 
$|\{w(t)\mid p-i_{max}<t\le p+j_{min}\}\cap\{1,\dots,p\}|=i$, we can find $\s\in W_1$ 
such that $w\s(t)\le p$ for $p-i<t\le p$  and $w\s(t)>p$ for 
$p-i_{max}<t\le p-i_{max}+i$ or $p<t\le p+j_{min}$. We can find $\eta\in (S_p\times S_q)$ so that $\eta w \s(t)=t$ for 
$p-i<t\le p+j_{min}$. Moreover we can also assume 
that $\eta w \s(p-i_{max}+t)=p+j_{min}+t$ for $1\le t \le i_{max}-i$. Finally, we can 
find $\tau\in W_2$ such that $\eta w \s\tau(t)= \eta w \s(t)$ for 
$p-i_{max}<t\le p+j_{min}$ ,  $\eta w \s\tau(p+j_{min}+t)= p-i_{max}+t$ for 
$p+j_{min}<t\le p+d-i$ and $\eta w \s\tau(t)=t$ otherwise. This means that 
$\eta w \s\tau= \s_{i}$ as wished. 

We set $\mathcal Y_{m-d}^{i}$ to be a set of coset representatives for $Stab_{i}\backslash W_2$ where $Stab_{i}$ is the stabilizer of the coset $(S_p\times S_q)\s_i$ under the right action of $W_2$ on $(S_p\times S_q)\backslash U_{i}$.

From our discussion it follows that we can choose $V_i=(S_p\times S_q)\s_i\mathcal Y^i_{m-d}$. Since $\mathcal Y^i_{m-d}\subset W_2$, we see that $W'_i$ is the subgroup of $W_1$ stabilizing $(S_p\times S_q)\s_i$, thus $\s_iW'_i\s_i$ is the subgroup of $W_1^i$ stabilizing $S_p\times S_q$. It follows that $\s_iW'_i\s_i=W^i_1\cap(S_p\times S_q)\simeq S_i\times S_{d-i}$. In particular $|W'_i|=i!(d-i)!$.

We can rewrite \eqref{mermaidupq} as
 \begin{align}\label{mermaidconY}
 e^{\rho} \check R&=\sum_{i\in L} \frac{1}{i!(d-i)!}{\check \cF}_{W_c\s_i\mathcal Y^i_{m-d}}\bigl(\cQ(X_{i_{max}})\bigr).
\end{align}
Set $\mathcal W_{m-d}^i=\s_i \mathcal Y_{m-d}^i\s_i$. Since $\mathcal Y_{m-d}^i\subset W_2$, it is clear that $\mathcal W_{m-d}^i\subset W_2^i$.  We can therefore rewrite \eqref{mermaidconY} as 
 \begin{align}\label{ss}
 e^{\rho} \check R&=\sum_{i\in L} {\check \cF}_{W_c\mathcal{W}^i_{m-d}}\bigl( sgn(\s_i)\s_i\cQ(X_{i_{max}})\bigr),
\end{align}
\begin{lemma}Given $i\in L$, set $r_i=i_{max}-i$. Then
\begin{align}\label{03}
 sgn(\s_{i})\s_{i}\bigl(\cQ(X_{i_{max}})\bigr)= 
e^{r_i(\llbracket \be^{i}_i\rrbracket-\llbracket \gamma^{i}_{d-i}\rrbracket)}\cQ(X_{i})
\end{align}
\end{lemma}
\begin{proof}
By an explicit computation, we see that 
\begin{align*}
\rho+\sum_{\gamma\in S(X_i)}\rrbracket \gamma\llbracket&=\sum_{r=1}^{p-i}\frac{m-n-2r+1}{2}\e_r+\sum_{r=p-i+1}^{p}\frac{q-p+i-j-1}{2}\e_r\\&+\sum_{r=p+1}^{p+j}\frac{q-p+i-j+1}{2}\e_r+\sum_{r=p+j+1}^m\frac{m+n-2r+1}{2}\e_r\\
&+\sum_{r=1}^i\frac{q-p+j-i+1}{2}\d_r+\sum_{r=i+1}^n\frac{q-p+j-i-1}{2}\d_r.
\end{align*}

 Observe that $s_{\e_{p-i+1}-\e_{p+d-i+1}}$ fixes $\rho+\sum_{\gamma\in S(X_i)}\rrbracket\gamma\llbracket$, hence we have
\begin{align*}
& s_{\e_{p-i+1}-\e_{p+d-i+1}}
\frac{e^{\rho+\sum_{\gamma\in S(X_
i)}\rrbracket\gamma\llbracket}}
{\prod_{r=1}^i (1-e^{-\e_{p-r+1}+\d_r})\prod_{r=1}^{d-i} (1-e^{\e_{p+r}-\d_{n-r+1}})
}=\\
& 
\frac{e^{\rho+\sum_{\gamma\in S(X_i)}\rrbracket \gamma\llbracket}}
{(1-e^{-\e_{p+d-i+1}+\d_i})\prod_{r=1}^{i-1} (1-e^{-\e_{p-r+1}+\d_r})\prod_{r=1}^{d-i} (1-e^{\e_{p+r}-\d_{n-r+1}})
}=\\
&-
\frac{e^{\rho+\sum_{\gamma\in S(X_{i-1})}\rrbracket \gamma\llbracket+\rrbracket \be^{i}_i\llbracket-\rrbracket \gamma^{i-1}_{d-i+1}\llbracket+\e_{p+d-i+1}-\d_i}}
{\prod_{r=1}^{i-1} (1-e^{-\e_{p-r+1}+\d_{r}})\prod_{r=1}^{d-i+1} (1-e^{\e_{p+r}-\d_{n-r+1}})
}
\end{align*}
Since $\rrbracket \be^{i}_{i}\llbracket=\llbracket \be^{i-1,}_{i-1}\rrbracket$ and $\rrbracket \gamma^{i-1}_{d-i+1}\llbracket+\e_{p+d-i+1}-\d_i=\llbracket \gamma^{i-1}_{d-i+1}\rrbracket$, we obtain that
\begin{align}\label{base}
s_{\e_{p-i+1}-\e_{p+d-i+1}}
 \frac{e^{\rho+\sum_{\gamma\in S(X_i)}\rrbracket \gamma\llbracket}}{\prod\limits_{\gamma\in 
S(X_i)}(1-e^{-\gamma})}=- \frac{e^{\rho+\sum_{\gamma\in S(X_{i-1})}\rrbracket\gamma\llbracket}}{\prod\limits_{\gamma\in 
S(X_{i-1})}(1-e^{-\gamma})}\ e^{\llbracket \be^{i-1}_{i-1}\rrbracket-\llbracket \gamma^{i-1}_{d-i+1}\rrbracket}.
\end{align}
 
Since 
$$
s_{\e_{p-i+1}-\e_{p+d-i+1}}(\llbracket \be^{i}_{i}\rrbracket-\llbracket \gamma^{i}_{d-i}\rrbracket)=\llbracket \be^{i-1}_{i-1}\rrbracket-\llbracket \gamma^{i-1}_{d-i+1}\rrbracket
$$
the lemma is proven by an obvious induction.
\end{proof}

Combining \eqref{ss} and \eqref{03}, we get
 \begin{align}\label{mermaidX}
 e^{\rho} \check R
=
&\sum_{i\in L}\frac{1}{i!(d-i)!}{\check \cF}_{W_c\mathcal W_{m-d}^{i}}\bigl( 
e^{r_i(\llbracket \be^{i}_i\rrbracket-\llbracket \gamma^{i}_{d-i}\rrbracket)}\cQ(X_{i})\bigr).
\end{align}
One last computation is needed:
\begin{lemma}
 \begin{align}
&\frac{1}{i!(d-i)!}{\check \cF}_{W_c\mathcal W_{m-d}^{i}}\bigl( 
e^{r_i(\llbracket \be^{i}_i\rrbracket-\llbracket \gamma^{i}_{d-i}\rrbracket)}\cQ(X_{i})\bigr)
={\check \cF}_{W_c\mathcal W_{m-d}^{i}}\bigl( 
\frac{e^{\rho+r_i(\llbracket \be^{i}_i\rrbracket-\llbracket \gamma^{i}_{d-i}\rrbracket)}}
{\prod_{\gamma\in S(X_i)} (1-e^{-\llbracket\gamma\rrbracket})
}\label{sommaij}\bigr).
\end{align}
\end{lemma}
\begin{proof}
Set $Y_i^1$ the subdiagram of $X_i$ whose support is $[\e_{p-i+1},\d_i]$ and $Y^2_i$ the subdia\-gram having support $[\d_{n-d+i+1},\e_{p+d-i}]$. Let $G_1=W_{[\e_{p-i+1},\d_i]}$ and $G_2=W_{[\d_{n-d+i+1},\e_{p+d-i}]}$. An easy calculation shows that 
$$
\llbracket \be^{i}_i\rrbracket-\llbracket \gamma^{i}_{d-i}\rrbracket=\sum_{h=p-i+1}^{p+d-i}\e_h-\sum_{h=1}^i\d_h-\sum_{h=n-d+i+1}^n\d_h.
$$
In particular $\llbracket \be^{i}_i\rrbracket-\llbracket \gamma^{i}_{d-i}\rrbracket$ is $G_1\times G_2$-invariant. Since $G_1\times G_2\subset W_c$ and commutes with $\mathcal W_{m-d}^i$, we can write
\begin{align*}
&{\check \cF}_{W_c\mathcal W_{m-d}^{i}}\bigl( 
\cQ(X_i)e^{r_i(\llbracket \be^{i}_i\rrbracket-\llbracket \gamma^{i}_{d-i}\rrbracket)}\bigr)
\\&={\check \cF}_{W_c/(G_1\times G_2)\mathcal W_{m-d}^{i}}\bigl( 
e^{\rho-\rho_{Y_i^1}-\rho_{Y_i^2}}e^{r_i(\llbracket \be^{i}_i\rrbracket-\llbracket \gamma^{i}_{d-i}\rrbracket)}{\check \cF}_{G_1}\bigl(\cQ(Y_i^1)\bigr){\check \cF}_{G_2}\bigl(\cQ(Y_i^2)
\bigr)\bigr).
\end{align*}
Combining  \eqref{formula1} with \eqref{mermaid}, we have that, for any $\g$, ${\check \cF}_{W_\g}\bigl(\cQ(X)
\bigr)={\check \cF}_{W_\g}\bigl(\cP(X)
\bigr)$, thus
\begin{align*}
&{\check \cF}_{W_c\mathcal W_{m-d}^{i}}\bigl( 
\cQ(X_i)e^{r_i(\llbracket \be^{i}_i\rrbracket-\llbracket \gamma^{i}_{d-i}\rrbracket)}\bigr)
\\&={\check \cF}_{W_c/(G_1\times G_2)\mathcal W_{m-d}^{i}}\bigl( 
e^{\rho-\rho_{Y_i^1}-\rho_{Y_i^2}}e^{r_i(\llbracket \be^{i}_i\rrbracket-\llbracket \gamma^{i}_{d-i}\rrbracket)}{\check \cF}_{G_1}\bigl(\cP(Y_i^1)\bigr){\check \cF}_{G_2}\bigl(\cP(Y_i^2)
\bigr)\bigr)\\
&={\check \cF}_{W_c\mathcal W_{m-d}^{i}}\bigl( 
\cP(X_i)e^{r_i(\llbracket \be^{i}_i\rrbracket-\llbracket \gamma^{i}_{d-i}\rrbracket)}\bigr).
\end{align*}
The lemma follows from the observation that 
$\cP(X_i)=i!(d-i)!\frac{e^{\rho}}
{\prod_{\gamma\in S(X_i)} (1-e^{-\llbracket\gamma\rrbracket})
}$.
\end{proof} 

Using  \eqref{sommaij}, we see that \eqref{mermaidX} becomes
\begin{align}\label{due}
e^{\rho} \check R
=
\sum_{i\in L} {\check \cF}_{W_c\mathcal W_{m-d}^{i}}\bigl( \frac{e^{\rho}}
{\prod\limits_{\gamma\in
S(X_i)}(1-e^{-\lb\gamma\rb})}e^{r_i(\llbracket \be^{i}_i\rrbracket-\llbracket \gamma^{i}_{d-i}\rrbracket)}\bigr).
\end{align}

Using \eqref{due},  and expanding in geometric series in the domain $\h^+=\{h\in\h\mid \a(h)>0, \a\in\D_1^+\}$, formula
 \eqref{due}   becomes
\begin{equation}\label{parziale}
e^{\rho} \check R
=\sum_{i\in L}\cF_{W_c\mathcal W_{m-d}^{i}}\bigl( \sum_{{\bf a}\in \PP_i,\,{\bf b}\in \PP_{d-i}}
 e^{V({\bf a}-{r_i}^i,{\bf b}+{r_i}^{d-i})}\bigr),
\end{equation}
where 
\begin{align*}
V({\bf a},{\bf b})=&\rho-\sum_{t=1}^ia_{i-t+1}\e_{p-i+t}+\sum_{t=1}^{d-i}b_{t}\e_{p+t}+\sum_{t=1}^ia_{t}\d_t-\sum_{t=1}^{d-i}b_{d-i-t+1}\d_{n-d+i+t}
\end{align*}
and we denote by $r^k$ the partition $(r,r,\dots,r)\in\PP_k$. Note that \eqref{parziale} can be written as
\begin{equation}\label{parziale2}
e^{\rho} \check R
=\sum_{i\in L} \cF_{W_c\mathcal W_{m-d}^{i}}\bigl( \sum_{{\bf a}\in \PP_i,\,{r_i}^{d-i}\subset{\bf b}\in \PP_{d-i}}
 e^{V({\bf a}-r_i^i,{\bf b})}\bigr).
\end{equation}

From the formula above we can already derive an explicit form of the Theta correspondence, but, in order to have the same parametrization of \cite{KV}, some extra work is needed.
For  $\ba\in \PP_i$ and $k\le i$ define
\begin{align*}
&f_k(\ba)=(a_1,\dots, a_k),\\
&l_k(\ba)=(a_{i-k+1},\dots,a_i).
\end{align*} If $\ba\subset r^i$, we set $r^i\backslash \ba=(r-a_i,\dots,r-a_1)$. If $\ba\in \PP_i$ then, given $b\in \PP_{d-i}$ with $a_1^{d-i}\subset {\bf b}$, we let ${\bf b}\cup{\ba} =(b_1,\dots, b_{d-i},a_1,\dots
, a_i )$.\par 
Given $r\in\nat$ and ${\bf a}\in \PP_i$,  let  $k\in\{0,\dots,i\}$ be such that $r^{k}\subset\bf a$.
 If  $(r-a_i)^{d-i}\subset {\bf b}$,  then we set 
$$s_{k}(\ba,{\bf b})=(f_{k}(\ba)-r^{k},{\bf b}\cup(r^{i-k}\backslash l_{i-k}(\ba))).
$$

We let  $\mathcal{W}^{i}_{m-d}(\ba,\mathbf b)$ be the set of $v\in \mathcal{W}^{i}_{m-d}$ such that $v(V(\ba,\mathbf b))$ is regular for $\D(A_{p-1}\times A_{q-1})$. Given $r\in\nat$ and ${\bf a}\in \PP_i$,  let  $k_r(\ba)$ be maximal in the set of  $k\in\{0,\dots,i\}$ such that $r^{k}\subset\bf a$.

\begin{lemma}\label{partitions}
Assume that $m<n$ and fix $i\in L$.
Let  $({\bf a},{\bf b})\in \PP_i\times \PP_{d-i}$ be such that $(r_i-a_i)^{d-i}\subset {\bf b}$ and set $k=k_{r_i}(\ba)$.
\begin{enumerate}
\item If $k\not\in L$, then $v(V({\ba}-{r_i}^i,\bf b))$ is $\D_c$-singular for any $v\in \mathcal W_{m-d}^{i}$.
\item If $k\in L$ then there is a bijection $v\mapsto z$ between $\mathcal{W}^{i}_{m-d}(\ba-{r_i}^i,\mathbf b)$ and $\mathcal{W}^{k}_{n-d}(s_{k}(\ba,{{\bf b}}))$ such that
$$
\cF_{W_c}\bigl( sgn(v)e^{vV({\bf a}-{r_i}^i,{\bf b})}\bigr)=\cF_{W_c}\bigl( sgn(z)e^{zV(s_{k}(\ba,{{\bf b}}))}\bigr).
$$
\end{enumerate}
\end{lemma}
\begin{proof}
We will prove both statements by induction on $i-k$. The case $i=k$ is obvious.

If $k<i$, set $t_0=p-i+1-(r_i-a_i)$. Note that $p-i_{max}+1\le t_0\le p-i$. 

If $p+d-i=m$ then, since $k<i$,  $k\not\in L$. Moreover we have that we can choose $\mathcal W_{m-d}^{i}=\{1\}$. 
An easy calculation shows that
$$
\rho=\sum_{t=1}^p(\frac{m-n+1}{2}-t)\e_t + \sum_{t=p+1}^m(\frac{n+m+1}{2}-t)\e_{t}+\sum_{t=1}^n(\frac{q-p+n+1}{2}-t)\d_t,
$$
so $(V(\ba,{\bf b}),\e_{t_0}-\e_{p-i+1})=0$. It follows that $v(V({\ba}-{r_i}^i,\bf b))$ is $\D_c$-singular for any $v\in \mathcal W_{m-d}^{i}$.

If $p+d-i<m$, let $s\in S_m$ be defined by setting $s(t_0)=p+d-i+1$, $s(p+d-i+1)=p-i+1$, $s(p-i+1)=t_0$, and $s(t)=t$ for $t\ne t_0,p-i+1, p+d-i+1$. Then $s(V(\ba,\mathbf b))=V(\ba'-(r_i+1)^{i-1},\mathbf b')$ where $\ba'=l_{i-1}(\ba)+1^{i-1}$ and $\mathbf b'=\mathbf b \cup (r_i-a_i)^1$. Suppose that $v\in \mathcal W_{m-d}^{i}$ and assume that 
$v(V(\ba-{r_i}^i,\mathbf b))$ is $\D_c$-regular. Then $v(t_0)>p+d-i$, so
$vs^{-1}(p-d+i+1)=s_0>p+d-i$. Hence $w'=s_{\e_{p-d+i+1}-\e_{s_0}}vs^{-1}\in W_2^{i-1}$
and
$$
\cF_{W_c}\bigl( sgn(v)e^{vV({\bf a}-{r_i}^i,{\bf b})}\bigr)= \cF_{W_c}\bigl( sgn(w')e^{w'V(\ba'-(r_i+1)^{i-1},\mathbf b')}\bigr).
$$
This implies that there is a unique $w\in \mathcal W_{m-d}^{i-1}(\ba'-(r_i+1)^{i-1},\mathbf b')$ such that
$$
\cF_{W_c}\bigl( sgn(v)e^{vV({\bf a}-{r_i}^i,{\bf b})}\bigr)= \cF_{W_c}\bigl( sgn(w)e^{wV(\ba'-(r_i+1)^{i-1},\mathbf b')}\bigr).
$$
Clearly the argument above can be reversed to prove that the map $v\mapsto w$ is a bijection. 
 Since $r_i+1=r_{i-1}$, $(r_{i-1}-a'_{i-1})^{d-i+1}\subset(r_{i}-a_i)^{d-i+1}\subset \mathbf b'$, and $k_{r_{i-1}}(\ba')=k_{r_i}(\ba)$, we see that we can apply the induction hypothesis. Thus, since $w(V(\ba'-(r_i+1)^{i-1},\mathbf b'))$ is $\D_c$-regular, we have $k\in L$. Moreover there is a bijection $w\mapsto z$ from ${\mathcal W_{m-d}}^{i-1}(\ba'-(r_i+1)^{i-1},\mathbf b')$ to $\mathcal W_{m-d}^{k}(s_{k}(\ba',\mathbf b'))$ such that
 $$
\cF_{W_c}\bigl( sgn(w)e^{wV({\bf a'}-{r_{i-1}}^{i-1},{\bf b}')}\bigr)= \cF_{W_c}\bigl( sgn(z)e^{zV(s_{k}(\ba',\mathbf b'))}\bigr).
$$ 
Since it is clear that $s_{k}(\ba',\mathbf b')=s_{k}(\ba,\mathbf b)$ we are done by composing the two bijections.
\end{proof}

\begin{prop}\label{GL}
\begin{align*}e^\rho \check R= 
\cF_{W_c\mathcal W_{m-d}^{i_{max}}}\bigl( 
\sum_{{\bf a}\in \PP_{i_{max}},\,{\bf b}\in \PP_{j_{min}}}
 e^{V({\bf a},{\bf b})}\bigr)
 +\sum_{i\in L, i<i_{max}}\cF_{W_c\mathcal W_{m-d}^{i}}\bigl(
\sum_{{\bf a}\in \PP_i,\,{\bf b}\in \PP^*_{d-i}}
 e^{V({\bf a},{\bf b})}\bigr).
\end{align*}
\end{prop}
\begin{proof}
The expression   \eqref{parziale2}  for $e^{\rho}\check R$ can be written as
\begin{align*}
&\cF_{W_c\mathcal W_{m-d}^{i_{max}}}\bigl(
\sum_{{\bf a}\in \PP_{i_{max}},\,{\bf b}\in \PP_{j_{min}}}
 e^{V({\bf a},{\bf b})}\bigr)+\sum_{i\in L,i<i_{max}}\cF_{W_c\mathcal W_{m-d}^{i}}\bigl(
\sum_{{\bf a}\in \PP_i,\, {r_i}^{d-i}\subset {\bf b}\in\PP_{d-i}}
 e^{V({\bf a},{\bf b})}\bigr)\\
 &+\sum_{i\in L}\cF_{W_c\mathcal W_{m-d}^{i}}\bigl(
\sum_{{r_i}^i\not\subset{\bf a}\in \PP_i,\,{r_i}^{d-i}\subset{\bf b}\in \PP_{d-i}}
 e^{V({{\bf a}-{r_i}^i},{\bf b})}\bigr).
\end{align*}

Let $III$ denote the third summand.
By Lemma \ref{partitions} we have
$$
III=\sum_{i\in L}\sum_{k\in L,\, k<i}\cF_{W_c\mathcal W_{m-d}^{k}}\bigl(
\sum_{
\substack{{r_i}^k\subset{\bf a}\in \PP_i,{r_i}^{k+1}\not \subset \ba\\
{\bf b}\in \PP_j\, {r_i}^{d-i}\subset {\bf b}}}
 e^{V(s_{k}({{\bf a}},{\bf b}))}\bigr),
$$
which is
$$
\sum_{k\in L, k<i_{max}}\sum_{i\in L,i>k}\cF_{W_c\mathcal W_{m-d}}^{k} \bigl(
\sum_{
\substack{{r_i}^k\subset{\bf a}\in \PP_i,{r_i}^{k+1}\not \subset \ba\\
{\bf b}\in \PP_{d-i}\, {r_i}^{d-i}\subset {\bf b}}}
 e^{V(s_k({{\bf a}},{\bf b}))}\bigr).
$$
We observe that  the map 
$s_k$ is a bijection between the set $$\bigcup_{i\in L,i>k} \{(\ba ,{\bf b})\in\PP_i\times \PP_{d-i}\mid k_{r_i}(\ba)=k,\, {r_i}^{d-i}\subset {\bf b}\}
$$
and the set $$\{(\ba, {\bf b})\in \PP_k\times \PP^*_h\mid(i_{max}-k)^h\not \subset {\bf b}\},$$ hence we can rewrite $III$ as
$$
\sum_{k\in L, k<i_{max}}\cF_{W_c\mathcal W_{m-d}}^{k} 
\bigl(
\sum_{
\substack{{\bf a}\in \PP_k\\
{\bf b}\in \PP^*_{d-k},\, (i_{max}-k)^{d-k}\not\subset {\bf b}}}
 e^{V(({{\bf a}},{\bf b}))}\bigr).
$$
Upon substituting we get the desired statement.
\end{proof}
Arguing as in the previous sections, it follows from Proposition \ref{GL} that 
\begin{align*}
&chM^{\Dp}(\g_1)=\\
&\sum_{{\bf a}\in \PP_{i_{max}},\,{\bf b}\in \PP_{j_{min}}}\sum_{
v\in \mathcal W_{m-d}^{i_{max}}} 
\!\!\!\!sgn(v) \\
&\frac{ch(F_{p,q}(v(-\rho_1+\mu(\ba,{\bf b})+\rho^{A_{m-1}})-\rho^{A_{m-1}}))}
{\prod_{\a\in\Dp(A_{m-1})\setminus\Dp(A_{p-1}\times A_{q-1})}(1-e^{-\a})}
ch\, F_n(\varepsilon(\ba,{\bf b}))\\
&+\sum_{i\in L, i<i_{max}}
\sum_{{\bf a}\in \PP_i,\,{\bf b}\in \PP^*_j}\sum_{v\in \mathcal W_{m-d}^{i}} 
sgn(v) \\&\frac{ch(F_{p,q}(v(-\rho_1+\mu(\ba,{\bf b})+\rho^{A_{m-1}})-\rho^{A_{m-1}}))}
{\prod_{\a\in\Dp(A_{m-1})\setminus\Dp(A_{p-1}\times A_{q-1})}(1-e^{-\a})}
ch\, F_n(\varepsilon(\ba,{\bf b}).
\end{align*}
Notation is as follows: $\Dp=\D_{gl}^{(p,q)}$; if $\ba\in\PP_i,\,{\bf b}\in\PP_{d-i}$, we set 
\begin{align*}
&\mu(\ba,{\bf b})= -\sum_{t=1}^ia_{i-t+1}\e_{p-i+t}+\sum_{t=1}^{d-i}b_{t}\e_{p+t},\\
&\varepsilon(\ba,{\bf b})=\sum_{t=1}^ia_{t}\d_t-\sum_{t=1}^{d-i}b_{d-i-t+1}\d_{n-d+i+t}.\end{align*}
Here $F_n(\varepsilon(\ba,{\bf b}))$ is the finite-dimensional irreducible $u(n)$-module with highest weight $\varepsilon(\ba,{\bf b})$. Finally, 
if $\l\in span(\e_i)$, set 
\begin{equation*}
ch F_{p,q}(\l)= \sum_{w\in S_p\times S_q}sgn(w)\frac{e^{w(\l+\rho^{p,q})-\rho^{p,q}}}{\Pi_{\a\in\Dp(A_{p-1}\times A_{q-1})}(1-e^{-\a})}.
\end{equation*}
Note that 
$$V(\ba, {\bf b})=\rho+\ve(\ba, {\bf b}) +\mu(\ba, {\bf b}).$$
\par
By \cite{Adams}, the inverse image $\tilde G_1$ of $G_1=U(n)$ in $\tilde K$ is isomorphic to the cover defined by the character $\det^{\frac{m}{2}}$. In particular, $\tilde G_1$  is connected, hence the set  $\Sigma$, occurring in Theorem \ref{howe}, is a subset of 
the set of finite-dimensional irreducible representations of its Lie algebra $u(n)$.
 \vskip5pt
 Set $\h_1=\h\cap u(n),\,\h_2=\h\cap u(p,q)$. 
\begin{cor}[Theta correspondence] With  notation as in  Theorem \ref{howe}, we have
$$
\Sigma=\bigcup_{\substack{
i\in L\\\ba\in\PP_i, {\bf b}\in \PP_{d-i}}}F_n(-(\rho_1)_{|\h_1}+\varepsilon(\ba,{\bf b})).$$
Moreover 
$$
\tau(F_n(\varepsilon(\ba, {\bf b})))=L^2(-(\rho_1)_{|\h_2}+\mu(\ba,  {\bf b})).
$$
\end{cor}
\begin{proof} Similar to the proof of Corollary \ref{corthetaD2}.\end{proof}

Let us finally display the character formula.
We let  $\overline{\mathcal{W}}^{i}_{m-d}$ be the set of $v\in \mathcal{W}^{i}_{m-d}$ such that $v(-(\rho_1)_{|\h_2}+\mu(\ba,  {\bf b})+\rho^{A_{m-1}})$ is regular for $\D(A_{p-1}\times A_{q-1})$.
For $v\in\overline{\mathcal{W}}^{i}_{m-d}$,
let  $$c_v=c_{v(-(\rho_1)_{|\h_2}+\mu(\ba,{\bf b})+\rho^{A_{n-1}})}.$$ From Proposition \ref{GL} we get that, if $\ba\in \PP_{i_{max}}$, ${\bf b}\in \PP_{j_{min}}$ then
\begin{align*}
&ch(L^2(-(\rho_1)_{|\h_2}+\mu(\ba,  {\bf b})))=\\&\sum_{v\in\overline{\mathcal{W}}^{i_{max}}_{m-d}}\!\!\!c_v sgn(v) ch V^2(\{v(-\rho_1+\mu(\ba,{\mathbf b})+\rho^{A_{m-1}})\}-\rho^{A_{m-1}}),
\end{align*} 
while, if $i\in L$ with $i<i_{max}$
and $\ba\in \PP_i$, ${\bf b}\in\PP^*_j$, then 
\begin{align*}
&ch(L^2(-(\rho_1)_{|\h_2}+\mu(\ba,  {\bf b})))=\\&\sum_{v\in\overline{\mathcal{W}}^{i}_{m-d}}\!\!\!c_v sgn(v) ch V^2(\{v(-\rho_1+\mu(\ba,\mathbf b)+\rho^{A_{m-1}})\}-\rho^{A_{m-1}}).
\end{align*}
However we can check that if  $\ba\in \PP^*_k$, ${\bf b}\in \PP^*_{d-i}$ with $k\le i$, then, $v\in  \overline{\mathcal{W}}^{i}_{m-d}$ implies that $W_cv=W_cw$ with $w\in W_2^{(i,d-k)}$, where $W_2 ^{(i,d-k)}=\{\s\in S_m\mid \s(\e_a)=\e_a,p-i+1\leq a\leq p+d-k\}$.
It follows that the character formula can be written in a more uniform and symmetric way as 
follows: fix $\ba\in \PP^*_k$, ${\bf b}\in \PP^*_h$ with $k\le p$, $h\le q$, and $h+k\le d$. Choose a set $\mathcal{W}^{(d-h,d-k)}_{m-d}$ of coset representatives for $W_c\backslash W_cW_2 ^{(d-h,d-k)}$. Then
\begin{align*}
&ch(L^2(-(\rho_1)_{|\h_2}+\mu(\ba,  {\bf b})))=\\&\sum_{v\in\overline{\mathcal{W}}^{(d-h,d-k)}_{m-d}}\!\!\!c_v sgn(v) ch V^2(\{v(-(\rho_1)_{|\h_2}+\mu(\ba,{\bf b})+\rho^{A_{m-1}})\}-\rho^{A_{m-1}}).
\end{align*}

Once the character formula is written in this form,  we can apply to it the argument given in  Corollary \ref{enrighteven}, thus, with notation as in  Section \ref{Bnm}, we obtain

\begin{cor} If $k\le p$, $h\le q$, $k+h\le d$, $\ba\in \PP^*_k$, $\mathbf b\in \PP^*_h$,  and $\l_0=-(\rho_1)_{|\h_2}+\mu(\ba,\mathbf b)+\rho^{A_{m-1}}$, then
$$
ch(L^2(\l_0))=\sum_{w\in W_{\l_0}^A}(-1)^{\ell_{\l_0}(w)}V^2(\{w(\l_0)\}-\rho^{A_{m-1}}).
$$
\end{cor}
\vskip5pt
\section{On the Kac-Wakimoto conjecture}\label{44}
\vskip5pt As a  final  application of our denominator formulas we verify   Kac-Wakimoto conjecture in a remarkable special case. Let us recall briefly this conjecture. 
Let $V$ be a finite dimensional irreducible highest weight module with highest weight $\L$. Recall that the {\it atypicality} of $\L$, denoted by $atp(\L)$ is the maximal 
number of linearly independent mutually orthogonal isotropic roots which are orthogonal to $\L$. The atypicality $atp(V)$ is defined as the atypicality of $\L+\rho$. The $\rho$-shift makes this definition independent of the chosen set of positive roots. Also recall that the supercharacter of $V$ is defined as $sch V=\sum_{\l\in\h^*}\sdim(V_\l)e^\l$.  In \cite{KW}Ê the following conjecture is stated. 
\vskip5pt
\par\noindent{\bf Conjecture.} {\it There exists  $b\in\mathbb{Q}$ such that 
\begin{equation}\label{KWfor}
\check Re^{\rho}sch V=b\check\cF_{W}\bigl(\frac{e^{\rho+\Lambda}}
{\prod_{i=1}^{atp(V)} (1-e^{-\beta_i})}\bigr), 
\end{equation}
where $\Lambda$ is the highest weight of $V$ and the set of simple roots 
contains mutually orthogonal isotropic roots
$\beta_1,\ldots,\beta_{atp(V)}$ satisfying $(\rho+\Lambda,\beta_i)=0$. }
\vskip5pt
We will refer to the latter condition as the KW condition. We prove this conjecture when $V$ is the natural representation of Lie superalgebras of type 
$gl(m,n),$  $D(m,n), D(n,m)$ with $m\geq n$ and $B(m,n), B(n,m)$ with $m>n$. Recall  that, in these cases, we have
$atp(V)=\min(m-1,n)$. Also note that  $atp(V)=m$ for $B(m,m)$, and in this case KW condition does not hold
for any choice of the set of simple roots, hence  we exclude $B(m,m)$ from our consideration.\par 
We fix the standard triangular decomposition of $\fg_0$ and consider the sets of simple roots compatible
with this triangular decomposition. We embed our root system into the lattice spanned by
 $\{\e_i\}_{i=1}^m\cup\{\delta\}_{i=1}^n$ (with $m\geq n$);
 for example, for $D(n,m), m>n$, $2\e_i$ is an even root.
 We choose the inner product $(\e_i,\e_j)=\delta_{ij}=-(\delta_i,\delta_j)$.

\subsection{Formula for a special root system}
\label{dist}
Consider  a
set of simple roots $\Pi$  
which contains the set
$$\{\e_1-\e_2,\ldots,\e_{m-1}-\e_m,
\e_m-\delta_1,\delta_1-\delta_{2},\ldots,\delta_{n-1}-\delta_n\}.$$
Then the corresponding total order on $\{\e_i\}_{i=1}^m\cup\{\delta\}_{i=1}^n$ 
is given by $\e_1>\e_2>\ldots>\e_m>\delta_1\ldots>\delta_n$.
For $i=1,\ldots, n$ set
$$\gamma_i:=\e_{m+1-i}-\delta_i.$$
Consider the natural representation $V=V(\e_1)$.

Set
$$\Gamma:=\left\{\begin{array}{ll} \{\gamma_i\}_{i=1,\ldots,n-1} &\text{ if } m=n,\\
     \{\gamma_i\}_{i=1,\ldots,n} &\text{ if } m>n.
     \end{array}\right. $$

Note that $\Gamma$ is a maximal set of isotropic pairwise orthogonal roots which are
orthogonal to $\rho+\e_1$. We have excluded $B(m,m)$ since in this case
the maximal set is $\{\gamma_i\}_{i=1}^{m-1}\cup\{\e_1+\delta_m\}$
and $\e_1+\delta_m$ is ``bad'' in the sense that this root does  not
lie in a set of simple roots obtained from $\Pi$ by a series of odd reflections. Recall that 
$\llbracket\gamma_j\rrbracket=\sum_{i=1}^j \gamma_i.$

We are going to show that there is a constant $j_V$  such that 
\begin{equation}\label{ChV}
\check Re^{\rho}sch V=j_V\check{\cF}_W\bigl(\frac{e^{\rho+\e_1}}
{\prod_{\gamma\in\Gamma} (1-e^{-\llbracket\gamma\rrbracket})}\bigr)
.
\end{equation}

Write $\check R=R_0/R_1,\,R_i=\prod_{\a\in\Dp_i}(1-e^{-\a})$.
\subsubsection{}
As $\fg_0$-module, $V$ is the sum of two simple modules
with highest weights $\delta_1$ and $\e_1$ respectively.
By the Weyl character formula one has
\begin{equation}\label{ChV1}
R_0e^{\rho_0}sch V= \cF_W(e^{\rho_0+\e_1}-e^{\rho_0+\delta_1}).
\end{equation}

\subsubsection{}Let $\g'$ be the simple subalgebra of $\g$ whose set  of roots $\D'$ is the set of  roots in $\D$ that are orthogonal to $\e_1$. We have that $\fg'$ is of type $gl(m-1,n), B(m-1,n), D(m-1,n), B(n,m-1), D(n,m-1)$ respectively
for $\fg$ of type $gl(m,n)$, $B(m,n)$, $D(m,n)$, $B(n,m)$, $D(n,m)$. Observe that, if $\Sigma$ is a set of simple roots for $\g$ such that $(\e_1, \a)\ge 0$ for any $\a\in\Sigma$, then a set of simple roots for $\g'$ is given by $\Sigma'=\Sigma\backslash \{\a\in\Sigma\mid (\a,\e_1)>0\}$. In particular the set $\Pi':=\Pi\setminus\{\e_1-\e_2\}$
is a set of simple roots
for $\g'$. We introduce $\rho', \rho_0', R_0', R_1', \check R', W'$ for this root system
in the standard way.
Recall  the denominator identity for $\fg'$ in the form
$$\check R'e^{\rho'}=c \check \cF_{W'}\bigl(\frac{e^{\rho'}}{\prod_{\gamma\in\Gamma} (1-e^{-\llbracket\gamma\rrbracket})}\bigr),$$
with $c=\frac{atp(V)!}{C_{\g'}}$, which can be rewritten as
$$R'_0e^{\rho'_0}= c\cF_{W'}\bigl(\frac{e^{\rho'_0}R_1'}{\prod_{\gamma\in\Gamma} (1-e^{-\llbracket\gamma\rrbracket})}\bigr),$$
or
$$ \cF_{W'}(e^{\rho'_0})=c \cF_{W'}\bigl(\frac{e^{\rho'_0}  R_1'}{\prod_{\gamma\in\Gamma} (1-e^{-\llbracket\gamma\rrbracket})}\bigr).$$

\subsubsection{}
We now prove
\begin{equation}\label{XX}
\cF_W(e^{\rho_0+\e_1}-e^{\rho_0+\delta_1})=j_V\cF_W\bigl(\frac{e^{\rho_0+\epsilon_1}R_1}
{\prod_{\gamma\in\Gamma} (1-e^{-\llbracket\gamma\rrbracket})}\bigr),
\end{equation}
where $j_V=\frac{1}{c}$ except for the case $D(m,m)$ with $\delta_i\in\Delta$ and in the latter case $j_V=\frac{1}{2c}$.
Note that, by~(\ref{ChV1}),  formula \eqref{ChV}  is equivalent to  formula \eqref{XX}.

Set
$$A=\cF_W\bigl(\frac{e^{\rho_0+\epsilon_1}R_1}{\prod_{\gamma\in\Gamma} (1-e^{-\llbracket\gamma\rrbracket})}\bigr).$$

Since $\frac{R_1}{R_1'}$ and $e^{\rho_0-\rho_0'}$ are $W'$-invariant, one has

\begin{align*}
A&=\cF_{W/W'}\bigl(\frac{R_1}{R_1'}e^{\epsilon_1+\rho_0-\rho'_0}\cdot \cF_{W'}\bigl(\frac{R_1'e^{\rho'_0}}
{\prod_{\gamma\in\Gamma}(1-e^{-\llbracket\gamma\rrbracket})}\bigr)\bigr)\\&=\cF_{W/W'}\bigl(\frac{R_1}{R_1'}
e^{\e_1+\rho_0-\rho'_0}\cdot \frac{1}{c}\cF_{W'}(e^{\rho_0'})\bigr)\\
&=\frac{1}{c}
\cF_W \bigl(e^{\rho_0+\e_1}\prod_{\beta\in\Delta_1^+\setminus\Delta'} (1-e^{-\beta})\bigr)
\end{align*}
so
$$A=\frac{1}{c}\sum_{J\subset \Delta_1^+\setminus\Delta'} (-1)^{|J|}\cF_W(e^{\lambda_J}),\ \ \ \text{ where }
\lambda_J=
\rho_0+\e_1-\sum_{\beta\in J}\beta.$$

Now  formula~(\ref{XX}) is equivalent to the formula
\begin{equation}\label{X}
A=a\cF_W(e^{\lambda_{\emptyset}}-e^{\lambda_{\{\e_1-\delta_1\}}}),
\end{equation}
where $a=1$ except for the case $D(m,m)$ with $\delta_i\in\Delta$ and $a=2$ in the latter case.

Recall that $\cF_W(e^{\lambda})=0$ if $\lambda$ is not regular. Let us find $J$ such that $\lambda_J$ is regular.
Write $\Delta_0=\Delta_{\e}\coprod\Delta_{\delta}$ with $\Delta_{\e}$ in the span
of $\{\e_i\}_{i=1}^m$ and
$\rho_0=\rho_{\e}+\rho_{\delta}$, where $\rho_{\e}$ (resp., $\rho_{\delta}$) 
is the standard ``$\rho$''
for $\Delta_{\e}$ (resp., for $\Delta_{\delta}$).
For any $\beta\in \Delta_1^+\setminus\Delta'$ one has
$(\beta,\e_1)=1,\ (\beta,\e_i)=0$ for $i>1$. Thus $\lambda_J=\rho_{\e}+(1-|J|)\e_1+\nu_J$,
where $\nu_J$ lies in the span of $\{\delta_j\}_{j=1}^n$. The regularity
of $\lambda_J$ is equivalent to the regularity of $\rho_{\e}+(1-|J|)\e_1$
and of $\nu_J$.

For $gl_m$ the element $\rho_{\e}+(1-|J|)\e_1$ is regular
if and only if $|J|=0,1$ or $|J|>m$, which is impossible
since $|\Delta_1^+\setminus\Delta'|=n\leq m$. For $|J|=1$ one has $\nu_J=\rho_{\delta}+\delta_i$, which
is regular only for $i=1$. This gives~(\ref{X}).

For $B(m,n), m>n$ the root system $\Delta_{\e}$ (resp., $\Delta_{\delta}$) is of type $B_m$ (resp., $C_n$).
The element $\rho_{\e}+(1-|J|)\e_1$ is regular if and only if $|J|=0,1$ or $|J|\geq 2m$, which is
impossible since $|\Delta_1^+\setminus\Delta'|=2n<2m$.
For $|J|=1$ one has $\nu_J=\rho_{\delta}\pm\delta_i$, which
is regular only for $\nu_J=\rho_{\delta}+\delta_1$; this gives~(\ref{X}).

For $B(n,m), m>n$ the root system $\Delta_{\e}$ (resp., $\Delta_{\delta}$) is of type $C_m$ (resp., $B_n$).
The element $\rho_{\e}+(1-|J|)\e_1$ is regular if and only if $|J|=0,1$ or $|J|\geq 2m+1$, which
is impossible because $|\Delta_1^+\setminus\Delta'|=2n+1<2m+1$.
Consider the case  $|J|=1$. Then $\nu_J=\rho_{\delta}$ if $J=\{\e_1\}$ or $\nu_J=\rho_{\delta}\pm\delta_i$
for $J=\{\e_1\mp\delta_i\}$.
Therefore $\lambda_J$ is regular for $J=\{\e_1\},\{\e_1-\delta_1\},\{\e_1+\delta_n\}$.
One has $\lambda_{\{\e_1\}}=s_{\delta_n}(\lambda_{\{\e_1+\delta_n\}})$
so $\cF_W(e^{\lambda_{\{\e_1\}}}+e^{\lambda_{\{\e_1+\delta_n\}}})=0$. This establishes~(\ref{X}).

For $D(n,m)$ the root system $\Delta_{\e}$ (resp., $\Delta_{\delta}$) is of type $C_m$ (resp., $D_n$).
The element $\rho_{\e}+(1-|J|)\e_1$ is regular if and only if $|J|=0,1$ or $|J|\geq 2m+1$, which is
impossible since $|\Delta_1^+\setminus\Delta'|=2n<2m+1$.
For  $|J|=1$ one has $\nu_J=\rho_{\delta}\pm\delta_i$ and this element is regular only
if $\nu_J=\rho_{\delta}+\delta_1$; this gives~(\ref{X}).

For $D(m,n), m>n$ the root system $\Delta_{\e}$ (resp., $\Delta_{\delta}$) is of type $D_m$ (resp., $C_n$).
The element $\rho_{\e}+(1-|J|)\e_1$ is regular if and only if $|J|=0,1$ or $|J|\geq 2m-1$, which is
impossible since $|\Delta_1^+\setminus\Delta'|=2n<2m-1$. For $|J|=1$ one has
$\nu_J=\rho_{\delta}\pm\delta_i$, which
is regular only for $\nu_J=\rho_{\delta}+\delta_1$; this gives~(\ref{X}).

Consider the case  $D(m,m)$ with $\Delta_{\e}$ (resp., $\Delta_{\delta}$)  of type $D_m$ (resp., $C_n$).
The element $\rho_{\e}+(1-|J|)\e_1$ is regular if and only if $|J|=0,1$ or $|J|\geq 2m-1$.
As above, for $|J|=1$ the element $\nu_J$ is regular only if
$J=\{\e_1-\delta_1\},\ \nu_J=\rho_{\delta}+\delta_1$.
If $|J|=2m$, then $J=\Delta_1^+\setminus\Delta'=\{\e_1\pm\delta_i\}_{i=1}^m$
so $\lambda_J=\rho_0+(1-2m)\e_1=s_{\e_1}s_{\e_m}(\rho_0+\e_1)$.
If $|J|=2m-1$, then $J=(\Delta_1^+\setminus\Delta')\setminus\{\beta\}$, where $\beta=\e_1\pm\delta_i$
and so $\nu_J=\rho_{\delta}\mp\delta_i$ which is regular only if $\nu_J=\rho_{\delta}+\delta_1$ and
$\lambda_J=\rho_0+(2-2m)\e_1+\delta_1=s_{\e_1}s_{\e_m}(\rho_0+\delta_1)$.
Since $s_{\e_1}s_{\e_m}\in W$ we obtain
$$A=2\cF_W(e^{\rho_0+\e_1}-e^{\rho_0+\delta_1}).$$
This establishes~\eqref{X} for this case.
Having established \eqref{X} in all cases, we have proven \eqref{ChV}.

\subsubsection{}
Let us deduce from~(\ref{ChV}) Kac-Wakimoto formula \eqref{KWfor} for the natural representation. If $\tilde \Pi$ is a set of simple roots for $\g$ and $V(\L)$ is an irreducible finite dimensional $\g$-module of  highest weight $\L$, recall that we say that the pair $(\tilde\Pi,V(\L))$ satisfies the KW condition if there are $\{\beta_1,\ldots,\beta_{atp(V)}\}\subseteq\tilde{\Pi}$ with $(\be_i,\be_j)=0$ for all $i,j$ and $(\L+\rho_{\tilde \Pi},\be_i)=0$ for all $i$.

Assume that $(\tilde \Pi, V)$ satisfies the KW condition, where, as above, $V$ is the natural representation of $\g$. Since KW condition is obviously invariant under the action of the Weyl group we can consider the total order on $\{\e_i\}_{i=1}^m\cup\{\delta_i\}_{i=1}^n$ corresponding to $\tilde \Pi$. We can also assume that $\e_1>\e_2>\ldots>\e_m$ and $\delta_1>\delta_2>\ldots >\delta_n$
with $\e_m,\delta_n\geq 0$ for $\fg\not=gl(m,n)$
(this means  that $\tilde{\Pi}$ induces the standard triangular decomposition on $\fg_{0}$).

Let $\Lambda$ be the highest weight of the standard representation. Clearly, $\Lambda$ is the
maximal element in $\{\e_i\}_{i=1}^m\cup\{\delta_i\}_{i=1}^n$ so $\Lambda\in\{\e_1,\delta_1\}$.
For $gl(m,m), D(m,m)$  we may (and will) assume that
$\e_1>\delta_1$ (since we can switch $\{\e_i\}$ and $\{\delta_i\}$ in subsection~\ref{dist}).
Let us show that $\Lambda=\e_1$ for $m>n$. Indeed, for $m>n$ one has $atp(V)=n$ and
so $(\beta_i,\delta_1)\not=0$
for some $i$. If $\Lambda=\delta_1$, then
$(\Lambda+\rho,\beta_i)=(\delta_1,\beta_i)\not=0$ which contradicts to the definition of $\beta_i$.
Hence $\Lambda=\e_1>\delta_1$.

Let $\alpha\in \tilde{\Pi}$ be a root satisfying $(\alpha,\e_1)=1$.
Notice that $(\rho,\alpha)\geq 0$ so $(\rho+\e_1,\alpha)\not=0$ and thus
$\alpha\not\in\{\beta_i\}_{i=1}^{atp(V)}$. The set $\tilde{\Pi}'=\tilde{\Pi}\setminus\{\alpha\}$
is a set of simple roots for $\fg'$. Since $\text{def}\,\g'=atp(V)$,  the set $\{\stackrel{\frown}{e_i d_i}\mid \be_i=e_i-d_i\}$ is an arc diagram $\tilde{X}$ for $\tilde{\Pi}'$.
Similarly, let $X$ be the arc diagram for  $\Pi'$ having the elements of $\Gamma$ as arcs.
Denominator identity for $\fg'$ gives, letting $\tilde \rho =\rho_{\tilde \Pi}$ and $\tilde \rho' =\rho_{\tilde \Pi'}$,
\begin{equation}\label{e?}atp(V)!\cdot\check \cF_{W'}\bigl(\frac{e^{\rho'}}{\prod_{\gamma\in\Gamma}  (1-e^{-\llbracket\gamma\rrbracket})}\bigr)= 
\check\cF_{W'}\bigl(\frac{e^{\tilde{\rho}'}}{\prod_{i=1}^{atp(V)}(1-e^{-\beta_i})}\bigr).\end{equation}
Let us show that $\rho-\rho'=\tilde{\rho}-\tilde{\rho}'$.
Indeed, observe that
 $\rho-\rho',\tilde{\rho}-\tilde{\rho}'$ are orthogonal to $\Delta'$.
Since $\e_1$ is maximal the inequality 
$(\alpha,\e_1)>0$ for $\alpha\in\Delta$ forces $\alpha\in\Delta^+\cap\tilde{\Delta}^+$.
Therefore
\begin{align*}2(\rho-\rho',\e_1)&=\sum_{\alpha\in\Delta^+_0\setminus\Delta'}(\alpha,\e_1)-
\sum_{\alpha\in\Delta^+_1\setminus\Delta'}(\alpha,\e_1)=\sum_{\alpha\in\tilde{\Delta}^+_0\setminus\Delta'}
(\alpha,\e_1)-
\sum_{\alpha\in\tilde{\Delta}^+_1\setminus\Delta'}(\alpha,\e_1)\\&=
2(\tilde{\rho}-\tilde{\rho}',\e_1).\end{align*}
Set $\xi=(\rho-\rho')-(\tilde{\rho}-\tilde{\rho}')$. We conclude that $\xi$ is orthogonal to $\Delta'$ and $\e_1$. 
 For $\fg\not=gl(m,n)$ this means that $\rho-\rho'=\tilde{\rho}-\tilde{\rho}'$.
 
 For $gl(m,n)$ we obtain  that $\xi$ is proportional to
 $\sum_{i=2}^m\e_i-\sum_{j=1}^n\delta_j$. The roots $\pm(\e_1-\e_2)$ are the only roots in $\Delta\setminus\Delta'$ which are not orthogonal to $\e_2$. Since $\Delta_0^+=\tilde{\Delta}_0^+$,
 one has 
 $$2(\tilde{\rho}-\tilde{\rho}',\e_2)=(\e_1-\e_2,\e_2)=2(\rho-\rho',\e_2).$$
 Hence $(\xi,\e_2)=0$ so $\xi=0$ as required.

Substituting \eqref{e?} in \eqref{ChV}, we obtain,
\begin{align*}
\check Re^{\rho}sch V&=j_V\check \cF_W\bigl(\frac{e^{\rho+\e_1}}{\prod_{\gamma\in\Gamma} (1-e^{-\llbracket\gamma\rrbracket})}\bigr)=j_V\cF_{W/W'}\bigl(e^{\rho-\rho'+\e_1}
\cF_{W'}\bigl(\frac{e^{\rho'}}{\prod_{\gamma\in\Gamma} (1-e^{-\llbracket\gamma\rrbracket})}\bigr)\bigr)\\
&=\frac{j_V}{atp(V)!} \cF_{W/W'}\bigl(e^{\rho-\rho'+\e_1}\cF_{W'}\bigl(\frac{e^{\tilde{\rho}'}}{\prod_{i=1}^{atp(V)}(1-e^{-\beta_i})}\bigr)\bigr)= \\&=\frac{j_V}{atp(V)!} 
\cF_{W}\bigl(\frac{e^{\tilde{\rho}+\e_1}}{\prod_{i=1}^{atp(V)}(1-e^{-\beta_i})}\bigr).
\end{align*}
This proves ~\eqref{KWfor} with $b=\frac{j_V}{atp(V)!}$. Using \eqref{valoricg}, one can check that $b$ has an expression depending only on $atp(V)$.

\vskip10pt
\centerline{\bf Acknowledgments}
\vskip5pt
We would like to thank C. De Concini and M. Vergne for useful discussions and E. Vinberg for correspondence. 

\footnotesize{

\noindent{\bf M.G.}: Department of Mathematics,
Faculty of Mathematics and Computer Science.
The Weizmann Institute of Science.
Rehovot 76100 Israel;\\
{\tt maria.gorelik@weizmann.ac.il}

\noindent{\bf V.K.}: Department of Mathematics, Rm 2-178, MIT, 77 
Mass. Ave, Cambridge, MA 02139;\\
{\tt kac@math.mit.edu}

\noindent{\bf P.MF.}: Politecnico di Milano, Polo regionale di Como, 
Via Valleggio 11, 22100 Como,
Italy;\\ {\tt pierluigi.moseneder@polimi.it}

\noindent{\bf P.P.}: Dipartimento di Matematica, Sapienza Universit\`a di Roma, P.le A. Moro 2,
00185, Roma , Italy;\\ {\tt papi@mat.uniroma1.it} }


\vskip10pt
\begin{thebibliography}{10}
\bibitem{Adams} J.~Adams, \emph{The Theta correspondence over $\R$}, Harmonic analysis, group representations, automorphic forms and 
invariant theory, Lecture Notes Series, Institut for Mathematical
Sciences, National University of Singapore, Vol. 12, pp. 1--37.
\bibitem{En}
T.~Enright,\emph{
Analogues of Kostant's ${\mathfrak u}$-cohomology formulas for unitary highest weight modules.}
J. Reine Angew. Math. 392 (1988), 27--36. 
\bibitem{EW}T.~Enright, J.~Willenbring, \emph{
Hilbert series, Howe duality and branching for classical groups. }
Ann. of Math. (2) 159 (2004), no. 1, 337--375. 
\bibitem{Gor} M.~Gorelik, \emph{Weyl denominator identity for finite-dimensional Lie superalgebras},  arXiv:0905.1181
\bibitem{thetaseries} R.~Howe, \emph{$\theta$-series and invariant theory},  Proc. Symp. Pure Math. \textbf{33} (1979), part I, 275--285.
\bibitem{howeclassical} R.~Howe, \emph{ Remarks on classical invariant theory}, Trans. Amer. Math. Soc. 313 (1989), no. 2, 539--570. 
\bibitem{howee} R.~Howe, \emph{Trascending classical invariant theory},  Journal of the A.M.S. \textbf{2} (1989), 535--552.
\bibitem{YK} K.~Iohara, Y.~Koga,\emph{Central extension of Lie superalgebras}, Comment. Math. Helv. 76 (2001) 110--154.
\bibitem{Kacsuper}
V.~G. Kac, \emph{Lie superalgebras},
Advances in Math. \textbf{26} (1977), no. 1, 8--96. 
\bibitem{Kacsuperp} V.~G. Kac, \emph{Representations of classical Lie superalgebras.} Differential geometrical methods in mathematical physics, II (Proc. Conf., Univ. Bonn, Bonn, 1977),  pp. 597--626, Lecture Notes in Math., 676, Springer, Berlin, 1978.
\bibitem{Kac}
V.~G. Kac, \emph{Infinite-dimensional {L}ie algebras}, third ed., Cambridge
  University Press, Cambridge, 1990.
\bibitem{MMJ}
V.~G. Kac, P.~M\"oseneder Frajria and P.~Papi, \emph{ On the Kernel of the affine Dirac operator},
Moscow Mathematical Journal , Vol. 8, no. 4, (2008), 759--788. 
\bibitem{KMP}
V.~G. Kac, P.~M\"oseneder Frajria and P.~Papi, \emph{Denominator identities for Lie superalgebras (extended abstract)},
 {\it Proceedings of the FPSAC 2010}, 839--850.
\bibitem{KW} V.~G. Kac, W.~Wakimoto, \emph{Integrable highest weight modules over affine superalgebras and number theory.}  Lie theory and geometry,  415--456, Progr. Math., 123, Birkh\"auser Boston, 1994.
  \bibitem{KV}
  M.~Kashiwara, M.~Vergne, \emph{
On the Segal-Shale-Weil representations and harmonic polynomials.}
Invent. Math. \textbf{44} (1978), no. 1, 1--47.
\bibitem{Kostant}
  B.~Kostant, \emph{
Lie algebra cohomology and the generalized Borel-Weil theorem.}
Ann. Math. \textbf{74} (1961), no. 2, 329--387. 
\bibitem{DD} J.-S.~Li,  A.~Paul, E.-C.~ Tan, C.-B.~Zhu, \emph{The explicit duality correspondence of $({\rm Sp}(p,q),{\rm O}^*(2n))$.}  J. Funct. Anal.  200  (2003),  no. 1, 71--100.
\bibitem{parker} M.~Parker,\emph{Classification of real simple Lie superalgebras of classical type.}  Indag. Math. (N.S.)  3  (1992),  no. 4, 419--466.

\bibitem{PS} I.~Penkov, V.~Serganova, \emph{Representations of classical Lie superalgebras of type ${\rm I}$.}  J. Math. Phys. \textbf{21} (4), (1980), 689--697.
\bibitem{SERG} V.~Serganova, \emph{Automorphisms of simple Lie superalgebras}, Math. USSR Izvestiya Vol. 24, (1985),  no. 3, 539--551.
\bibitem{serganova} V.~Serganova, \emph{Classification of real simple superagebras and symmetric superspaces}, Translation from Funks. An. i Ego Priloz. Vol. 17, (1983), no. 3, 46--54.
\bibitem{vdl} J.~W.~van de Leur,\emph{ A classification of contragredient Lie superalgebras of finite growth}. Communications
in Algebra \textbf{17} (1989), 1815 - 1841.
\end{thebibliography}
\end{document}